\documentclass{jams-l} 

\usepackage{amsthm,amsfonts,amssymb,amsmath,mathabx}
\usepackage[all]{xy}
\usepackage{xr-hyper}
\usepackage[colorlinks=true]{hyperref}
\usepackage{mathrsfs}

\newcommand{\BA}{{\mathbb {A}}}

\newcommand{\BC}{{\mathbb {C}}}

\newcommand{\BR}{{\mathbb {R}}}

\newcommand{\BZ}{{\mathbb {Z}}}

\newcommand{\CF}{{\mathcal {F}}}

\newcommand{\CH}{{\mathcal {H}}}

\newcommand{\CK}{{\mathcal {K}}}

\newcommand{\CN}{{\mathcal {N}}}
\newcommand{\CO}{{\mathcal {O}}}

\newcommand{\CV}{{\mathcal {V}}}
\newcommand{\CW}{{\mathcal {W}}}

\newcommand{\sC}{{\mathscr{C}}}

\newcommand{\sL}{{\mathscr{L}}}

\newcommand{\sP}{{\mathscr{P}}}

\newcommand{\sX}{{\mathscr{X}}}

\newcommand{\sZ}{{\mathscr{Z}}}

\newcommand{\fkb}{{\mathfrak{b}}}
\newcommand{\fkc}{{\mathfrak{c}}}
\newcommand{\fkg}{{\mathfrak{g}}}
\newcommand{\fkh}{{\mathfrak{h}}}
\newcommand{\fkl}{{\mathfrak{l}}}
\newcommand{\fkn}{{\mathfrak{n}}}
\newcommand{\fkp}{{\mathfrak{p}}}
\newcommand{\fku}{{\mathfrak{u}}}
\newcommand{\fks}{{\mathfrak{s}}}

\newcommand{\nat}{{\natural}}

\newcommand{\spade}{{\spadesuit}}

\newcommand{\RA}{{\mathrm {A}}}

\newcommand{\RN}{{\mathrm {N}}}

\newcommand{\RR }{{\mathrm {R}}}

\newcommand{\disc}{{\mathrm{disc}}}

\newcommand{\End}{{\mathrm{End}}}

\newcommand{\Gal}{{\mathrm{Gal}}}
\newcommand{\GL}{{\mathrm{GL}}}

\newcommand{\Hom}{{\mathrm{Hom}}}

\newcommand{\Ker}{{\mathrm{Ker}}}

\renewcommand{\Re}{{\mathrm{Re}}}

\newcommand{\Res}{{\mathrm{Res}}}

\newcommand{\Spec}{{\mathrm{Spec}}}
\newcommand{\SO}{{\mathrm{SO}}}

\newcommand{\tr}{{\mathrm{tr}}}

\newcommand{\vol}{{\mathrm{vol}}}

\newcommand{\wc}{\widecheck}
\newcommand{\wt}{\widetilde}
\newcommand{\wh}{\widehat}

\newcommand{\pair}[1]{\langle {#1} \rangle}

\newcommand{\ov}{\overline}
\newcommand{\incl}{\hookrightarrow}
\newcommand{\lra}{\longrightarrow}

\newcommand{\bs}{\backslash}

\newtheorem{thm}{Theorem}[section]
\newtheorem{cor}[thm]{Corollary}
\newtheorem{lem}[thm]{Lemma}
\newtheorem{prop}[thm]{Proposition}
\newtheorem {conj}[thm]{Conjecture}
\newtheorem{defn}[thm]{Definition}

\theoremstyle{definition}

\theoremstyle{remark}
\newtheorem{remark}{Remark}

\numberwithin{equation}{section}

\newcommand{\indexx}[1]{\index{$#1$}}

\begin{document}

\title{Automorphic period and the central value of Rankin--Selberg L-function}

\author{Wei Zhang}

\address{ Department of Mathematics, Columbia University, MC 4423,
2990 Broadway, New York, NY 10027 }
\curraddr{ Department of Mathematics, Columbia University, MC 4423,
2990 Broadway, New York, NY 10027}

\email{wzhang@math.columbia.edu}

\thanks{The author was supported in part by NSF Grant \#1204365 and a Sloan research fellowship.}

\date{July 08, 2012 and, in revised form, April 13, 2013.}


\maketitle 

\begin{abstract} Using the relative trace formula of Jacquet and Rallis, under some local conditions
we prove a refinement of the
global Gan--Gross--Prasad conjecture proposed by Ichino--Ikeda and N. Harris for unitary groups. We need to assume some expected properties of L-packets and some part of the local Gan-Gross-Prasad conjecture.
\end{abstract}

\setcounter{tocdepth}{1}
\tableofcontents

\section{Introduction}
In this article, as a sequel of \cite{Zh2}, we prove a conjectural refinement of the global Gan--Gross--Prasad conjecture (\cite{GGP}) for unitary groups under some local conditions. This refinement is modeled on the pioneering work of Waldspurger (\cite{W1}) on toric periods and the central values of L-functions on $\GL_2$. In an influential paper \cite{II},  Ichino and Ikeda first formulated the refinement for orthogonal groups.  After the Ichino--Ikeda formulation, N. Harris considered the case of unitary groups in his Ph.D Thesis at UCSD (\cite{H}). 
\subsection{The conjecture of Ichino--Ikeda and N. Harris.}
We now recall the conjectural refinement. Let $E/F$ be a quadratic
extension of number fields with adeles denoted by $\BA=\BA_F$ and $\BA_E$ respectively. Let $V$ be a Hermitian space of dimension $n+1$ and $W$ a (non-degenerate) subspace of codimension one. Denote the unitary groups by $U(V)$ and $U(W)$ respectively. Let $G=U(W)\times U(V)$ be the
product and $H$ the diagonal embedding of
$U(W)$ into $G$. Let $\pi=\pi_n\otimes \pi_{n+1}$ be a cuspidal automorphic
representation of $G(\BA)$ and let $\pi_{i,E}$ be the base change of $\pi_i$ to
$\GL_i(\BA_E)$, $i=n,n+1$. Denote by $L(s,\pi_E)$ the Rankin-Selberg convolution L-function $L(s,\pi_{n,E}\times
\pi_{n+1,E})$ due to Jacquet--Piatetski-Shapiro--Shalika (\cite{JPSS}). It is known to be the same as the one defined by the Langlands--Shahidi method. The reader may consult the introduction of \cite{GJR1} for an overview of the study of the central value $L(1/2,\pi_{n,E}\times
\pi_{n+1,E})$. We also consider the adjoint L-function of $\pi$ (cf. \cite[\S7]{GGP},\cite[Remark 1.4]{H})
$$
L(s,\pi, Ad)=L(s,\pi_n,Ad)L(s,\pi_{n+1},Ad).
$$
We refer to Remark \ref{rem def L Ad} for the definition of the adjoint L-function at bad places (also cf. Remark \ref{rem IIH} after Conjecture \ref{conj IIH} below).

Denote the constant
$$ \Delta_{n+1}=\prod_{i=1}^{n+1}L(i,\eta^i)=L(1,\eta)L(2,1_F)L(3,\eta)\cdots L(n+1,\eta^{n+1}),$$
where $\eta$ is the quadratic character of $F^\times \bs\BA^\times$ associated to $E/F$ by class field theory. 
Note that here $\Delta_{n+1}=L(M^{\vee}(1))$ where $M^\vee$ is the motive
dual to the motive $M$ associated to the quasi-split reductive group $U(n+1)$ defined by
Gross (\cite{G}).
We will be interested in the following combination of L-functions
\begin{align}
\sL(s,\pi)=\Delta_{n+1}\frac{L(s,\pi_E)}{L(s+\frac{1}{2},\pi, Ad)}.
\end{align}
We also write $\sL(s,\pi_v)$ for the local factor at $v$.
 
Let $[H]$ denote the quotient $H(F)\backslash H(\BA)$
and similarly for $G$. We endow $H(\BA)$ ($G(\BA)$, resp.) with their Tamagawa measures \footnote{Since the unitary group $H$ has a nontrivial central torus, we need to introduce a convergence factor: $dh=L(1,\eta)^{-1}\prod_{v}L(1,\eta_v)|\omega|_v$ for a nonzero invariant differential $\omega$ of top degree on $H$. Similarly for $G$.} and
$[H]$ ($[G]$, resp.) with the quotient measure by the counting measure on $H(F)$ ($G(F)$, resp.),  cf. \S2. In \cite{GGP}, Gan, Gross and Prasad propose to study an automorphic period integral \indexx{\sP}
$$
\sP(\phi)=\sP_H(\phi):=\int_{[H]}\phi(h)\,dh,\quad \phi\in \pi.
$$
They conjecture that the non-vanishing of the linear functional $\sP$ on $\pi$ (possibly by varying the Hermitian spaces $(W,V)$ and switching to another member in the Vogan L-packet \footnote{For the term ``Vogan L-packet ", cf. \cite[\S9-11]{GGP}.} of $\pi$) is equivalent to the non-vanishing of the central value $L(\frac{1}{2},\pi_{E})$ of the Rankin-Selberg L-function. This conjectural equivalence is proved for $\pi$ satisfying some local conditions in our previous paper (\cite{Zh2}). One direction of the equivalence had also been proved by Ginzburg--Jiang--Rallis (cf. \cite{GJR1}, \cite{GJR2}).

For arithmetic application, it is necessary to have a more precise relation between the automorphic period integral $\sP$ and the L-value $\sL(1/2,\pi_E)$. To state the precise refinement of the Gan-Gross-Prasad conjecture, we need to introduce more notations. Let $\langle \cdot,\cdot \rangle_{Pet}$ be the Peterson inner product
\begin{align}
\pair{\phi,\varphi}_{Pet}=\int_{[G]}\phi(g)\ov{\varphi}(g)\,dg,\quad \phi,\varphi\in \pi.
\end{align}
Fix a decomposition as a product $$\langle \cdot,\cdot
\rangle_{Pet}=\prod_v \langle \cdot,\cdot \rangle_v$$ under the decomposition $\pi=\otimes \pi_v$. In this way we fix an invariant inner product on $\pi_v$. 
Ichino and Ikeda first consider the following integration of matrix coefficient: for $\phi_v,\varphi_v\in \pi_v$, we define when $\pi_v$ is tempered 
\begin{align}\label{eqn def alpha}
\alpha_v(\phi_v,\varphi_v)=\int_{H_v}{\langle \pi_v(h)\phi_v,\varphi_v
\rangle _v}\, dh,\quad H_v=H(F_v).
\end{align}
It has the following nice properties for {\em tempered} $\pi_v$:
\begin{enumerate}
\item It converges absolutely and it is positive definite: $\alpha_v(\phi_v,\phi_v)\geq 0$.
\item When $\pi_v$ is unramified\footnote{For a non-archimedean place $v$ we say that $\pi_v$ is unramified if the quadratic extension $E/F$ is unramified at $v$, the group $G(F_v)$ has a hyperspecial subgroup $K_v=G(\CO_v)$ and $\pi_v$ has a nonzero $K_v$-fixed vector.
}, and the vectors $\phi_v,\varphi_v$ are fixed by $K_v$ such that $\langle \phi_v,\varphi_v \rangle_v=1$, we have $$\alpha_v(\phi_v,\varphi_v)=\sL(\frac{1}{2},\pi_v) \cdot \vol(H(\CO_v)).$$
\item If $\Hom_{H(F_v)}(\pi_v,\BC)\neq 0$, then the form $\alpha_v$ does not vanish identically. 
\end{enumerate}
 The first two were proved by Ichino and Ikeda (N. Harris in the unitary group case). The third property was conjectured by them and proved by Sakellaridis and Venkatesh (\cite[\S6.4]{SV}) in a more general setting. Waldspurger also proved the third property in the p-adic orthogonal case. Because of the second property, we normalize the form $\alpha_v$ as follows:
\begin{align}\label{eqn def alpha nat}
\alpha_v^\nat(\phi_v,\varphi_v)=\frac{1}{\sL(\frac{1}{2},\pi_v)}\int_{H_v}{\langle \pi_v(h)\phi_v,\varphi_v
\rangle _v}\,dh.
\end{align}
Clearly $\alpha_v^\nat$ is invariant under $H_v\times H_v$ and we may call it the ``local canonical invariant form".

We are now ready to state the conjecture of Ichino--Ikeda and N. Harris (cf. \cite{II}, \cite[Conjecture 1.3]{H}) that refines the global Gan--Gross--Prasad conjecture for unitary groups. Assume that the measure on $H(\BA)$ defining $\sP$ and the measures on $H(F_v)$ defining $\alpha_v$ satisfy:
$$
dh=\prod_{v}dh_v.
$$
\index{Conj}
\begin{conj}\label{conj IIH}
Assume that $\pi$ is {\em tempered}, i.e., $\pi_v$ is tempered for all $v$.
For any decomposable vector $\phi=\otimes
\phi_v\in \pi=\otimes \pi_v$, we have
\begin{equation}\label{fml *} \frac{|\sP(\phi)|^2}{\langle
\phi,\phi \rangle_{Pet}}=\frac{1}{|S_\pi|}\sL(\frac{1}{2},\pi)
\prod_{v}\frac{\alpha_v^\nat(\phi_v,\phi_v)}{\pair{\phi_v,\phi_v}_v},
\end{equation}
where $S_\pi$ is a finite elementary $2$-group: the component group associated to the L-parameter of $\pi=\pi_n\otimes \pi_{n+1}$.
\end{conj}
\begin{remark}\label{rem IIH}
The right hand side of the conjectural formula is insensitive to the definition of local L-factors at the finitely many bad places, as long as we choose the same definition in $\sL(s,\pi_v)$ and in the local canonical invariant form $\alpha_v^\nat$. 
\end{remark}

The conjectural formula
of this kind goes back to the celebrated work of Waldspurger (\cite{W1}) for the central values of L-functions of $\GL(2)$  (more or less
equivalent to the case $U(1)\times U(2)$ in the unitary setting). An arithmetic geometric version generalizing the formula of Gross--Zagier and S. Zhang (\cite{GZ},\cite{YZZ}) is also formulated in \cite[\S27]{GGP}, \cite{Z10} and \cite[\S3.2]{Zh1}. More explicit formulae were obtained by Gross (\cite{G87}), S. Zhang (\cite{Z04}) and many others. The formula of Waldspurger and the formula of Gross--Zagier and S. Zhang (\cite{GZ}, \cite{Z01},\cite{Z01b},\cite{YZZ}) play an important role in the spectacular development in application to the Birch and Swinnerton-Dyer conjecture for elliptic curves in the past thirty years. More recently, Y. Tian (\cite{T}) applies both formulae together to a classical Diophantine question and proves the infinitudes of square-free congruent numbers with an arbitrary number of prime factors.

The refined global conjecture for $\SO(3)\times \SO(4)$, concerning  ``the triple product L-function", was established after 
the work by Garrett \cite{Gar}, Piatetski-Shapiro--Rallis, Harris--Kudla \cite{HK}, Gross--Kudla, Watson \cite{Wat}, and Ichino \cite{I}. Recently Gan and Ichino (\cite{GI})
established some new cases for $\SO(4)\times \SO(5)$ (for endoscopic L-packets on $\SO(5)$). All of the known cases utilize
the theta correspondence in an ingenious way. 

The Waldspurger formula was also reproved by Jacquet and Jacquet--Chen (\cite{CJ}) using relative trace
formulae. 
\subsection{Main results.}
We now state our main result. 
Throughout this paper, we will assume two hypothesis, denoted by {\bf RH(I)} and {\bf RH(II)}. 

The first one is about some expected properties of the (global and local) L-packets of unitary groups (for all Hermitian spaces $W,V$), analogous to the work of Arthur on orthogonal groups (cf. \cite{Mok}, \cite{Wh} for the progress towards the unitary group case).

\begin{itemize} 
\item[] {\bf RH(I)}: Let $E/F$ be a quadratic extension of number fields. For $i=1,2$, let $V_i$  be a Hermitian space of dimension $N$, $U(V_i)$ the unitary group, $\pi_i$ an irreducible cuspidal automorphic representation of $U(V_i)$.  We further assume that at one place $v_0$ split in $E/F$, the representation $\pi_{i,v_0}$ ($i=1,2$) is supercuspidal. Then we have:
\begin{itemize} 
\item[(i)]
The (weak) base change $\pi_{i,E}$ to $\Res_{E/F}\GL(N)$ exists and $\pi_{i,E}$ is cuspidal with unitary central character, the Asai L-function $L(s,\pi_E, As^{(-1)^{N-1}})$ (cf. Remark \ref{rem def L As}) has a simple pole at $s=1$.
\item[(ii)]The multiplicity of $\pi_i$ in $L^2([U(V_i)])$ is one. 
\item[(iii)]
 Assume that $\pi_1$ and $\pi_2$ are nearly equivalent (i.e., $\pi_{1,v}\simeq \pi_{2,v}$ with respect to fixed isomorphisms $V_{1,v}\simeq V_{2,v}$,  for all but finitely many places $v$ of $F$). Then for every place $v$ of $F$, $\pi_{1,v}$ and $\pi_{2,v}$ are in the same local Vogan L-packet and  this local Vogan L-packet is generic.\end{itemize}
\end{itemize}
The second one is a part of the local Gan-Gross-Prasad conjecture in the unitary group case .
\begin{itemize}
\item[] {\bf RH(II)}: Let $E/F$ be a quadratic extension of local fields, and $(W_0,V_0)$ a pair of Hermitian spaces of dimension $n$ and $n+1$. Then in a generic local Vogan L-packet $\Pi_\psi$ of $U(W_0)\times U(V_0)$, there is at most  one representation $\pi$ of a relevant pure inner form $G=U(W)\times U(V)$  that admits a nonzero invariant linear form, i.e.:
$$
\Hom_{H}(\pi,\BC)\neq 0.
$$
\end{itemize}
We refer to  \cite[Conj. 17.1]{GGP} and \cite[\S9]{GGP} for the detailed description. A proof (of an even stronger version) for tempered L-packets for $p$-adic fields is recently posted by Beuzart-Plessis (\cite[Theorem 1]{BP}); the local conjecture in the orthogonal case for $p$-adic fields has earlier been proved by Waldspurger.

We need the fundamental lemma for the Jacquet--Rallis relative trace formulae. In \cite{Y} and its appendix, this fundamental lemma is proved when the residue characteristic $p\geq c(n)$ for a constant $c(n)$ depending only on $n$ (cf. Theorem \ref{thm FL}). 
\begin{thm}
\label{thm main} Let
$\pi$ be a tempered (i.e., $\pi_v$ is tempered for every place $v$) cuspidal automorphic representation of $G(\BA)$. Assume that the running hypothesis {\bf RH(I)} and {\bf RH(II)} holds. Denote by $\Sigma$ the finite set of non-split places $v$ of $F$ where $\pi_v$ is not unramified.
Assume that
\begin{itemize}
\item[(i)] There exists a split place $v_0$ such that the local component $\pi_{v_0}$ is supercuspidal.
\item[(ii)]  If $v\in \Sigma$, then either $H_v$ is compact or  $\pi_v$ is supercuspidal.
\item[(iii)] The set $\Sigma$ contains all non-split $v$ whose residue characteristic is smaller than the constant $c(n)$.
\end{itemize}
Then we have the following two cases:
\begin{itemize}
\item[(1)] (the totally split case) when every archimedean place $v$ of $F$ is split in the extension $E/F$ (i.e., $G_{F_\infty}\simeq (\GL_{n}\times \GL_{n+1})_{F_\infty}$), we have 
\begin{align*}\frac{|\sP(\phi)|^2}{\langle
\phi,\phi \rangle_{Pet}}=2^{-2}\sL\left(\frac{1}{2},\pi\right)
\prod_{v}\frac{\alpha^{\nat}_v(\phi_v,\phi_v)}{\pair{\phi_v,\phi_v}_v}.
\end{align*}
\item[(2)](the totally definite case) if $G(F_\infty)$ is compact where $F_\infty=\prod_{v|\infty}F_v$, then there is a non-zero constant $c_{\pi_\infty}$ depending only on the archimedean component $\pi_\infty$ of $\pi$ such that:
\begin{align*}\frac{|\sP(\phi)|^2}{\langle
\phi,\phi \rangle_{Pet}}=c_{\pi_\infty}2^{-2}\sL\left(\frac{1}{2},\pi\right)
\prod_{v}\frac{\alpha^{\nat}_v(\phi_v,\phi_v)}{\pair{\phi_v,\phi_v}_v}.
\end{align*}
   
\end{itemize}
\end{thm}

\begin{remark}
Under our assumptions $(i)$, the base change of $\pi_E$ of $\pi$ to the general linear group is cuspidal and hence
$$
|S_\pi|=|S_{\pi_n}|\cdot |S_{\pi_{n+1}}|=4.
$$ 

\end{remark}

\begin{remark}
The condition $(i)$ is due to the fact that currently we do not have a complete spectral decomposition of the Jacquet--Rallis relative trace formulae. The condition $(ii)$ seems to be only a technical restriction for our approach and will be discussed in \S9. We have the the restriction for the archimedean place because:  1) we have not proved the existence of smooth transfer at archimedean places (cf. \S5); 2) it is probably a more technical problem to evaluate the constant $c_{\pi_\infty}$.
\end{remark}

\begin{remark}
For a non-archimedean place $v$, the unitary group $H_v$ is possibly compact only when $n\leq 2$. When $n=1$, $H_v$ is always compact for a non-split $v$.  In this case, our proof is essentially the same as the one in \cite{CJ}. 
\end{remark}

We also make a local conjecture (Conjecture \ref{conj local}) for each place $v$. Together with a suitable spectral decomposition of the relative trace formulae, this conjecture would imply Conjecture \ref{conj IIH} for those $\pi$ with cuspidal base change $\pi_E$.

\subsection{Some applications.}

 We have the following application to the positivity of some central L-values. The positivity is also predicted by the grand Riemann hypothesis.  Lapid has obtained a more general result for Rankin--Selberg central L-values by a different method (\cite{L2}, cf. also \cite{L1} for the positivity of the central value of L-function of symplectic type). 

\begin{thm}
Assume that $\pi$ satisfies the conditions of Theorem \ref{thm main} and $E/F$ is split at all archimedean places. 
Then we have
$$
L(\frac{1}{2},\pi_E)\geq 0.
$$
\end{thm}

\begin{proof}It suffices to show this when $L(\frac{1}{2},\pi_E)\neq 0$. Then by \cite{Zh2}, there exists $\pi'$ in the same Vogan L-packet of $\pi$ such that the period $\sP$ on $\pi'$ does not vanish. By replacing $\pi$ by $\pi'$, we may assume that  the space $\Hom_{U(W)(F_v)}(\pi_v,\BC)$ does not vanish for every $v$. Then the local terms $\alpha'_v$ do not vanish. Now the positivity follows from the fact that the $\alpha'_v$ are 
all positive definite, and the other L-values appearing in $\sL(1/2,\pi_E)$ except $L(\frac{1}{2},\pi_E)$ are all positive. 
\end{proof}
\begin{remark}
 As another application,
M. Harris showed that Conjecture \ref{conj IIH} would imply the algebraicity of the L-value $\sL(1/2,\pi)$ up to some simple constant when $G(F_\infty)$ is compact and $\alpha_\infty\neq 0$ (cf. \cite[\S 4.1]{HM}). 
\end{remark}
\subsection{Outline of proof.}
We now sketch the main ideas of the proof, following the strategy of Jacquet and Rallis (\cite{JR}). First of all, by the multiplicity one result (\cite{AGRS},\cite{SZ}), we know a priori that there is a constant denoted by $\sC_\pi$ depending on $\pi$ such that for all decomposable $\phi,\varphi\in \pi$:
\begin{align}\label{eqn 1}
\sP(\phi)\ov{\sP(\varphi)}=\sC_\pi \prod_{v}\alpha_v^\nat(\phi_v,\varphi_v).
\end{align}
Instead of working with an individual $\phi\in \pi$ as in the conjecture, we switch our point of view to a distribution attached to $\pi$.

\begin{defn}\label{def global char}
We define the {\em (global) spherical character} $J_\pi$ associated to a cuspidal automorphic representation $\pi$ as the distribution 
\begin{align}
J_\pi(f):=\sum_{\phi} \sP(\pi(f)\phi)\ov{\sP(\phi)},\quad f\in \sC^\infty_c(G(\BA)),
\end{align}
where the sum of $\phi$ is over an orthonormal basis of $\pi$ (with respect to the Petersson inner product).
\end{defn}
The name ``spherical character" is suggested by many early analogous distributions (cf. \cite{RR} etc.). 
We also have a local counterpart: 
 \begin{defn}\label{def local char}
 
 We define  the {\em (local) spherical character} $J^\nat_{\pi_v}$ associated to $\pi_v$ as the distribution:
\begin{align}\label{def local char G}
J^\nat_{\pi_v}(f_v):=\sum_{\phi_v} \alpha^\nat_v(\pi_v(f_v)\phi_v,\phi_v),\quad f_v\in \sC^\infty_c(G(F_v)),
\end{align}
where the sum of $\phi_v$ is over an orthonormal basis of $\pi_v$. Similarly we define an unnormalized one $J_{\pi_v}$:
\begin{align}\label{def local char G unnorm}
J_{\pi_v}(f_v):=\sum_{\phi_v} \alpha_v(\pi_v(f_v)\phi_v,\phi_v).\end{align}

\end{defn}

 By (\ref{eqn 1}), we clearly have for decomposable $f=\bigotimes_{v} f_v$:
\begin{align}
J_\pi(f)=\sC_\pi \prod_{v}J^\nat_{\pi_v}(f_v),
\end{align}
where in the product in the right hand side, for a given $\pi$ and $f$, the local term $J^\nat_{\pi_v}(f_v)=1$ for all but finitely many $v$.  Then we have the following consequence of Conjecture \ref{conj IIH}.
\begin{conj}\label{conj dis}Assume that $\pi$ is a {\em tempered} cuspidal automorphic representation.
For all $f=\bigotimes_{v} f_v\in \sC^\infty_c(G(\BA))$, we have 
$$
J_\pi(f)=\frac{1}{|S_\pi|}\sL(1/2,\pi) \prod_{v}J^\nat_{\pi_v}(f_v).
$$
\end{conj}

\begin{lem}\label{lem equiv conj}
The 
Conjecture \ref{conj dis} is equivalent to Conjecture \ref{conj IIH}.
\end{lem}
\begin{proof}

It suffices to show that Conjecture \ref{conj dis} implies Conjecture \ref{conj IIH}. To see this, we note that the following are equivalent: {\em $(1)$ $\Hom_{H(F_v)}(\pi_v,\BC)\neq 0$, $(2)$ $\alpha_v\neq 0$, $(3)$ the distribution $J^\nat_{\pi_v}$ does not vanish}. \footnote{It is clear that $(2)$ is equivalent to $(3)$. The equivalence of $(1)$ and $(2)$ follows from the third property of $\alpha_v$ listed earlier.} Hence, Conjecture \ref{conj IIH} holds if for some $v$ the linear form $\alpha_v$ vanishes. Now assume that for all $v$, the linear forms $\alpha_v$ do not vanish. Then the distributions $J_{\pi_v}$ do not vanish. Then by 
Conjecture \ref{conj dis}, the constant $\sC_\pi$ must be $\frac{1}{|S_\pi|}\sL(1/2,\pi)$ which implies Conjecture \ref{conj IIH}. 
 \end{proof}
 
Note that there is a parallel question for the general linear group. This question can essentially be reduced to the celebrated theory of ``Rankin--Selberg convolution" due to Jacquet--Piatetski-Shapiro--Shalika (\cite{JPSS}). The idea of Jacquet and Rallis is to transfer the question from the unitary group to the general linear group via (quadratic) base change. They (\cite{JR}) introduced two relative trace formulae (RTF), one on the unitary group and the other on the general linear group. This is the main tool of this paper and the previous one (\cite{Zh2}).

In the general linear group case, there is a decomposition of a global spherical character into a product of the local ones, analogous to Conjecture \ref{conj dis}. But this time one may prove it without too much difficulty. Hence,
to deduce Conjecture \ref{conj dis}, it suffices to compare the two local spherical characters. Moreover, since we only need to find the constant $\sC_\pi$, we may just choose some special test functions $f$, as long as the local spherical character on the unitary group does not vanish for our choice. Therefore the main innovation of this paper is a formula for the local spherical character evaluated at some special test functions. The formula can be viewed as a truncated {\em local expansion of the local spherical character}, analogous to the local expansion of a character due to Harish-Chandra. The result may be of independent interest in view of local harmonic analysis in the relative setting. 

For comparison, let us recall briefly a result of Harish--Chandra. Let $F$ be a $p$-adic field. We temporarily use the notation $G$ for the $F$-points of a connected reductive group, and $\fkg$ the Lie algebra of $G$. Let $\CN$ be the nilpotent cone of $\fkg$ and $\CN/G$ the set of $G$-conjugacy classes in $\CN$. The set $\CN/G$ is finite. Let $\mu_\CO$ be the nilpotent orbital integral associated to $\CO\in \CN/G$ for a suitable choice of measure.  The exponential map defines a homeomorphism $\exp:\omega\to \Omega$ where $\omega$ ($\Omega$, resp.) is some neighborhood of $0$ in $\fkg$ ($1$ in $G$, resp.).
Let $\pi$ be an irreducible admissible representation of $G$. Then Harish-Chandra showed that there are constants $c_\CO(\pi)$ indexed by $\CO\in \CN/G$ such that when $\Omega$ is sufficiently small, for all $f$ supported in $\Omega$:
\begin{align}
\label{eqn HC char}
\tr(\pi(f))=\sum_{\CO\in \CN/G}c_{\CO}(\pi)\mu_{\CO}(\wh{f_\nat}).
\end{align}
Here $f_\nat$ is the function on $\omega$ via the homeomorphism $\exp$, and $\wh{f_\nat}$ is its Fourier transform. The constants $c_{\CO}(\pi)$ contain important information about $\pi$. For example, there is a distinguished nilpotent conjugacy class, namely the class of $0\in \CN$. If $\pi$ is a discrete series representation, the constant $c_{\{0\}}(\pi)$ is equal to the formal degree of $\pi$ for a suitable choice of Haar measure on $G$.

Now we return to our relative setting. We consider the local spherical character on the general linear group. Let $G':=\Res_{E/F}(\GL_{n}\times \GL_{n+1})$ and let $\Pi$ be an irreducible unitary generic representation of $G'(F)$. Then the local spherical character $I_\Pi$ (cf. (\ref{def local char G'})) defines a distribution on $G'(F)$ with a certain invariance property. These distributions are related to distributions on the $F$-vector space:
$$
\fks_{n+1}=\{X\in M_{n+1}(E)|X+\ov{X}=0\}.
$$ 
Here $X\mapsto \ov{X}$ denotes the Galois involution (entry-wise).
The group $\GL_{n+1}(F)$ acts on $\fks_{n+1}$ by conjugation. We will be interested in the restriction of this action to the subgroup $\GL_{n}(F)$ (as a factor of the Levi of the parabolic of $(n,1)$-type).
We let $\omega$ be a small neighborhood of $0$ in the $F$-vector space $\fks_{n+1}$. Then we have a natural way to pull back a function $f'$ on a small neighborhood of $1$ in $G'$ to a function denoted by $f'_\nat$ on $\omega$ (cf. \S8 for the precise definition). It is tempting to guess that there exists an analogous expansion of $I_\Pi$ in terms of the (relative to $\GL_{n}(F)$-action) unipotent orbital integrals on $\fks_{n+1}$. \footnote{Relative to the $\GL_{n}(F)$-action, an $X\in \fks_{n+1}$ is ``nilpotent" if the closure of its $\GL_{n}(F)$-orbit contains zero.} However, so far there are some difficulties. For example, when $n\geq 2$ there are infinitely many $\GL_n(F)$-nilpotent orbits in $\fks_{n+1}$ and these nilpotent orbital integrals often need to be regularized. 
We then restrict ourselves to a subspace of {\em admissible functions} (cf. Definition \ref{def adm}) supported on a small $\omega$. The precise definition is very technical. We expect that admissible functions have vanishing nilpotent orbital integrals (however, generally not even defined so far), except for one of the two regular unipotent orbits denoted by $\xi_-$. An expansion such as (\ref{eqn HC char}) of $I_\Pi(f)$ would then tell us that there should be only one term left, corresponding to the regular unipotent orbit $\xi_-$. Though it seems challenging to prove something such as (\ref{eqn HC char}) in our setting, we nevertheless manage to establish a truncated version (see Theorem \ref{thm germ gl} for the detail):

\begin{thm} Let $\Pi$ be an irreducible unitary generic representation of $G'(F)$. Then for any small neighborhood  $\omega$ of $0$ in $\fks_{n+1}$, there exists an admissible function $f'\in \sC^\infty_c(G'(F))$  such that $f'_{\nat}$ is supported in $\omega$ and
$$
I_{\Pi}(f')=(\ast)\mu_{\xi_-}(\wh{f'_\nat})\neq 0,
$$
where $(\ast)$ is an explicit non-zero constant depending only on  the central character of $\Pi$.
\end{thm}

We have a similar result for a local spherical character $J_\pi$ on the unitary group when either $\pi$ is a supercuspidal representation or the group $U(W)$ is compact. See \S9 for more details (Theorem \ref{thm germ u}). Then our main Theorem \ref{thm main} follows from the local comparison of the two spherical characters (cf. \S4 Conjecture \ref{conj local}).  

The proof for the unitary group case seems to be harder and needs the full strength of our previous results in the companion paper \cite{Zh2}. Namely we have to make use of the following results (cf. \S9)
\begin{enumerate}
\item The existence of smooth transfer.
\item Compatibility of smooth transfer with Fourier transform.
\item Local (relative) trace formula on ``Lie algebra".
\end{enumerate}
Note that the proof in \cite{Zh2} of these ingredients are in the reverse order listed here.

\subsection{Structure of this paper.} After fixing some notations in \S2, we review several global periods involving the general linear group in \S3 and deduce the decomposition analogous to Conjecture \ref{conj dis}.  Then in \S4 we recall the Jacquet--Rallis RTF and reduce the question to a comparison of local spherical characters. Then we give the proof of Theorem \ref{thm main} assuming a local result (Theorem \ref{thm local}). In \S5 we deal with the totally definite case (i.e., $G(\BR)$ compact). In \S6 we prepare some (relative) harmonic analysis on Lie algebras. In \S7 and \S8, we prove the local character expansion for the general linear group. The two key ingredients are Lemma \ref{lem W_n fml} and Lemma \ref{lem fml mu f}. In \S9 we show the local character expansion for the unitary group under some conditions, and complete the proof via the comparison of both spherical characters. 

Finally, we warn the reader of the change of measures: Only in the introduction, we use the Tamagawa measures associated to a differential form $\omega$ on $H$ normalized by:
  $$dh=L(1,\eta)^{-1}\prod_{v}L(1,\eta_v)|\omega|_v.$$
To have a natural local decomposition, below we will immediately switch to
$$dh=\prod_vdh_v,\quad dh_v=\prod_{v}L(1,\eta_v)|\omega|_v.$$
Another change comes when we move to the local setting (cf. the paragraph before Lemma \ref{lem local new}):  there we consider the unnormalized local measure 
 $$dh_v=|\omega|_v.$$
Similar warning applies to other groups such as $G$ and the general linear group.

\part{Global theory }
\section{Measures and notations}
We always endow discrete groups with the counting measure.
\subsection{Measures and notations related to the general linear group.}
We first list the main notations and conventions throughout this paper. 
We denote $H_n=\GL_n$, its standard Borel $B_n$ with the diagonal torus $A_n$, the unipotent radical $N_n$ of $B_n$. We denote by $B_{n,-}$ the opposite Borel subgroup, and $N_{n,-}$ its unipotent radical, and an open subvariety $H_n'=N_nA_nN_{n-}$ of $H_n$ (essentially the open cell of Bruhat decomposition). Their Lie algebras are denoted by $\fkh_n,\fkn_n$ etc. We denote by $M_{n,m}(F)$ the $F$-vector space of all $n\times m$ matrices with coefficients in $F$; and if $n=m$ we write it as $M_n(F)$. Then we have a natural embedding $H_n\subset M_{n}$. We denote 
$$e_n =(0,0,...,0,1)\in M_{1,n}(F),$$ 
and let $e_n^\ast \in M_{n,1}(F)$ be the transpose of $e_n $. The letter $u$ ($v$, resp.) usually denotes an upper (lower, resp.) triangular unipotent matrix or a column (row, resp.) vector.

We usually understand $H_{n-1}$ as a subgroup of $H_n$ via the block-diagonal embedding
$$
H_{n-1}\ni h\mapsto \left( \begin{array}{cc}h & \\
 &  1
\end{array}\right)\in H_n.
$$ 
We thus have a sequence of embeddings $...\subset H_{n-2} \subset H_{n-1}\subset H_n$.
Similarly we have have a sequence of embeddings  for the diagonal torus $A_n$, the unipotent $N_n$, etc..

 For a quadratic extension $E/F$ (local or global), we assume that 
 $$E=F[\tau],
 $$ where $\tau=\sqrt{\delta}$, $\delta\in F^\times$. We write $E^\pm$ the $F$-vector space where the nontrivial Galois automorphism in $\Gal(E/F)$ acts by $\pm 1$, and $E^+=F$. 
 
Let now $F$ be a local field. We will fix an additive
character $\psi=\psi_F$ of $F$ and then define a character $\psi_E$ of $E$ by
$$\psi_E(z)=\psi(\frac{1}{2}\tr_{E/F} z)$$
for the trace map $\tr_{E/F}:E\to F$. In particular, we have the compatibility $\psi_E|F=\psi$. We also say that $\psi$ is unramified if $F$ is non-archimedean and the largest fractional ideal  of $F$ over which $\psi$ is trivial is $\CO_F$. Similarly for $\psi_E$.  On $M_{n}(E)$ there is a bi-$E$-linear pairing valued in $E$
given by 
\begin{align}
\label{eqn bilinear form}
\pair{X,Y}:=\tr(XY).
\end{align}
 We then have a Fourier transform for $\phi\in\sC_c^\infty(M_{n}(E))$ 
$$
\wh{\phi}(X):=\int_{M_{n}(E)}\phi(Y)\psi_E(\pair{X,Y})\, dY.
$$
Here we use the self-dual measure on $M_{n}(E)$, i.e., the unique Haar measure characterized by:
\begin{align}\label{eqn measure}
\wh{\wh{\phi}}(X)=\phi(-X).
\end{align}
Note that this is also the same measure obtained by identifying $M_{n}(E)$ with $E^{n^2}$ and using the self-dual measure on $E=M_{1}(E)$. We now view both $M_n(F)$ and $M_n(E^-)$ as $F$-vector subspaces of $M_n(E)$. Then the restriction of the pairing $\pair{\cdot,\cdot}$ to each of them is non-degenerate $F$-valued pairing.  In this way we may define the Fourier transform of $f\in \sC_c^\infty(M_n(E^\pm))$ and we normalize the Haar measure on $M_{n}(E^\pm)$ as the self-dual one characterized by the analogous equation to  (\ref{eqn measure}). Set $n=1$ and we have a measure for $F=E^+$ and $E^-$.
Note that if we use the isomorphism $F\simeq E^-$ by $x\mapsto \sqrt{\delta} x$, then the measure on $E^-$ is $|\delta|^{1/2}_F dx$ for the self-dual measure $dx$ on $F$. Here our absolute values on $F$ and $E$ are normalized such that 
$$
d(ax)=|a|_F\,dx,\quad a\in F,
$$
and similarly for $E$. 

On $F^\times$ we denote the {\em normalized} Tamagawa measure associated to 
the differential form $x^{-1}dx$:
$$
d^\times x=\zeta_F(1)\frac{dx}{|x|_F},
$$
and the {\em unnormalized} one:
$$
d^\ast x=\frac{dx}{|x|_F}.
$$
Similarly for $E^\times$. On $H_n(F)$ we will take the Haar measure
$$
dg=\zeta_F(1)\frac{\prod_{ij}dx_{ij}}{|\det(g)|_F^n},\quad g=(x_{ij}),
$$
and similarly for $H_n(E)$ (where we replace $\zeta_F(1)$ by $\zeta_{E}(1)$).
Sometimes we also shorten $|\det(g)|$ by $|g|$ if no
confusion arises. 

We will assign the measure on $N_n(F)$ the additive self-dual measure:
$$
du=\prod_{1\leq i<j\leq n }\,du_{ij},\quad u=(u_{ij})
\in N_n(F).
$$
We denote the modular character by
$$
\delta_n(a)=\det(Ad(a):\fkn_n)=\prod_{i=1}^n a_i^{n+1-2i},
$$
where $a=diag[a_1,a_2,...,a_n]\in A_n(F)$ acts on $\fkn$ by $Ad(a)X=aXa^{-1}$. Similarly, we have $\delta_{n,E}$ if we replace $F$ by
$E$.
For $x\in M_{n,m}(F)$, we define $$\|x\|=\max\{|x_{ij}|_F\}_{1\leq i\leq n,1\leq j\leq m}.$$

Now let $F$ be a number field and let $\psi=\prod_v \psi_v$ be a nontrivial character of $F\bs \BA$. We denote by $\BA^1$ the subgroup of $\BA^\times$ consisting of $x=(x_v)_v\in \BA^\times $ with $|x|=\prod_v|x_v|_v=1$.
We endow the group
$H_n(\BA)$ with the product measure 
$$
dg=\prod_v \,dg_v.
$$
We denote by $Z_n$ the center of $H_n$ and the measure is determined by the measure on $\BA_F^\times$:
$$d^\times x=\prod_v\, d^\times x_v .$$
Note that under our measure, if $\psi_{v}$ is unramified, the volume of the maximal compact subgroup of $H_n(F_v)$ is given by
$$
\vol(H_n(\CO_{F_v}))=\zeta_v(2)^{-1}\zeta_v(3)^{-1}\cdots\zeta_v(n)^{-1}.
$$

If $E$ is a quadratic extension of $F$, we take similar conventions for $H_n(\BA_E)$, $Z_n(\BA_E)$ et al. 

\subsection{Measures and notations related to unitary groups.} 
In this paper, $W\subset V$ will denote an embedding of Hermitian spaces of dimension $n$ and $n+1$ respectively, $U(W)$ and $U(V)$ the corresponding unitary group, $G=U(W)\times U(V)$ and its subgroup $H$ being the diagonal embedding of $U(W)$.

Our method involves the comparison of orbital integrals between the unitary and general linear group cases, and between their Lie algebras, respectively.  We thus need to choose compatible measures on them. Let $\theta$ be a non-singular Hermitian matrix of size $n+1$. Then we may and will view the group $U(\theta)(F)$ as the subgroup of $\GL_{n+1}(E)$ consisting of $g\in \GL_{n+1}(E)$ such that
$$
\ov{g}^t\cdot \theta g\theta^{-1}=1.
$$
We may and will view the Lie algebra $\fku(\theta)$ of $U(\theta)$ as the subspace of $M_{n+1}(E)$ consisting of $X\in M_{n+1}(E)$ such that
$$
\ov{X}^t+\theta X\theta^{-1}=0.
$$
We denote by $\fku(\theta)^\dagger$ a companion space where the last equality is replaced by
$$
\ov{X}^t=\theta X\theta^{-1}.
$$
For any number $\tau\in E$ such that $\ov{\tau}=-\tau\neq 0$, we have an isomorphism (as $F$-vector spaces) from  $\fku(\theta)$ to $\fku(\theta)^\dagger$ mapping $X$ to $\tau^{-1} X$.

We will need to consider the symmetric space $S_{n+1}(F)\simeq H_{n+1}(F)\bs H_{n+1}(E)$. We identify $S_{n+1}$ with the subspace of $\GL_{n+1}(E)$ consisting of $g\in \GL_{n+1}(E)$ such that
\begin{align}\label{eqn def sn0}
\ov{g}g=1.
\end{align}
We have its tangent space $\fks=\fks_{n+1}$ at $1\in S_{n+1}$, which we call the {\em Lie algebra of $S_{n+1}(F)$}. Viewed as a subspace of $M_{n+1}(E)$, the vector space $\fks$ consists of $X\in M_{n+1}(E)$ such that
\begin{align}\label{eqn def fks}
\ov{X}+X=0.
\end{align}
Its companion is the space $M_{n+1}(F)$ (or $\fkg\fkl_{n+1}(F)$) viewed as a subspace of $M_{n+1}(E)$, namely consisting of $X\in  M_{n+1}(E)$ such that
$$
\ov{X}=X.
$$
For any number $\tau\in E$ such that $\ov{\tau}=-\tau\neq 0$, we have an isomorphism (as $F$-vector spaces) from $\fks(F)$ to $M_{n+1}(F)$ mapping $X$ to $\tau^{-1} X$.

We consider both $\fks$ and $\fku$ as $F$-vector subspaces of $M_{n+1}(E)$.  The restrictions of the bilinear form $\pair{\cdot,\cdot}$ (cf. (\ref{eqn bilinear form})) on $M_{n+1}(E)$ to $\fks$ and $\fku=\fku(\theta)$ take values in $F$, and are non-degenerate. The additive characters $\psi$ and $\psi_E$ then determine self-dual measures on $M_{n+1}(\BA_E)$, $\fku(\BA)$, $\fks(\BA)$ and the local analogues.
Moreover, if we change the Hermitian matrix  $\theta$ defining $\fku$ to an equivalent one, the subspace $\fku$ changes to its conjugate by an element in $\GL_{n+1}(E)$. Hence the measures are compatible with the change of $\theta$. These measures can also be treated as Tamagawa measures associated to top degree invariant differential forms. Let $\omega_0$ be a differential form on $\fku$ so that $|\omega_0|_v$ defines the self-dual measure for every place $v$. We also use the form $\omega_0$ to normalize the differential form $\omega$ that defines the measure on $U(\theta)(F_v)$ as follows. We consider the Cayley map
\begin{align}\label{def cay ugp}
\fkc(X):=(1+X)(1-X)^{-1}.
\end{align}
It defines a birational map between $\fku$ and $U(\theta)$, and it is defined at $X=0$. We normalize the invariant differential form $\omega$ on $U(\theta)$ by requiring that the pull back $\fkc^\ast \omega$ evaluating at $0$ is the same as $\omega_0$ evaluated at $0$. It follows that, when $v$ is non-archimedean, under the Cayley map, the restriction of the self-dual measure to a small neighborhood of $0$ in $\fku$ is compatible with the restriction of the Tamagawa measure $|\omega|_v$ to a small neighborhood of $1$ in $U(\theta)(F_v)$. In this way we choose the measure on $U(\theta)(\BA)$, globally and locally, as follows:
$$
dh=\prod_v L(1,\eta_v)|\omega|_v.
$$Our global measure is therefore not the Tamagawa measure, which should be $L(1,\eta)^{-1}dh$.
In particular, under our choice of measure, the volume of $[U(1)]=U(1)(F)\bs U(1)(\BA)$ is given by
\begin{align}\label{eqn vol U(1)}
\vol([U(1)])=2L(1,\eta).
\end{align}
This is due to the fact that the Tamagawa number for $U(1)\simeq \SO(2)$ is equal to $2$.

\section{Explicit local factorization of some periods}
In this section, we decompose several global linear forms on the general linear group  into explicit products of local invariant linear forms. Nothing is original in this section but we need to determine all constants in order to prove the main result of this paper.

\subsection{Invariant inner product.}
Let $\Pi=\Pi_n$ be a cuspidal automorphic representation of
$H_{n}(\BA_E)$ with unitary central character $\omega_\Pi$. We recall some basic facts on the Whittaker model of $\Pi=\otimes_{w}\Pi_w$. We extend the additive character $\psi_E$ to a character of $N_n(E)$ by
$$
\psi_E(u)=\psi_E(\sum_{i=1}^{n-1} u_{i,i+1}),\quad u=(u_{i,j})\in N_n(E).
$$
Similar convention applies to the other unipotent matrices in $N_n(F)$ et al.
We denote by $\sC^\infty(N_n(\BA_E)\bs H_n(\BA_E),\psi_E)$ the space of smooth functions $f$ on $H_n(\BA_E)$ such that
$$
f(ug)=\psi(u)f(g),\quad u\in N_n(\BA_E),g\in H_n(\BA_E).
$$Similarly we have the local counterpart $\sC^\infty(N_n(E_w)\bs H_n(E_w),\psi_w)$ for each place $w$ of $E$.
The Fourier coefficient of $\phi\in\Pi$ is defined as
$$
W_\phi(g)=\int_{N(E)\bs N(\BA_E)}\phi(ug)\ov{\psi}_E(u)\, du.
$$
Then we have $W_\phi\in \sC^\infty(N_n(\BA_E)\bs H_n(\BA_E),\psi_E)$. The map $\phi\mapsto W_\phi$ realizes an equivariant embedding $\Pi\incl \sC^\infty(N_n(\BA_E)\bs H_n(\BA_E),\psi_E)$. The image, the Whittaker model of $\Pi$, is denoted by $\CW(\Pi,\psi_E)$. For $\phi\in \Pi=\otimes_w\Pi_w$,
we assume that $W_\phi$ is decomposable 
\begin{align}\label{eqn whittaker }
W_{\phi}(g)=\prod_wW_{\phi,w}(g_w), \quad W_{\phi,w}\in  \sC^\infty(N_n(E_w)\bs H_n(E_w),\psi_{E,w})
\end{align}
where $w$ runs over all places of $E$, and $W_w(1)=1$ for almost all places $w$.
 
We need to compare the unitary structure in the decomposition 
$$
\Pi\simeq \bigotimes_w \CW(\Pi_w,\psi_{E,w}).
$$
On $\Pi$ we have the Petersson inner product, for $\phi,\phi'\in \Pi$:
$$
\langle \phi,\phi'\rangle_{Pet}=\int_{Z_n(\BA_E)H_n(E)\bs H_n(\BA_E)}\phi(g)\ov{\phi'(g)}\, dg.
$$
On $\CW(\Pi_w,\psi_{E,w})$ we have an invariant inner product defined by
\begin{align}\label{eqn defn vartheta}
\vartheta_w(W_w,W'_w)=
\int_{N_{n-1}(E_w)\backslash H_{n-1}(E_w)}W_{w}\left( \begin{array}{cc}h & \\
 &  1
\end{array}\right) \overline{W}'_{w}\left( \begin{array}{cc}h & \\
 &  1
\end{array}\right)\, dh.
\end{align}
The integral $\vartheta_w$ converges absolutely if $\Pi_w$ is generic unitary. 
When $\Pi_w$ and $\psi_{E,w}$ are unramified, the vectors $W_w=W_w'$ are fixed by $K_{n,w}:=H_n(\CO_{E,w})$ and normalized by $W_w(1)=1$, we have 
\begin{align}\label{eqn vartheta unram} 
\vartheta_{w}=\vol(K_{n,w})L(1,\Pi_w\times \wt{\Pi}_w).
\end{align}
This can be deduced from \cite[Prop. 2.3]{JS} (also a consequence of the proof of  Prop.\ref{prop inner prod}, particularly (\ref{eqn Psi unram}) and  (\ref{eqn Psi vartheta})).
Therefore we deine a normalized invariant inner product
\begin{align}\label{eqn defn vartheta nat} 
\vartheta_w^\nat(W_w,W'_w)=\frac{\vartheta_v(W_w,W'_w)}{L(1,\Pi_w\times \wt{\Pi}_w)}.
\end{align}
Then the product $\prod_w \vartheta_w^\nat$ converges and defines an invariant inner product on $\CW(\Pi,\psi_E)$. It is a natural question to compare it with the Petersson inner product. We now recall a result of Jacquet--Shalika (implicitly in \cite[\S4]{JS}, cf. \cite[p.265]{FLO}).
\begin{prop} 
\label{prop inner prod}
We have the following decomposition of the Petersson inner product in terms of the local inner product $\vartheta^\nat_w$:
\begin{align}
\langle \phi,\phi'\rangle_{Pet}=\frac{n \cdot \RR es_{s=1}L(s,\Pi\times
\widetilde{\Pi})}{\vol(E^\times \backslash\BA_E^1)}\prod_v
\vartheta_w^\nat(W_{\phi,w},W_{\phi',w}),
\end{align}
where $W_{\phi}=\otimes_{w}W_{\phi,w}$ and $W_{\phi'}=\otimes_{w}W_{\phi',w}$.
\end{prop}

\begin{proof}Up to a constant this is proved by \cite[\S4]{JS}. We thus recall their proof in order to determine this constant and the same idea of proof will also be used below to decompose the Flicker-Rallis period. We consider an Eisenstein series associated to a  Schwartz--Bruhat function
$\Phi$ on $\BA_E^n$. We consider the action of $H_n(E)$ on the row vector space $E^n$ from right multiplication. Then the stabilizer of $e_n =(0,0,...,1)\in E^n$  is the mirabolic subgroup $P_n$ of $H_n$. Set
$$
f(g,s)=|g|^s\int_{\BA_E^\times}\Phi(e_n ag)|a|^{ns}\, d^\times a,\quad \RR
e(s)>>0.
$$
Consider the Epstein-Eisenstein series
\begin{align}\label{eqn Ep-Es}
E(g,\Phi,s):=\sum_{\gamma\in ZP(E)\backslash H_n(E)}f(\gamma g,s)
\end{align}
which is absolutely convergent when $\RR e(s)>1$. Equivalently, we
have $$ E(g,\Phi,s)=|g|^s \int_{E^\times
\backslash\BA_E^\times}\sum_{\xi\in E^n-\{0\}}\Phi(\xi
ag)|a|^{ns}\, d^\times a.
$$
(Note: this corresponds to the case $\eta=1$ in \cite[\S4]{JS}.)
It has meromorphic continuation to $\BC$ and has a simple pole at
$s=1$ with residue (\cite[Lemma 4.2]{JS})
$$
\frac{\vol(E^\times \backslash\BA_E^1)}{n}\widehat{\Phi}(0).
$$
Note that the only non-explicit constant denoted by $c$ in \cite[Lemma 4.2]{JS} is the volume of $E^\times \backslash\BA_E^1$. Now consider the zeta integral
$$
I(s,\Phi,\phi,\phi')=\int_{Z_n(\BA_E)H_n(E)\bs H_n(\BA_E)}E(g,\Phi,s)\phi(g)\ov{\phi'}(g)\, dg.
$$ 
One one hand,
it has a pole at $s=1$ with residue 
$$
\frac{\vol(E^\times \backslash\BA_E^1)}{n}\widehat{\Phi}(0)\pair{ \phi,\phi'}_{Pet}.
$$
On the other hand, when $\Re(s)$ is large, it also equal to the following integral
\begin{align}\label{eqn unr 1}
\Psi(s,\Phi,W_\phi,W_{\phi'})=\int_{N_n(\BA_E)\bs H_n(\BA_E)}\Phi(eg)W_\phi(g) \ov{W_{\phi'}}(g)|\det(g)|^s\, dg.
\end{align}
This is equal to the product:
$$
\prod_w \Psi(s,\Phi_w,W_{\phi,w},W_{\phi',w}) 
$$
where the local integral is defined as
$$
\Psi(s,\Phi_w,W_{\phi,w},W_{\phi',w}) =\int_{N_n(E_w)\bs H_n(E_w)}\Phi_w(eg)W_{\phi,w}(g) \ov{W_{\phi',w}}(g)|\det(g)|^s\, dg.
$$
By \cite[Prop. 2.3]{JS}, we have, for unramified data of $(\Phi_w,W_w,W_w')$ and $\psi_{E,w}$ at a place $w$, normalized such that $W(1)=W(1)=\Phi_w(0)=1$,
\begin{align}\label{eqn Psi unram}
\Psi(s,\Phi,W_w,W'_w)=\vol(K_{n,w})L(s,\Pi_w\times \wt{\Pi}_w).
\end{align}
(Note: for our measure on $N_n(E_w)$, we have $\vol(N_n(E_w)\cap K_w)=1$.)
 From this we may deduce that 
$$
\Psi(s,\Phi,W_\phi,W_{\phi'})=L(s,\Pi\times \wt{\Pi})\prod_{w}\frac{\Psi(s,\Phi_w,W_{\phi,w},W_{\phi',w})}{L(s,\Pi_w\times \wt{\Pi}_w)},
$$
where the local factors are entire functions of $s$ and for almost all $w$ they are equal to one. Moreover, all local factors converge absolutely in the half plane $\Re(s)>1-\epsilon$ for some $\epsilon>0$ (\cite{JS}). From this we deduce that its residue at $s=1$ is given by another formula 
$$
\RR es_{s=1}L(s,\Pi\times \wt{\Pi})\prod_{w}\frac{\Psi(1,\Phi_w,W_{\phi,w},W_{\phi',w})}{L(1,\Pi_w\times \wt{\Pi}_w)}.
$$
From the two formulae of the residue, we will first deduce that  $\vartheta_w(W_{\phi,w},W_{\phi',w})$ is $H_n(E_w)$-invariant and second that 
\begin{align}
\label{eqn Psi vartheta}
\Psi(1,\Phi_w,W_{\phi,w},W_{\phi',w})=\wh{\Phi_w}(0)\vartheta_w(W_{\phi,w},W_{\phi',w}).
\end{align}
To see this,
let $N_{n,1,+}(E_w)$ be the unipotent part of the mirabolic $P_n$ and $N_{n,1,-}(E_w)$ the transpose of $N_{n,1,+}(E_w)$. We consider the open dense subset $N_{n,1,+}H_{n-1}N_{n,1,-}Z_n=P_nN_{n,1,-}Z_n$. We may decompose the measure on $H_n$ (or more precisely its restriction to the open subset)
$$
dg=|\det(h)|^{-1}\,dn_+\,dh \, dn_-\, d^\ast a,
$$where
$$
g=n_+hn_-a,\quad h\in H_{n-1},\quad n_\pm\in N_{n,1,\pm},\quad a\in Z_n.
$$
Note also that the embedding $H_{n-1}\incl P_n$ induces an isomorphism $N_{n}\bs P_n\simeq N_{n-1}\bs H_{n-1}$.
For an integrable function $f$ on $N_n\bs H_n$, we may write 
\begin{align*}
&\int_{N_n(E_w)\bs H_n(E_w)}f(g)dg\\
=&\int_{Z_nN_{n,1,-}(E_w)} \left(\int_{N_{n-1}(E_w)\bs H_{n-1}(E_w)} f(hn_- a)|\det(h)|^{-1}dh\right)\, dn_- \,d^\ast a
\end{align*}
We now apply this formula to the integral $\Psi(1,\Phi_w,W_{\phi,w},W_{\phi',w})$. For simplicity we write $W_w=W_{\phi,w}$, and $W_w'=W_{\phi',w}$.
Clearly $\vartheta_w$ is $P_n(E_w)$-invariant. Therefore we may write
$$
\Psi(1,\Phi_w,W_{\phi,w},W_{\phi',w})=\int_{Z_nN_{n,1,-}(E_w)}\Phi(ean_-)\vartheta_w(\Pi_w(an_-)W_w,\Pi_w(an_-)W_w') |a|^n\, d^\ast a \, dn_-.
$$
For $X\in E_w^n$ (with last entry nonzero), let $n_-(X)$ be the element in $Z_nN_{n,1,-}(E_w)$ with last row equal to $X$.  We consider 
$$
\Gamma(X):=
\vartheta_w(\Pi_w(n_-(X))W_w,\Pi_w(n_-(X))W_w'),
$$
whenever it is defined. A suitable substitution yields 
$$
\Psi(1,\Phi_w,W_{\phi,w},W_{\phi',w})=\int_{E_w^n}\Phi_w(X)\Gamma(X)\, dX.
$$
Since by the other residue formula, we also know that this is equal to a constant multiple times $\wh{\Phi_w}(0)$
times an invariant inner product on $\CW(\Pi_w,\psi_{E,w})$, for all $\Phi_w$ and $W_w,W_w'$. We deduce that $\Gamma(X)$ is a constant function (whenever it is defined). Therefore $\Gamma(X)=\vartheta_w(W_w,W_w')$ and $\vartheta_w$ is $H_n(E_w)$-invariant. Moreover we now have $$
\Psi(1,\Phi_w,W_{\phi,w},W_{\phi',w})=\vartheta_w(W_w,W_w') \int_{E_w^n}\Phi(X) \, dX=\vartheta_w(W_w,W_w')\wh{\Phi_w}(0).
$$
This completes the proof. We also note that if we use the $H_n(E_w)$-invariance of $\vartheta_w$, which can be proved independently, then the proposition can be deduced immediately from the two residue formulae.
\end{proof}

For later use, as we will be dealing with the case of a quadratic extension $E/F$,
we will consider $H_{n,E}$ as an algebraic group over the base field $F$. Therefore we rewrite the result as
\begin{align}\label{eqn inn prod}
\langle \phi,\phi'\rangle_{Pet}=\frac{n \cdot \Res_{s=1}L(s,\Pi\times
\widetilde{\Pi})}{\vol(E^\times \backslash\BA_E^1)}\prod_v
\vartheta_v^\nat(W_{v},W'_{v}),
\end{align}
where $\vartheta_v=\prod_{w|v}\vartheta_w$ for all (one or two) places $w$ above $v$.
\subsection{Flicker-Rallis period.}
Now let $E/F$ be a quadratic extension of number fields and $\Pi=\Pi_n$ a cuspidal automorphic representation of $H_n(\BA_E)$. Assume that its central character satisfies 
$$\omega_\Pi|_{\BA^\times}=1.
$$  We would like to decompose the Flicker-Rallis period (\cite{F0},\cite{GJaR}) explicitly. It can be viewed as a twisted version of the Petersson inner product (it indeed gives the Petersson inner product if we allow $E=F\times F$ to be split globally). Therefore it is natural that the method is similar as well.

We first assume that $n$ is odd. Then we have the global Flicker-Rallis period, an
$H_n(\BA)$-invariant linear form on $\Pi$:
\begin{align}\label{eqn defn FR}
\beta(\phi)=\beta_n(\phi):=\int_{Z_n(\BA)H_n(F)\bs H_{n}(\BA)}\phi(h)\, dh,\quad \phi\in\Pi.
\end{align}
The global period $\beta$ is related to the Asai L-function $L(s,\Pi,\RA s^+)$ (for the definition of $\RA s^\pm$, cf. \cite[\S7]{GGP}).
We set
\begin{align}
\wt{\epsilon}_n=diag(\tau^{n-1},\tau^{n-2},...,1)\in H_n(E)
\end{align}
and 
\begin{align}
\epsilon_{n-1}=\tau\cdot diag(\tau^{n-2},\tau^{n-3},...,1)=\tau\wt{\epsilon}_{n-1}\in H_{n-1}(E).
\end{align} 
(Note: $\tau\in E^-$.)  Indeed we may choose any $\wt{\epsilon}_n=diag(a_1,...,a_{n-1},a_n)$ such that $a_i/a_{i+1}\in E^-$ for $i=1,...,n-1$ and $a_n=1$.
We again use the Whittaker model of $\Pi$. We will again consider $H_{n,E}$ as an algebraic group over $F$. In particular, we consider $\Pi_v$ as a representation of $H_n(E_v)$ where $E_v=E\otimes_F F_v$ is a semisimple $F_v$-algebra of rank two. For a place $v$ of $F$, we define the local Flicker-Rallis period $\beta_v$  as follows:  for $W_v\in \CW(\Pi_v,\psi_{E_v}$)
\begin{align}\label{eqn defn FR}
 \beta_v(W_{v})=\int_{N_{n-1}(F_v)\backslash
H_{n-1}(F_v)}W_{v}\left( \begin{array}{cc}\epsilon_{n-1}h & \\
 &  1
\end{array}\right)\, dh.
\end{align}
The integral $ \beta_v$ converges absolutely if $\Pi_v$ is generic unitary. 
It depends on the choice of $\tau=\sqrt{\delta}$. 
For unramified data with normalization $W_v(1)=1$, we have
\begin{align}\label{eqn beta unram} 
\beta_v(W_{v})=\vol(K_{n,v})L(1,\Pi_v,\RA s^+).
\end{align}
We thus define a normalized linear form
\begin{equation}\label{eqn defn FR nat}
\beta^\nat_v(W_v)=\frac{\beta(W_v)}{L(1,\Pi_v,\RA s^+)}.
\end{equation} 
\begin{remark}\label{rem def L As}
For bad places $v$, we may define the local factor $L(s,\Pi_v,\RA s^+)$ as the GCD of the local zeta integral in (\ref{eqn def Psi FR}). Then the local factor $L(s,\Pi_v,\RA s^+)$ has no pole or zero when at $s=1$ for a unitary generic $\Pi_v$.
\end{remark}

\begin{prop}\label{prop FR}We have an explicit decomposition
\begin{equation}\label{eqn FR}
\beta(\phi)=\frac{n \cdot \RR es_{s=1}L(s,\Pi,\RA
s^+)}{\vol(F^\times \backslash\BA^1)} \prod_v \beta^\nat_v(W_v),
\end{equation}
where $W=W_\phi=\otimes_v W_v\in \CW(\Pi,\psi_E)$.
\end{prop}

\begin{proof}
For a Schwartz--Bruhat function
$\Phi$ on $\BA^n$, we consider the Epstein-Eisenstein series $E(g,\Phi,s)$ (cf. (\ref{eqn Ep-Es})) replacing the field $E$ by $F$. Then we define
$$
I(s,\Phi,\phi):=\int_{Z(\BA)H_n(F)\bs H_n(\BA)}E(g,\Phi,s)\phi(g)\, dg.
$$We then have (\cite[p.303]{F0})
\begin{equation}\label{eqn I=Phi}
I(s,\phi,\Phi)=\Psi(s,W_\phi,\Phi),
\end{equation}
where $$
\Psi(s,W_\phi,\Phi)=\int_{N_n(\BA)\backslash
H_n(\BA)}W_\phi(\wt{\epsilon}_n h)\Phi(e_n  h)|h|^s\, dh.
$$(Note that we use a different choice of the additive character $\psi_E$.)
Indeed, we have
\begin{align*}
I(s,\Phi,\phi)&
=\int_{P_n(F)\bs H_n(\BA)}\Phi(e_n  g)\phi(g)|g|^s\, dg
\\&=\int_{P_n(F)N_n(\BA)\bs H_n(\BA)}\Phi(e_n  g)\left(\int_{N_n(F)\bs N_n(\BA)}\phi(ng)\, dn\right)|g|^s\, dg.
\end{align*}
We have a Fourier expansion $$
\phi(g)=\sum_{\gamma\in N_n(E)\bs P_n(E)}W_\phi(\gamma g).
$$
Only those $\gamma$ such that $\psi_{E}(\gamma n\gamma^{-1})=1$ for all $n\in N_n(\BA)$ contribute nontrivially.  Therefore we may replace the sum by $\gamma \in \wt{\epsilon}_nP_n(F)$:
\begin{align*}
I(s,\Phi,\phi)&=\int_{P_n(F)N_n(\BA)\bs H_n(\BA)}\Phi(e_n g)\left(\sum_{\gamma\in N_n(F)\bs P_n(F)}W_\phi(\wt{\epsilon}_n\gamma g)\right)|g|^s\, dg
\\&=\int_{N_n(\BA)\bs H_n(\BA)}\Phi(e_n  g)W_\phi(\wt{\epsilon}_n g)|g|^s \, dg
\\&=\Psi(s,W_\phi,\Phi).
\end{align*}
(Note that $\vol(N_n(F)\bs N_n(\BA))=1$.)
We define for each place
$v$ of $F$,
\begin{align}\label{eqn def Psi FR}
\Psi(s,W_v,\Phi_v)=\int_{N_n(F_v)\backslash
H_n(F_v)}W_v(\wt{\epsilon}_n h)\Phi_v(e_n  h)|h|^s \, dh.
\end{align}
For unramified data, we have $$ \Psi(s,W_v,\Phi_v)=\vol(K_n(\CO_{F_v}))L(s,\Pi_v,\RA
s^+).
$$
And we have (\cite[p.185]{GJaR})
$$
\Psi(1,W_v,\Phi_v)=\beta_v(W_v)\widehat{\Phi}_v(0).
$$
Alternatively we may prove this using (\ref{eqn I=Phi}), analogous to the proof of Prop. \ref{prop inner prod}.
Again, analogous to the proof of Prop. \ref{prop inner prod}, we may take residue of  (\ref{eqn I=Phi}) to obtain:
$$
\frac{vol(F^\times \backslash\BA^1)}{n}\widehat{\Phi}(0)\beta(\phi)
=\RR es_{s=1}L(s,\Pi,\RA s^+)\widehat{\Phi}(0)
\prod_v \beta_v^\nat(W_v)
$$
This completes the proof.
\end{proof}
When $n$ is even, we insert the character $\eta$ in the
definition of $\beta$
\begin{align}\label{eqn defn FR even}
\beta(\phi)=\beta_n(\phi):=\int_{Z_n(\BA)H_n(F)\bs H_{n}(\BA)}\phi(h)\eta(h)\, dh,\quad \phi\in\Pi,
\end{align}
where, for simplicity, we denote $\eta(h)=\eta(\det(h)).$

The Asai L-function
is then replaced by $L(s,\Pi,\RA s^-)$, or we may write it as $L(s,\Pi,\RA s^{(-1)^{n-1}})$. We also modify the definition 
\begin{align} \beta_v(W_{v})=\int_{N_{n-1}(F_v)\backslash
H_{n-1}(F_v)}W_{v}\left( \begin{array}{cc}\epsilon_{n-1}h & \\
 &  1
\end{array}\right)\eta_v(h)\, dh.
\end{align}The same argument shows that (\ref{eqn FR}) still holds.

\subsection{Rankin-Selberg period.} We now follow \cite{JPSS}.
Let $\Pi=\Pi_n\otimes \Pi_{n+1}$ and $\Pi_i$ a cuspidal automorphic representation of
$H_{i} (\BA_E)$, $i=n,n+1$. We define the global Rankin-Selberg period as
\begin{align}\label{eqn defn RS}
\lambda(\phi)=\int_{H_n(E)\bs H_n(\BA_E)}\phi(h)\, dh,\quad \phi\in \Pi,
\end{align}
where $H_n$ embeds diagonally into $H_n\times H_{n+1}$. To decompose it,
we need the Whittaker model $\CW(\Pi_n,\ov{\psi}_E)$ ($\CW(\Pi_{n+1},\psi_E)$, resp.) of $\Pi_n$ ($\Pi_{n+1}$, resp.) with respect to the additive character $\ov{\psi}_E$ ($\psi_E$, resp.). 
We define a local Rankin-Selberg period on the local Whittaker model which associates to $W_w\in \CW(\Pi_n,\ov{\psi}_E)\otimes \CW(\Pi_{n+1},\psi_E)$:
\begin{align} \label{def local RS}
\lambda_w(s,W_w)=\int_{N_n(E_w)\bs H_n(E_w)} W_w(h)|\det(h)|^s\, dh,\quad s\in \BC,
\end{align}
and a normalized one using the local Rankin-Selberg L-function $L(s,\Pi_{n,w}\times \Pi_{n+1,w})$ (cf. \cite{JPSS}):
\begin{align}
\lambda^\nat_w(s,W_w)=\frac{\lambda_w(W_w)}{L(s+1/2,\Pi_{n,w}\times \Pi_{n+1,w})}.
\end{align}
When $\Pi_w$ is generic, the integral $\lambda_w(s,\cdot)$ is absolutely convergent when $\Re(s)$ is large enough and extends to a meromorphic function in $s\in \BC$. The normalized $\lambda^\nat_w(s,\cdot)$ extends to an entire function in $s\in \BC$. Moreover there exists $W_w$ such that $\lambda^\nat_w(s,W_w)=1$
(cf. \cite[Theorem 2.1, 2.6]{J} for archimedean places).  Therefore we will define
\begin{align}\label{eqn defn RS nat}
\lambda^\nat_w(W_w)=\lambda^\nat_w(0,W_w).
\end{align}
In particular, $\lambda_w^\nat$ defines a non-zero element of the (one-dimensional) space $\Hom_{H_n(E_w)}(\Pi_w,\BC)$ for generic $\Pi_w$.

If $\Pi_w$ is tempered, then the integral $\lambda_w(s,\cdot)$ is absolutely convergent when $\Re(s)>-1/2$ (cf. \cite[Lemma 5.3]{J} for archimedean places). Therefore in this case we may even define $\lambda_w(W_w)=\lambda(0,W_w)$ directly.

When $\Pi_w$ and $\psi_{E,w}$ are unramified, the vector $W_w$ is fixed by $K_{n,w}\times K_{n+1,w}$ and normalized by $W_w(1)=1$, we have  (\cite[p. 781]{JS2}):
\begin{align}\label{eqn lambda unram} 
\lambda_w(s,W_w)=\vol(K_{n,w})L(s+1/2,\Pi_{n,w}\times \Pi_{n+1,w}),
\end{align}
and therefore 
$$
\lambda^\nat_w(W_w)=\vol(K_{n,w}).
$$
We also form the global (complete) Rankin-Selberg L-function 
$$
L(s,\Pi_n\times \Pi_{n+1})=\prod_w L(s,\Pi_{n,w}\times \Pi_{n+1,w}).
$$
It is an entire function in $s\in \BC$.
\begin{prop}\label{prop RS}
 We have the following decomposition if $\Pi$ is cuspidal unitary, and $\phi\in \Pi$
\begin{equation}\label{eqn local RS}
\lambda(\phi)=L(\frac{1}{2},\Pi_n\times \Pi_{n+1})\prod_w \lambda^\nat_w(W_w),
\end{equation}
where $W_\phi=\prod_w W_{\phi,w}$ is as before.
\end{prop}
\begin{proof}
This is due to 
Jacqquet, Piateski-Shapiro and Shalika  (\cite{JPSS}).
\end{proof}
For generic unitary $\Pi$, we need to use the completed L-function. Indeed the local L-factor may have poles at $s=1/2$ since we do not know the temperedness of each $\Pi_w$.

\subsection{Decomposing the spherical character on the general linear group.} We now denote
\begin{align}
G'=\Res_{E/F}\left(\GL_{n}\times \GL_{n+1}\right)
\end{align}viewed as an $F$-algebraic group. We consider its two subgroups: $H_1'$ is the diagonal embedding of $\Res_{E/F}\GL_{n}$ (where $\GL_{n}$ is embedded into $\GL_{n+1}$ by $g\mapsto diag[g,1]$) and $H_2'$ is $\GL_{n,F}\times \GL_{n+1,F}$ embedded into $G'$ in the obvious way. 

Now we consider a cuspidal automorphic representation $\Pi=\Pi_n \otimes
\Pi_{n+1}$ of $G'(\BA)$.  Denote by $\beta=\beta_n\otimes \beta_{n+1}$ the (product of) Flicker--Rallis period on $\Pi=\Pi_n \otimes
\Pi_{n+1}$. 
\begin{defn}
We define the {\em global spherical character} $I_\Pi$ as the following distribution on $H(\BA)$: for $f'\in \sC_c^\infty(G'(\BA))$, 
\begin{align}\label{def char G'}
I_\Pi(f')=\sum_{\phi}\frac{\lambda(\Pi(f')\phi)\overline{\beta(\phi)}}{\langle
\phi,\phi\rangle_{Pet}},
\end{align}
where the sum runs over an {\em orthogonal} basis of $\Pi$. Equivalently,\begin{align*}
I_\Pi(f')=\sum_{\phi}\lambda(\Pi(f')\phi)\overline{\beta(\phi)},
\end{align*}
where the sum runs over an {\em orthonormal} basis of $\Pi$ for the Petersson inner product.
\end{defn}
 Note that the definition of $\Pi(f')$ involves a choice of the measure on $G'(\BA)$ to define the Petersson inner product. We could choose any one as long as then we use the same measure (quotient by the counting measure on $G'(F)$).

By definition of $\RA s^\pm$, we have for $i=n,n+1$:
$$L(s,\Pi_i\times \Pi_i^\sigma)=L(s,\Pi_i,\RA s^+)L(s,\Pi_i,\RA s^-).$$

Now recall that in Introduction we have a product of unitary groups $G=U(W)\times U(V)$ for Hermitian spaces $W\subset V$ with $\dim W=n,\dim V=n+1$. Assume that $\Pi=\pi_E$ is the base change of a cuspidal automorphic representation $\pi=\pi_n\otimes \pi_{n+1}$ of $G(\BA)$. By \cite[Prop. 7.4]{GGP} we also have $$
L(s,\Pi_i,\RA s^{(-1)^i})=L(s,\pi_i,Ad).
$$ 
\begin{remark}\label{rem def L Ad}
For bad places $v$, we may define the local factor $L(s,\pi_i,Ad)$ by this formula. 
But note that for our purpose, it only matters to know the local L-factors at unramified places. 
\end{remark}

Since such $\Pi$ must be  conjugate self-dual:
$\widetilde{\Pi}\simeq \Pi^\sigma$ where $\sigma$ is the nontrivial element in $\Gal(E/F)$, we deduce that $L(s,\Pi_i\times \wt{\Pi}_{i})$ has a simple pole at $s=1$. By our running hypothesis {\bf RH(I)}(i), the Asial $L(s,\Pi_i,\RA s^{(-1)^{i-1}})$ has a simple pole. We conclude that $L(s,\Pi_i,\RA s^{(-1)^i})=L(s,\pi_i,Ad)$ is regular at $s=1$ and:
\begin{align}\label{eqn ad=As}
\frac{\RR es_{s=1}L(1,\Pi_i\times \wt{\Pi}_{i})}{\RR es_{s=1}L(s,\Pi_i,\RA s^{(-1)^{i-1}})}=L(1,\Pi_i,\RA s^{(-1)^i})=L(1,\pi_i,Ad).
\end{align}

We denote by $\CW(\Pi,\psi)$ the Whittaker model $\CW(\Pi_n,\ov{\psi})\otimes\CW(\Pi_{n+1},\psi)$. Let $\Pi=\otimes _v\Pi_v$. Let $
\lambda^\nat_v$, $\beta^\nat_v$ be the local Rankin-Selberg period (\ref{eqn defn RS nat}) and the local Flicker-Rallis period (\ref{eqn defn FR nat}). Let $\vartheta^\nat_v$ be the normalized local  invariant inner product (\ref{eqn defn vartheta nat}).
\begin{defn}
We define the {\em normalized local spherical character} $I^\nat_{\Pi_v}$ associated to a unitary generic representation $\Pi_v$ :
\begin{align}\label{def local char G'}
I^\nat_{\Pi_v}(f'_v)=\sum_{W_v}
\frac{\lambda^\nat_v(\Pi_v(f'_v)W_v)\overline{\beta_v^\nat}(W_v)}{\vartheta^\nat_v(W_v,W_v)},
\end{align}
where the sum runs over an {\em orthogonal} basis $W_v\in \CW(\Pi_v,\psi_{v})$.  We also define an unnormalized  local spherical character $I_{\Pi_v,s}$ as the meromorphic function in $s\in\BC$
\begin{align}\label{def unnorm char G'}
I_{\Pi_v,s}(f'_v)=\sum_{W_v}
\frac{\lambda_v(s,\Pi_v(f'_v)W_v)\overline{\beta_v}(W_v)}{\vartheta_v(W_v,W_v)}.
\end{align}
We will write $I_{\Pi_v}(f'_v)$ for its value $I_{\Pi_v,0}(f'_v)$ at $s=0$.
\end{defn}

We now summarize to arrive at an
analogue of the decomposition in Conjecture \ref{conj dis}:
\begin{prop}
\label{prop char GL}
Assume that the cuspidal automorphic representation $\Pi$ of $G'(\BA)$ is the base change $\pi_E$ of a cuspidal automorphic representation $\pi$ of $G(\BA)$. Then we have
\begin{equation}\label{eqn I Pi=C
local}I_\Pi(f')=L(1,\eta)^2\frac{L(1/2,\Pi)}{L(1,\pi,Ad)} \prod_v I^\nat_{\Pi_v}(f'_v).
\end{equation}
\end{prop}
\begin{proof}By the assumption, $\Pi$ is unitary generic. 
Note that 
$$
\frac{\vol(E^\times \backslash\BA_E^1)}{\vol(F^\times \backslash\BA_F^1)}=L(1,\eta).
$$
Then the result follows from Prop. \ref{prop inner prod}, \ref{prop FR}, \ref{prop RS} and the relation (\ref{eqn ad=As}).
\end{proof}

\begin{remark}
Note that we do not need to assume the temperedness of $\pi$ at this moment.
\end{remark}

\section{Relative trace formulae of Jacquet and Rallis}

\subsection{The construction of Jacquet and Rallis.}
We recall the Jacquet--Rallis relative trace formulae (\cite{JR}) and we refer to \cite{Zh2} for more details.

First we recall the construction of Jacquet--Rallis' RTF in  the unitary group case.
For $f \in \sC_c^\infty(G(\BA))$  we consider a kernel
function
$$
K_f(x,y)=\sum_{\gamma\in G(F)}f(x^{-1}\gamma y),
$$
 and a distribution
$$
J(f):=\int_{H(F)\bs H(\BA)} \int_{ H(F)\bs H(\BA)}K_f(x,y)\,dx \, dy.
$$The integral converges when the test function $f$ is nice in the sense of \cite[\S2.3]{Zh2} (the precise definition will not be used in this paper). Associated to the RTF we have two objects:  
 \begin{itemize}
 \item the global spherical character $J_\pi$ associated to a cuspidal automorphic representation $\pi$ of $G(\BA)$ (Definition  \ref{def global char}  in Introduction), and
 
 \item
  the (relative) orbital integral associated to
 a {\em regular semisimple} element \footnote{See \cite[\S2.1]{Zh1} for the definition, where ``regular" corresponds to ``regular semisimple" in this paper.}  $\delta\in G(F)$: for $f\in \sC_c^\infty(G(\BA))$,  we define its orbital integral:
 \end{itemize}
 \begin{align}
O(\delta,f):=\int_{H(\BA)\times H(\BA)}f(x^{-1}\delta y)\, dx\, dy.
\end{align}
We have a local counterpart associated to a regular semisimple element $\delta\in G(F_v)$: for $f_v\in \sC_c^\infty(G(F_v))$ we define
 \begin{align}\label{defn orb delta}
O(\delta,f_v)=\int_{H(F_v)\times H(F_v)}f_v(x^{-1}\delta y)\, dx\, dy.
\end{align}
 
We now recall the RTF in the  general linear group case. Recall that $G'=\Res_{E/F}\left(\GL_{n}\times \GL_{n+1}\right)$ as an $F$-algebraic group. We consider its two subgroups:
\begin{itemize}
\item $H'_1$ is the diagonal embedding of $\Res_{E/F}\GL_{n}$ (where $\GL_{n}$ is embedded into $\GL_{n+1}$ by $g\mapsto diag[g,1]$), and
\item  $H_2'=\GL_{n,F}\times \GL_{n+1,F}$ embedded into $G'$ in the obvious way.
\end{itemize} 
For $f'\in \sC_c^\infty(G'(\BA))$, we define a kernel
function $$K_{f'}(x,y)=\int_{Z_{H'_2}(\BA)}\sum_{\gamma\in
G'(F)}f'(x^{-1}\gamma z y)dz.$$ 
We then consider a distribution on $G'(\BA)$:
$$
I(f')=\int_{H'_1(F)\backslash H_1'(\BA) }
\int_{Z_{H'_2}(\BA)H_2'(F)\backslash H_2'(\BA)}K_{f'}(h_1,h_2)\eta(h_2)\, dh_1\, dh_2,
$$where $\eta(h_2):=\eta^{n-1}(g_n)\eta^{n}(g_{n+1})$ if $h_2=(g_n,g_{n+1})\in H_n(\BA)\times H_{n+1}(\BA)$. The integral converges when the test function $f$ is nice in the sense of \cite[\S2.2]{Zh2}. Associated to the RTF we have two objects:  
 \begin{itemize}
 \item the global spherical character $I_\Pi$ (cf. \cite[\S 2]{Zh1}) associated to a cuspidal automorphic representation $\Pi$ of $G'(\BA)$ (Definition  \ref{def char G'}), and

 \item   the (relative) orbital integral associated to a regular semisimple element (cf. \cite[\S 2]{Zh1}) $\gamma\in G'(F)$: for $f'\in \sC_c^\infty(G'(\BA))$,  we define its orbital integral:
\begin{equation}
O(\gamma,f'):= \int_{H_1'(\BA)}\int_{H_2'(\BA)} f'(h_1^{-1}\gamma h_2)\eta(h_2)\, dh_1\, dh_2.
\end{equation}
 \end{itemize}
Similarly we have a local counterpart:
a regular semisimple element $\gamma\in G'(F_v)$: for $f'_v\in \sC_c^\infty(G'(F_v))$ we define
 \begin{align}\label{defn orb gamma}
O(\gamma,f'_v)=\int_{H_1'(F_v)}\int_{H_2'(F_v)} f'_v(h_1^{-1}\gamma h_2)\eta(h_2)\, dh_1\, dh_2.
\end{align}

We now recall the comparison of the orbits (cf. \cite[\S2]{Zh1}). Denote by $(H_1(F)\bs G'(F)/H_2(F))_{rs}$ the set of regular semisimple $(H'_1\times H_2')(F)$-orbits in $G'(F)$ and $(H(F)\bs G(F)/H(F))_{rs}$ the set of regular semisimple $(H\times H)(F)$-orbits in $G(F)$. We need to vary the pair $W\subset V$ of Hermitian spaces of dimension $n$ and $n+1$ modulo the equivalence relation: $(W,V)$ is equivalent to $(W',V')$ if there a constant $\kappa\in F^\times$ such that $\kappa W\simeq W'$ and $\kappa V\simeq V'$ (here $\kappa W$ means that we multiply the Hermitian form by the constant $\kappa$). Without loss of generality, we may and will assume that  $V$ is an {\em orthogonal} sum of $W$ and a one-dimensional Hermitian space $Ee$ with a norm one vector: 
\begin{align}\label{eqn V=W+e}
V=W \oplus Ee,\quad \pair{e,e}=1.
\end{align} In particular, $V$ is determined by $W$ so that we only need to vary the Hermitian space $W$. \footnote{In terms of \cite[\S2]{GGP}, we only consider Hermitian pairs $(W,V)$ that are {\em relevant} to each others.}
To indicate the dependence on the Hermitian spaces $W$, we will write $G_{W}$ for $G$ and $H_W$ for $H$. Then there is a natural bijection (\cite[Lemma 2.3]{Zh1})
 \begin{align}\label{eqn orbit match}
 (H_1(F)\bs G'(F)/H_2(F))_{rs}\simeq \coprod_{W} (H_W(F)\bs G_{W}(F)/H_W(F))_{rs},
 \end{align}
 where on the right hand side the disjoint union runs over all Hermitian space $W$ of dimension $n$. 
  Moreover, the same holds if we replace $F$ by $F_v$ for every place $v$ of $F$.
 When $v$ is non-archimedean, there are precisely two isomorphism classes of Hermitian spaces $W_v$.
 
In \cite[\S2.4]{Zh2} we defined an explicit {\em transfer factor $\{\Omega_v\}_v$} 
on the regular semisimple locus of $G'(F_v)$ for any place $v$. It satisfies the following properties:
\begin{itemize}
\item If $\gamma\in G'(F)$ is regular semisimple, then we have a product formula $\prod_{v}\Omega_{v}(\gamma)=1$. 
\item For any $h_i\in H'_{i}(F_v)$ and $\gamma\in G'(F_v)$, we have $\Omega(h_1\gamma h_2)=\eta(h_2)\Omega_v(\gamma)$.
\end{itemize}
The construction is as follows. It depends on an auxiliary character $\eta'$: 
\begin{align}\label{eqn char eta'}
\eta':E^\times \bs \BA_E^\times \to \BC^\times
\end{align}(not necessarily quadratic) such that its restriction $\eta'|_{\BA^\times}=\eta$.
 Let $S_{n+1}$ be the subvariety of $\Res_{E/F}\GL_{n+1}$ defined by the equation $s\bar s=1$. By Hilbert Satz-90, we have an isomorphism of two affine varieties $$\Res_{E/F}\GL_{n+1}/\GL_{n+1,F}\simeq S_{n+1},$$
induced by the following morphism $\nu$ between $F$-varieties:
\begin{align}
\label{eqn def nu}
\nu: \Res_{E/F} \GL_{n+1}&\to S_{n+1}\\
g& \mapsto g\bar g^{-1},
\end{align}
and in the level of $F$-points:
\begin{align}\label{eqn def nu0}
\GL_{n+1}(E)/\GL_{n+1}(F)\simeq S_{n+1}(F).
\end{align}
  Write $\gamma=(\gamma_1,\gamma_2)\in G'(F_v)$ and $s=\nu(\gamma_1^{-1}\gamma_2)$.
 We define for a regular semisimple $s\in S_{n+1}(F_v)$:
  \begin{align}\label{eqn def Omega(s)}
\Omega_v(s):=\eta_v'(\det(s)^{-[(n+1)/2]}\det(e,es,...,es^{n})).
\end{align}
Here $e=e_{n+1}=(0,...,0,1)$ and $(e,es,...,es^{n})\in M_{n+1}$ is the matrix whose $i$-th row is $es^{i-1}$.
 If $n$ is odd, we define:
\begin{align}
\Omega_v(\gamma):=\eta'_v(\det(\gamma_1^{-1}\gamma_2))\Omega_v(s),
\end{align}
and if $n$ is even, we simply define:
\begin{align}
\Omega_v(\gamma):=\Omega_v(s).
\end{align}

For a place $v$ of $F$,  we say that the function $f'\in \sC^\infty_c(G'(F_v))$ and the tuple $(f_{W})_{W}, f_{W}\in \sC_c^\infty(G_{W}(F_v))$, indexed by the set of all equivalence classes of Hermitian spaces $W$ over $E_v=E\otimes F_v$, {\em are smooth transfer of each other} or {\em match} if 
\begin{align}\label{eqn mat grp}
\Omega_v(\gamma)O(\gamma,f')=O(\delta,f_{W}),
\end{align}
whenever  a regular semisimple $\gamma\in G'(F_v)$ matches $\delta\in G_{W}(F_v)$ via (\ref{eqn orbit match}).
One of the main local results in \cite{Zh2} is the existence of smooth transfer at non-archimedean non-split places (cf. \cite[Theorem 2.6]{Zh2}) and arbitrary split places (cf. \cite[Prop. 2.5]{Zh2}). 

In this paper, we usually need to consider a fixed $W$ and we say that $f'$ and $f_{W}\in \sC_c^\infty(G_{W}(F_v))$ match if there exist some $f_{W'}$ for each equivalence class $W'\neq W$ such that $f'$ matches the completed tuple $f_{W},f_{W'}$.

Moreover,  the fundamental lemma of Jacquet--Rallis predicts a specific case of matching functions: 

\begin{thm}[\cite{Y}]\label{thm FL}
Assume that the quadratic extension $E_v/F_v$ is unramified. Denote by $\{W_v,W'_v\}$ the two isomorphism classes of Hermitian spaces of dimension $n$ where $W_v$ contains a self-dual (with respect to the Hermitian form) $\CO_{E_v}$-lattice. Set  \footnote{Note that the measures in the fundamental lemma proved in \cite{Y} are different from ours.}
\begin{align}
f_{W_v}=\frac{1}{\vol(H_{W_v}(\CO_v))^2}1_{G_{W_v}(\CO_v)},\quad f_{W'_v}=0,\quad f'_v=\frac{1}{\vol(H_1'(\CO_v))\vol(H'_2(\CO_v))}1_{G'(\CO_v)}.
\end{align}
Then there is a constant $c(n)$ depending only on $n$ such that, when the characteristic of the residue field of $F_v$ is larger than $c(n)$, the function $f_v'$ matches  the pair $(f_{W_v},f_{W'_v})$.

\end{thm}

We need some simplification of orbital integrals (\cite[\S2.1]{Zh2}). 
Identify $H_1'\backslash G'$ with $\Res_{E/F}\GL_{n+1}$.  
Now we write $F$ for $F_v$ for a fixed place $v$. We may integrate $f'$ over $H'_1(F)$ to get a function on $\Res_{E/F}\GL_{n+1}(F)$:
\begin{align}\label{eqn def wtf}
\wt f'(g):=\int_{H'_1(F)}f'(h_1(1,g))\, dh_1,\quad g\in \Res_{E/F}\GL_{n+1}(F).
\end{align}
Using the fiber integral of $\nu$ (cf. (\ref{eqn def sn0}) and (\ref{eqn def nu0})) we define
\begin{align}\label{eqn def nuf even}
\wt{\wt {f'}}(s):=\int_{H_{n+1}(F)}\wt{f'}(gh)\,dh,\quad \nu(g)=s,
\end{align}
if $n$ is even, and
\begin{align}\label{eqn def nuf odd}
\wt{\wt{f'}}(s):=\int_{H_{n+1}(F)}\wt{f'}(gh)\eta'(gh)\,dh,\quad \nu(g)=s,
\end{align}
when $n$ is odd (then this depends on the auxiliary character $\eta'$).
Then $\wt{\wt{f'}}\in \sC_c^\infty(S_{n+1}(F))$ and all functions in  $\sC_c^\infty(S_{n+1}(F))$ arise in this way. 

Now it is easy to see that for $\gamma=(\gamma_1,\gamma_2)$:
\begin{align}
O(\gamma,f')=\eta'(\det(\gamma_1^{-1}\gamma_2))\int_{H_{n}(F)}\wt{\wt{f}}'(h^{-1}s h)\eta(h)\, dh,\quad s=\nu(\gamma_1^{-1}\gamma_2),
\end{align} 
if $n$ is odd, and 
\begin{align}
O(\gamma,f')=\int_{H_{n}(F)}\wt{\wt{f'}}_v(h^{-1}s h)\eta(h)\,dh,\quad s=\nu(\gamma_1^{-1}\gamma_2),
\end{align}
if $n$ is even.
Up to a sign, the integral on the right hand side depends only on the orbit of $s$ under the conjugation by $H_{n}(F)$. 
Therefore, we define the orbital integral associated to a regular semisimple element $s\in S_{n+1}(F)$:
\begin{align}
O(s,\wt{\wt{f'}}):=\int_{H_n(F)} \wt{\wt{f'}}(h^{-1}sh)\eta(h)\, dh,\quad \wt{\wt{f'}}\in \sC_c^\infty(S_{n+1}(F)).
\end{align}
Then we always have, for regular semisimple $\gamma=(\gamma_1,\gamma_2)\in G'(F_v)$ (cf. (\ref{eqn def Omega(s)})):
\begin{align}
\Omega(\gamma)O(\gamma,f')=\Omega(s)O(s,\wt{\wt{f'}}),\quad s=\nu(\gamma_1^{-1}\gamma_2).
\end{align}

 \subsection{A trace formula identity.}
 We are lead to a comparison of the two RTFs and the two spherical characters $I_\Pi$ and $J_\pi$ when $\Pi=\pi_E$ is the base change of $\pi$.  
\begin{conj}\label{conj tr id}
Let $\pi$ be an irreducible cuspidal automorphic representation on $G(\BA)$ that admits invariant linear functional: 
$$
\Hom_{H(\BA)}(\pi,\BC)\neq 0.
$$
Let $\pi_E$ be the base change of $\pi$ and assume that $\pi_E$ is cuspidal. Then, for every $f\in\sC_c^\infty(G(\BA))$ and a smooth transfer $f'\in\sC_c^\infty(G'(\BA))$ of $f$, we have :
$$
2^{-2}L(1,\eta)^{-2}I_{\pi_E}(f')=J_\pi(f).
$$

\end{conj}
 
 \begin{remark}
Note that we do not need to assume that $\pi$ is tempered.
\end{remark}

\begin{thm} \label{thm tr id}
Assume the following:
\begin{itemize}
\item[(1)] At a split place $v_1$, $\pi_{v_1}$ is supercuspidal.  
\item[(2)] The test functions $f$ and $f'$ are nice and $f'$ is a smooth transfer of $f$.
\end{itemize}
Then Conjecture \ref{conj tr id} holds for such $\pi$ and the test functions $f,f'$. 
\end{thm}
  
\begin{proof}
We would like to apply the result from \cite{Zh2}. But we need to compare the difference on the normalization of Petersson inner product in the unitary group case  (caused by the presence of center). There implicitly we use a different Petersson inner product
$$
\pair{\phi,\phi'}'=\int_{Z(\BA)G(F)\bs G(\BA)}\phi(g)\ov{\phi'}(g)\, dg=\vol(Z(F)\bs Z(\BA))^{-1}\pair{\phi,\phi'}.
$$
Note that the center $Z$ of $G$ is isomorphic to $U(1)\times U(1)$. Hence the volume for our choice of measure is (cf. (\ref{eqn vol U(1)}))
$$
\vol(Z(F)\bs Z(\BA))=(2L(1,\eta))^2.
$$
Now taking into this correction, we apply the trace formula identity \cite[Prop. 2.11]{Zh2}: if a nice function $f'$ matches a tuple $(f_W)$ indexed by equivalence classes of $W$, we have:
$$
I_{\pi_E}(f')=(2L(1,\eta))^2\sum_{W}\sum_{\pi_W}J_{\pi_W}(f_W),
$$
where the sum is over all equivalence classes of $W$ and all cuspidal automorphic representations $\pi_W$ of $G_W(\BA)$ that are nearly equivalent to $\pi$ and at $v_1$ all $\pi_{W,v_1}$ are isomorphic to $\pi_{v_1}$. We denote by $(W_0,V_0)$ the Hermitian spaces we started with, $\pi_{W_0}=\pi$,  and by $f_{W_0}=f$ the function in the assumption of the theorem.

By our running hypothesis {\bf RH(I)},  we have:
\begin{itemize}
\item[(1)]
the multiplicity of each cuspidal $\pi_W$ in $L^2([G_W])$ is one. Namely, for a fixed $W$, all $\pi_W$ occurring in the sum are non-isomorphic.
\item[(2)]Note that for all $W$, all $\pi_W$ occurring in the sum are in the same nearly equivalence class and $\pi_{W,v_1}$ are supercuspidal (so $\pi_E$ is cuspidal and particularly $\pi_{E,v}$ is generic for every $v$). Hence for every $v$, the $\pi_{W,v}$'s are in the same Vogan L-packet and this L-packet is generic.\end{itemize}

Now by our running hypothesis {\bf RH(II)}, there exists at most one $\pi_W$ and $W$ in the sum such that $\Hom_{H_W(\BA)}(\pi_W,\BC)\neq 0$.  By our assumption $\Hom_{H(\BA)}(\pi,\BC)\neq 0.$ 
Hence the sum reduces to one term contributed by the $\pi$ we started with:
$$
I_{\pi_E}(f')=(2L(1,\eta))^2J_{\pi}(f).
$$
\end{proof}
\begin{remark}
If $\pi_E$ is not cuspidal, then we may reformulate the conjecture at least for tempered representation $\pi$. We also need to regularize the definition of $I_{\pi_E}$ in the above equality and the constant $2^2$ should be replaced by $|S_\pi|$. Then the analogous conjecture should ultimately follow from the full spectral decomposition of  the Jacquet-Rallis relative trace formulae. 
\end{remark}

\subsection{Reduction to a local question.}
Our main ingredient is an identity between the two local distributions $I_{\Pi_v}$ (cf. (\ref{def local char G'}) and $J_{\pi_v}$ (cf. (\ref{def local char G})). Note that the distribution  $J_{\pi_v}$ does not depend on the choice of the inner product on $\pi_v$. Denote $d_n={n \choose 3}$, which satisfies:
\begin{align}\label{eqn def d_n}
\tau^{d_n}=\delta_{n-1}(\epsilon_{n-1})=\det(Ad(\epsilon_{n-1}): N_{n-1}(E)).
\end{align}
We have a local conjecture:
\begin{conj}
\label{conj local} Let $\pi_v=\pi_{n,v}\otimes \pi_{n+1,v}$ be an irreducible tempered unitary representation of $G(F_v)$ with $\alpha_v\neq 0$. Assume that the base change $\Pi_v=\Pi_{n,v}\otimes \Pi_{n+1,v}$ of $\pi_v$ is generic unitary (so that $I_{\Pi_v}$ is well-defined). If the functions $f_v\in\sC^\infty_c(G(F_v))$ and $f_v'\in \sC^\infty_c(G'(F_v))$ match, then we have
\begin{align}\label{eqn conj local}
 I_{\Pi_v}(f_v')=\kappa_vL(1,\eta_v)^{-1} J_{\pi_v}(f_v),
 \end{align}
 where the constant $\kappa_v$ is given by
$$
\kappa_v=\kappa_v(\eta',\tau,n,\psi)=|\tau|_{E,v}^{(d_n+d_{n+1})/2}(\epsilon(1/2,\eta_{v},\psi_v)/\eta'(\tau))^{n(n+1)/2}\eta_v(\disc(W))\omega_{\Pi_{n,v}}(\tau).
$$ 
Here $\omega_{\Pi_{n,v}}$ is the central character of $\Pi_{n,v}$, and $\disc(W)\in F^\times/\RN E^\times$ is the discriminant of the Hermitian space $W$, $I_{\Pi_v}$ ($J_{\pi_v}$, resp.) is defined by (\ref{def unnorm char G'}) ((\ref{def local char G unnorm}), resp.).
\end{conj}

\begin{prop}\label{prop imp}
Let $\pi$ be a tempered cuspidal automorphic representation of $G(\BA)$ with cuspidal base change $\Pi=\pi_E$. Assume that there exists a test function $f=\otimes f_v$ and a smooth transfer $f'=\otimes f'_v$ such that for every place $v$: $$
J^\nat_{\pi_v}(f_v)\neq 0.
$$ 
 Assume that \begin{itemize}
 \item Conjecture \ref{conj tr id} holds for $\pi$, $f,f'$.
 \item  For every $v$, Conjecture \ref{conj local} holds for $\pi_v,f_v,f_v'$.
 \end{itemize} Then Conjecture \ref{conj dis} and \ref{conj IIH} holds for $\pi$.
 \end{prop}
 
\begin{proof}
By Conjecture \ref{conj tr id} and Prop. \ref{prop char GL} we have
 \begin{align*}
 J_\pi(f)&=2^{-2}L(1,\eta)^{-2}I_{\Pi}(f')=2^{-2}\frac{L(1/2,\pi_E)}{L(1,\pi,Ad)} \prod_{v} I^\nat_{\Pi_v}(f_v').
\end{align*}
Conjecture \ref{conj local} is equivalent to the identity between the normalized distributions
 $$
 I^\nat_{\Pi_v}(f_v')=\kappa_vL(1,\eta_v)^{-1} \Delta_{n+1,v} J^\nat_{\pi_v}(f_v).
 $$ 
 Since $$\prod_v \epsilon(1/2,\eta_{v},\psi_v)=\epsilon(1/2,\eta)=1,$$ we have 
 $$
\prod_v \kappa_v=1.
$$Note that the product $\prod_{v}L(1,\eta_v)^{-1} \Delta_{n+1,v} $ converges absolutely to $L(1,\eta)^{-1} \Delta_{n+1}$. We thus obtain
 \begin{align*}
 J_\pi(f)=2^{-2}L(1,\eta)^{-1}\Delta_{n+1}\frac{L(1/2,\pi_E)}{L(1,\pi,Ad)} \prod_{v} J^\nat_{\pi_v}(f_v).
\end{align*}
Note that the global measure on $H$ and $G$  in Introduction are normalized by $L(1,\eta)^{-1}$ and $L(1,\eta)^{-2}$ respectively. Correction of measures yields Conjecture \ref{conj dis} for the choice of test function $f$.  Since $J^\nat_{\pi_v}(f_v)\neq 0$ for all $v$ (and equal to one for almost all $v$), it follows that Conjecture \ref{conj dis}  holds for all test functions $f\in \sC_c^\infty(G(\BA))$. We have shown in Lemma \ref{lem equiv conj} that Conjecture \ref{conj dis} implies Conjecture \ref{conj IIH}.
\end{proof}

We have the following evidence of Conjecture \ref{conj local}. 
\begin{thm}\label{thm local} Let $v$ be a place of $F$ and let $\pi_v$ be a tempered representation as in Conjecture \ref{conj local}. 
\begin{itemize}
\item[(1)]Conjecture \ref{conj local} holds if the place $v$ is split in $E/F$. 
\item[(2)] If $v$ is a non-archimedean place non-split in $E/F$,  then under any one of the following conditions, 
there exists $f_v$ and a smooth transfer $f'_v$, such that the equality  (\ref{eqn conj local}) holds and $J_{\pi_v}(f_v)\neq 0$:
\begin{itemize}
\item[(i)] The representation $\pi_v$ is unramified and the residue characteristic $p\geq c(n)$,
\item[(ii)] The group $H(F_v)$ is compact,
\item[(iii)] The representation $\pi_v$ is supercuspidal,
\end{itemize}
\end{itemize}
 \end{thm}

Below we first prove Theorem \ref{thm local} when $\pi_v$ is unramified (case (1)) or $v$ is split in Corollary \ref{cor thm local} (case (2)-(i)). We postpone the proof of the case (2)-(ii) and (2)-(iii) to the last part of \S9.

We now give the proof of the first part of Theorem \ref{thm main} assuming Theorem \ref{thm local}. 
 \begin{proof}[Proof of Theorem \ref{thm main}: Case $(1)$]
 We may assume that $\Hom_{H(F_v)}(\pi_v,\BC)\neq 0$ for all $v$ (otherwise the formula holds trivially). This implies that the linear form $\alpha_v'$ does not vanish for all $v$. We then construct nice test functions $f=\otimes f_v$ on $G(\BA)$ and $f'=\otimes f_v'$ as follows
\begin{itemize}
\item  at each inert $v$ with residue characteristic $p\geq c(n)$, $f_v,f_v'$ are given by the fundamental lemma (Theorem \ref{thm FL}).
\item at each $v\in \Sigma$,
we choose $f_v,f_v'$ as in (2)-(ii) or (2)-(iii) of Theorem \ref{thm local}.
\item at almost every split place, we choose the unit element in the spherical Hecke algebra.
\item at the remaining finitely many split places including $v_0$ and the archimedean ones, we choose suitable functions so that $f,f'$ are nice and such that $J_{\pi_v}(f_v)\neq0$.
\end{itemize} 
Apply Theorem \ref{thm tr id} to $\pi$ and $f,f'$ to obtain
 $$
2^{-2}L(1,\eta)^{-2}I_{\pi_E}(f')=J_\pi(f).
$$
Now Theorem \ref{thm main} case (1) follows from Prop. \ref{prop imp}. 
 \end{proof}

\subsection{Proof of Theorem \ref{thm local}: the case of $\pi_v$ unramified and $p\geq c(n)$.} Then we may assume that 
\begin{enumerate}
\item The quadratic extension $E/F$ is unramified at $v$. 
\item The number $\tau$ is a $v$-adic unit.
\item The character $\psi$ is unramified and hence so is $\psi_E$.
\end{enumerate}
Indeed, it is easy to see how $I_{\Pi_v}$ depends on $\tau$: only the local period $\beta_v$ involves the choice of $\tau$ and we see that $ |\tau|_{E,v}^{-(d_n+d_{n+1})/2}\eta'(\tau)^{n(n+1)/2}I_{\Pi_v}$ is independent of the choice of $\tau$.
If we twist $\psi$ by $a\in F_v^\times$, it amounts to change $\tau$ by $a\tau$.
 
We need to utilize the fundamental lemma: by Theorem \ref{thm FL}, we have a matching pair:
$$
f_v=\frac{1}{\vol(H(\CO_v))^2}1_{G(\CO_v)},\quad f'_v=\frac{1}{\vol(H_1'(\CO_v))\vol(H'_2(\CO_v))}1_{G'(\CO_v)}.
$$
Let $W_0\in \CW(\Pi_v,\psi_E)$ be the unique spherical element normalized such that $W_0(1)=1$.
Then we have
$$
\Pi_v(f'_v)W_0=\frac{\vol(G'(\CO_v))}{\vol(H_1'(\CO_v))\vol(H'_2(\CO_v))}W_0.
$$
and
$$
I_{\Pi_v}(f'_v)=\frac{\lambda(\Pi_v(f'_v)W_0)\ov{\beta}(W_0)}{\vartheta_v(W_0,W_0)}=\frac{\vol(G'(\CO_v))}{\vol(H_1'(\CO_v))\vol(H'_2(\CO_v))}\frac{\lambda(W_0)\ov{\beta}(W_0)}{\vartheta_v(W_0,W_0)}.
$$
Note that by (\ref{eqn lambda unram})
$$
\lambda(W_0)=L(1/2,\Pi_v)\cdot \vol(H_1'(\CO_v))
$$and by (\ref{eqn vartheta unram}) and (\ref{eqn beta unram})
$$
\frac{\ov{\beta}(W_0)}{\vartheta_v(W_0,W_0)}=L(1,\pi_v,Ad)^{-1}\frac{\vol(H'_2(\CO_v))}{\vol(H_n(\CO_{E,v}))\vol(H_{n+1}(\CO_{E,v}))}.
$$We obtain
\begin{align*}I_{\Pi_v}(f'_v)&=
\frac{L(1/2,\Pi_v)}{L(1,\pi_v,Ad)}\cdot\frac{\vol(G'(\CO_v))}{\vol(H_1'(\CO_v))\vol(H'_2(\CO_v))}\cdot \frac{\vol(H_1'(\CO_v))\vol(H'_2(\CO_v))}{\vol(H_n(\CO_{E,v}))\vol(H_{n+1}(\CO_{E,v}))}.
\end{align*}
In summary we have:
\begin{align}\label{eqn Ipi f unr}
I_{\Pi_v}(f'_v)=\frac{L(1/2,\Pi_v)}{L(1,\pi_v,Ad)}.
\end{align}
In the unitary group case , we take $\phi_0\in \pi_v^{K_v}$ normalized by $\pair{\phi_0,\phi_0}=1$:
$$
\pi_v(f_v)\phi_0=\frac{\vol(G(\CO_v))}{\vol(H(\CO_v))^2}\phi_0.
$$Therefore we have 
$$
J_{\pi_v}(f_v)=\alpha_v(\pi_v(f_v)\phi_0,\phi_0)=\frac{\vol(G(\CO_v))}{\vol(H(\CO_v))^2}\alpha_v(\phi_0,\phi_0).
$$
By the unramified computation in \cite{H}:
$$
\alpha_v(\phi_0,\phi_0)=\vol(H(\CO_v)) \Delta_{n+1,v}\frac{L(1/2,\Pi_v)}{L(1,\pi_v,Ad)}.
$$
We thus obtain $$
J_{\pi_v}(f_v)=\frac{\vol(G(\CO_v))}{\vol(H(\CO_v))}\Delta_{n+1,v}\frac{L(1/2,\Pi_v)}{L(1,\pi_v,Ad)}.
$$
Note that $\frac{\vol(G(\CO_v))}{\vol(H(\CO_v))}$ is equal to the volume of the hyperspecial compact open of $U(V)(F_v)$, which is equal to $L(1,\eta)\Delta_{n+1,v}^{-1}.$
Therefore we obtain that 
\begin{align}\label{eqn Jpi f unr}
J_\pi(f)=L(1,\eta_v)\frac{L(1/2,\Pi_v)}{L(1,\pi_v,Ad)}.
\end{align}
By (\ref{eqn Ipi f unr}) (\ref{eqn Jpi f unr}), we have
\begin{align}
I_{\Pi_v}(f_v')=L(1,\eta_v)^{-1}J_{\pi_v}(f_v).
\end{align}
This completes the proof of case (i) of Theorem \ref{thm local}.

 {\bf Change of measures.}
 {\em From now on, all measures will be the unnormalized (namely, without the convergence factor $\zeta_v(1),L(1,\eta_v)$ etc.) Tamagawa measures with the natural invariant differential forms on the general linear groups, their subgroups, and Lie algebras. } 
 \begin{lem}
\label{lem local new} When using the unnormalized measures, the identity in Conjecture in \ref{conj local} becomes
 \begin{align}
\label{eqn local new}
 I_{\Pi_v}(f_v')= \kappa_v J_{\pi_v}(f_v),
\end{align}
for matching functions $f_v$ and $f_v'$ (also under the unnormalized measures).
 \end{lem}
 
 \begin{proof}
The old distribution $I_{\Pi_v}$ is the new one times 
 $$
 \zeta_{E,v}(1)^2 \frac{\zeta_{E,v}(1)\zeta_{F,v}(1)^2}{\zeta_{E,v}(1)^2},
 $$
 where the first term comes from the measure on $G'$ involving the definition of $\Pi_v(f_v')$, and the fraction comes from the measures in $\lambda_v,\beta_v$ and $\vartheta_v$. Similarly, the old distribution $J_{\pi_v}$ is the new one times 
 $$
L(1,\eta_v)^2\cdot L(1,\eta_v),$$
where the first term comes from the measure on $G$ involving the definition of $\pi_v(f_v)$, and the second from the measure on $H(F_v)$ in the definition of $\alpha_v$. Moreover, the change of measures on $H'_{1}(F_v),H_2'(F_v)$ and $H(F_v)$ also change the requirement of smooth matching: if $f_v$ and $f_v'$ match for the normalized measures, then $\zeta_{E,v}(1)\zeta_{F,v}(1)^2f_v$ and $L(1,\eta_v)^2f_v'$ match for the unnormalized measures.
Therefore, when using the unnormalized measures, the identity in Conjecture \ref{conj local} becomes the asserted one (\ref{eqn local new}).
\end{proof}
\subsection{Proof of Theorem \ref{thm local}: the case of a split place $v$.}
 Assume that $F=F_v$ is split. 
 Let $\pi=\pi_n\otimes\pi_{n+1}$ be an irreducible unitary generic representation of $G(F)$. We may identify
 $H_n(E)$ with $\GL_{n}(F)\times \GL_n(F)$ and  identify $U(W)(F_v)$ with a subgroup consisting of elements of the form $(g,^tg^{-1})$, $g\in \GL_n(F)$ and $^tg$ is the transpose of $g$. Let $p_1,p_2$ be the two isomorphisms between $U(W)(F)$ with $\GL_n(F)$ induced by the two projections from $\GL_{n}(F)\times \GL_n(F)$ to $\GL_n(F)$.  If $\pi_n$ is an irreducible generic representation of $U(W)(F_v)$, the representation $\Pi_n=BC(\pi_n)$ can be identified with $p_1^\ast \pi_v\otimes p_2^\ast \pi_v$ of $H_n(E)$ where $p_i^\ast \pi_v$ is a representation of $\GL_{n}(F)$ obtained by the isomorphism $p_i$. For simplicity, we will write $\Pi_n=\pi_n\otimes \wt{\pi}_n$. Similarly for $\pi_{n+1},\Pi_{n+1}$. We fix a Whittaker model $\CW(\pi_i),i=n,n+1$ using the additive character $\psi$ at $v$.  We define an auxiliary element 
  $\alpha'\in\Hom(\pi\otimes \wt{\pi},\BC)$ by
\begin{align}
\alpha'(W,W')= \lambda(W)\ov{\lambda(W')},\quad W,W'\in\CW(\pi).
  \end{align}
Again we have identified $\wt{\pi}$ with $\ov{\pi}$. Here we require that the invariant inner product on $\CW(\pi)$ is the one defined by $\vartheta$ (cf. (\ref{eqn defn vartheta})). Then we define a variant of the local spherical character:
\begin{align}
J'_{\pi}(f)=\sum_{W} \alpha'(\pi(f)W,W),\quad f\in \sC^\infty_c(G(F)),
\end{align}
where the sum of $W$ runs over an orthonormal basis of $\CW(\pi)$. Similarly we have a normalized one $J^{'\nat}_{\pi}(f)$.
 
 \begin{lem}
\label{lem split all} Let $f'=f_1\otimes f_2\in \sC_c^\infty(G'(F))$ and $f=f_1\ast f_2^\ast\in \sC_c^\infty(G(F)) $ ($f_2^\ast(x)=\ov{f}_2(x^{-1})$) matching $f'$. Then we have 
 \begin{align}\label{eqn split all}
 I_{\Pi_v}(f')=\kappa_v J'_{\pi}(f).
  \end{align}
 \end{lem}
 
 \begin{proof} We identify $\Pi_n$ with $\pi_n\otimes \wt{\pi}_n=\pi_n\otimes \ov{\pi}_n$. Then we have for $W,W'\in \CW(\pi_n)$:
 $$
 \beta_n(W\otimes \ov{W'})=\int_{N_{n-1}(F)\bs H_{n-1}(F)}W\left(\begin{array}{cc}\epsilon_{n-1}h & \\
 &  1
\end{array}\right)\ov{W'}\left(\begin{array}{cc}\epsilon_{n-1} h & \\
 &  1
\end{array}\right)\, dh.
$$ 
This yields 
 $$ \beta_n(W\otimes \ov{W'})=|\tau|^{d_n/2}\vartheta_n(W,W').
 $$
 Similarly for $\beta_{n+1}$.
 Then the desired equality follows by the definition of $I_{\Pi}$ in terms of the linear functional $\lambda,\beta$ and $\vartheta$ (note that $\delta=\tau^2$ is indeed a square in $F$).
\end{proof}
Now it remains to identify the distribution $J'_\pi$ with $J_{\pi}$, or equivalently,  to prove that $\alpha'=\alpha$.  The key ingredient is from \cite{LM}; in the non-archimedean case, we could also use \cite[\S3.5]{W3}.

 Note that we may write $\alpha$ in terms of the Whittaker model $\CW(\pi)$: $$
 \alpha(W,W')=\int_{H_n(F)}\pair{\pi(h)W,W'}\, dh,\quad \pair{W,W'}=\vartheta(W,W').
 $$
We temporarily denote $N_-=N_{n-}(F)$ and $N=N_n(F)$. Let $N^\circ=[N,N] $ be the commutator subgroup  of $N$, $N^{ab}=N^\circ\bs N$ the maximal abelian quotient of $N$, and $\wh{N^{ab}}$ the group of characters of $N^{ab}$. The diagonal subgroup $A_n$ of $H_n$ acts on $N$ (by conjugation), on $N^{ab}$ and hence on $\wh{N^{ab}}$. Moreover, $A_n$ acts transitively on the subset $\wh{N^{ab}}_{reg}$ of $\wh{N^{ab}}$ consisting of regular characters (i.e., with minimal stabilizer under the action of $A_n$). The character $\psi$ on $N$ is regular and we denote by $\psi_t$ the character of $N$ (equivalently, of $N^{ab}$) defined by $$\psi_t(u)=\psi(tu t^{-1}).$$
For $W_n,W'_n\in \CW(\pi_n,\ov{\psi})$, we denote by $\Phi_{W_n,W_n'}$ the matrix coefficient $$\Phi_{W_n,W_n'}(g)=\pair{\pi_n(g)W_n,W'_n}.
$$
 \begin{lem}  \label{lem Wald}   Assume that $\pi_n$ is tempered. 
   \begin{itemize}
  \item[(i)] The integral 
\begin{align}\label{eqn def CF}
\CF_{W_n,W'_n}(u):=\int_{N^\circ} \Phi_{W_n,W_n'}(vu)dv
 \end{align} 
  is absolutely convergent and defines a square integrable function $\CF_{W_n,W'_n}\in L^2(N^{ab})$. Its  Fourier transform 
 $ \wh{\CF}_{W_n,W'_n}\in L^2(\wh{N^{ab}})$ is smooth on the open subset $\wh{N^{ab}}_{reg}$ of $\wh{N^{ab}}$.
 \item[(ii)] For all $t\in A_n$, and $W_n,W'_n\in\CW(\pi_n,\ov{\psi})$, we have
\begin{align}
 \wh{\CF}_{W_n,W'_n}(\psi_t)= |\delta_n(t)|^{-1}W_n(t)\ov{W'}_n(t).
 \end{align} 
 Here the left hand side denotes the value of the Fourier transform at the character $\psi_t\in \wh{N^{ab}}_{reg}$.
  \end{itemize}
\end{lem}

\begin{proof}
The first part of (i) follows from \cite[Corollary 2.8]{LM}. The second part of (i) follows from \cite[Lemma 3.2]{LM} for a special class of $W_n$ and $W_n'$. The general case follows from this special case together with the Dixmier-Malliavin theorem (cf. \cite[Remark 3.3]{LM}). The assertion in (ii) for $t=1$ is \cite[Prop. 3.4]{LM}. The general case of $t\in A_n$ follows easily from this. 
\end{proof}
 
 \begin{prop}
\label{prop a=a'}
 Assume that $\pi=\pi_n\otimes\pi_{n+1}$ is tempered.  Then we have, for all $W,W'\in\CW(\pi_n,\ov{\psi})\otimes \CW(\pi_{n+1},\psi)$:
 $$
 \alpha(W,W')=\lambda(W)\ov{\lambda(W')}.
 $$
 Namely, $\alpha=\alpha'$ as non-zero elements in $\Hom(\pi\otimes \wt{\pi},\BC)$.
 \end{prop}
 \begin{proof}
 The right hand side does not vanish by the non-vanishing of local Rankin-Selberg integral (\cite{JPSS}, \cite{J}). By the multiplicity one theorem for generic representations: $\dim\Hom_{H_n(F)}(\pi,\BC)=1$, the left hand side is a constant multiple of the right hand side for all $W,W'$.  Hence it suffices to prove the identity for some choice of $W,W'$ so that $\lambda(W)\lambda(W')\neq 0$. 
 
 Let $W=W_n\otimes W_{n+1}, W'=W'_n\otimes W'_{n+1}$. We choose $W_{n+1},W_{n+1}'$ as follows.  
 Let $\varphi$ be in $\sC_c^\infty(B_-)$. Then there is a unique element  in $\CW(\pi_{n+1},\psi)$, denoted by $W_{\varphi}$,  such that the restriction $W_{\varphi}|_{H_n}$ is supported in $NB_-$ and
 $$
 W_{\varphi}\left(\begin{array}{cc}ub & \\
 & 1 
\end{array}\right)=\psi(u)\varphi(b),\quad u\in N, b\in B_-.
 $$
Similarly we choose $\varphi'\in \sC_c^\infty(B_-)$ and define $W_{\varphi'}\in \CW(\pi_{n+1},\psi)$.

We may and will consider the action of $ \sC_c^\infty(B_-)$ on $\CW(\pi_n,\ov{\psi})$ by
\begin{align}\label{eqn act pin}
\pi_n(\varphi)W(g)=\int_{B_-}W(g b)\varphi(b)db,
\end{align}
where $db$ is the right invariant measure on $B_-$ normalized so that the measure on $H_n$ decomposes as $dg=\,du\,db$ where $g=ub,u\in N, b\in B_-$.

Let $c\in \BR_+$ and consider the subset $N_c$ of $N$ consisting of elements $u=(u_{ij})_{1\leq i,j\leq n}$ such that
$$
|u_{i,i+1}|\leq c,\quad 1\leq i\leq n-1.
$$ 
We denote by $N^{ab}_c$ the image of $N_c$ in the abelian quotient $N^{ab}$. 
 Let us consider the integral parameterized by $t\in A_n$:
\begin{align}\label{eqn def Ic}
I_c(W,W';\psi_t):=\int_{B_-}\int_{B_-}\int_{N_c}\Phi_{W_n,W_{n}'}(b'^{-1}ub)\psi_t(u)W_{\varphi}(b)\ov{W_{\varphi'}}(b')\,du\,db\,db'.
\end{align}
This is the same as $$
I_c(W,W';\psi_t)=\int_{B_-}\int_{B_-}\int_{N_c}\Phi_{\pi_n(b)W_n,\pi_n(b')W_{n}'}(u)\psi_t(u)\varphi(b)\ov{\varphi'}(b')\,du\,db\,db'.
$$ 
We {\em claim} that the triple integral (\ref{eqn def Ic}) converges absolutely. Since $supp(\varphi)$ and $supp(\varphi')$ are compact, by \cite[Theorem 2]{CHH} and \cite[Theorem 1.2]{Sun}, there exists a constant $C$ such that, for all $b\in supp(\varphi),b'\in supp(\varphi')$, the matrix coefficients are bounded in terms of the Harish-Chandra spherical function $\Xi$ (cf. \cite{Sun})
$$
|\Phi_{\pi_n(b)W_n,\pi_n(b')W_{n}'}(g)|\leq C\cdot \Xi(g),\quad g\in H_n.
$$
Hence the triple integral $I_c$ is bounded above by 
$$
C\int_{B-}|\varphi(b)|\,db\int_{B-}|\varphi'(b')|\,db' \int_{N_c}\Xi(u)\,du.
$$
It suffices to prove that $\int_{N_c}\Xi(u)\,du$ is finite. We may write it as
\begin{align}\label{eqn Xi int}
\int_{N_c}\Xi(u)\,du=\int_{N^{ab}_c}\left(\int_{N^\circ}\Xi(vu)\,dv\right)\,du.
\end{align}
Since $\Xi$ is also a matrix coefficient of a tempered representation, the function $u\in N^{ab}\mapsto \int_{N^\circ}\Xi(vu)\,dv$ is in $L^2(N^{ab})$ by Lemma \ref{lem Wald} (or rather directly, \cite[Lemma 2.7]{LM}). Now the integral (\ref{eqn Xi int}) is finite since $N^{ab}_c$ is compact. This proves the claim.

For $\varphi\in \sC_c^\infty(B_-)$ and $t\in A_n$, we define $\varphi_t\in \sC_c^\infty(B_-)$ by $\varphi_t(b)=\varphi(t^{-1}b)$. 
For simplicity we denote 
$W_t=W_n\otimes W_{\varphi_t}$ and $W'_t=W'_n\otimes W_{\varphi'_t}$.
Then we have
\begin{align}\label{eqn pin var W(t)}
\pi_n(\varphi)W_n(t)=\int_{B-}W_n(tb)\varphi(b)db=|\delta_n(t)|\int_{B-}W_n(b)\varphi(t^{-1}b)db=|\delta_n(t)|\lambda(W_t).
\end{align}

We now study the integral $I_c$ as $c\to \infty$. We first substitute $u\mapsto t^{-1}ut$ in (\ref{eqn def Ic}): 
\begin{align*}
I_c(W,W';\psi_t)=|\delta_n(t)|^{-1}\int_{B_-}\int_{B_-}\int_{N_{c,t}}\Phi_{W_n,W_{n}'}((tb')^{-1}utb)\psi(u)W_{\varphi}(b)\ov{W_{\varphi'}}(b')\,du\,db\,db',
\end{align*}
where $N_{c,t}:=tN_ct^{-1}$.
Substitute $b\mapsto t^{-1}b$ and $b'\mapsto t^{-1}b'$:
\begin{align*}
I_c(W,W';\psi_t)=&|\delta_n(t)|\int_{B_-}\int_{B_-}\int_{N_{c,t}}\Phi_{W_n,W_{n}'}(b'^{-1}ub)\psi(u)W_{\varphi}(t^{-1}b)\ov{W_{\varphi'}}(t^{-1}b')\,du\,db\,db'
\\=&|\delta_n(t)|\int_{B_-}\int_{B_-}\int_{N_{c,t}}\Phi_{W_n,W_{n}'}(b'^{-1}ub)\psi(u)W_{\varphi_t}(b)\ov{W_{\varphi'_t}}(b')\,du\,db\,db'.
\end{align*}
Since the triple integral is absolutely convergent and $\psi(u)W_{\varphi_t}(b)=W_{\varphi_t}(ub)$, we could rewrite it by Fubini's theorem as
 $$
I_c(W,W';\psi_t)=|\delta_n(t)| \int_{B_-}\int_{N_{c,t} B_{-}}\Phi_{W_n,W_{n}'}(b'^{-1}g)W_{\varphi_t}(g)\ov{W_{\varphi'_t}}(b')\,dg\,db'.
$$ 
Now we make a substitution $g\mapsto b'g$ and then interchange the order of integration:
\begin{align*}
I_c(W,W';\psi_t)=&|\delta_n(t)|\int_{B_-}\int_{b'^{-1} N_{c,t} B_-}\Phi_{W_n,W_{n}'}(g)W_{\varphi_t}(b'g)\ov{W_{\varphi'_t}}(b')\,dg\,db'
\\=&|\delta_n(t)|\int_{B_{-} N_{c,t} B_-}\Phi_{W_n,W_{n}'}(g)\left(\int_{B-} W_{\varphi_t}(b'g)\ov{W_{\varphi'_t}}(b')\,db' \right)\,dg
\\=&|\delta_n(t)|\int_{B_{-} N_{c,t}B_-}\Phi_{W_n,W_{n}'}(g)\Phi_{W_{\varphi_t},W_{\varphi'_t}}(g)\,dg.
\end{align*} 
Since the integral
$$
\alpha(W_t,W'_t)=\int_{H_{n}(F)}\Phi_{W_n,W_{n}'}(g)\Phi_{W_{\varphi_t},W_{\varphi'_t}}(g)\,dg
$$
converges absolutely and $H_n\setminus \bigcup_{c\to \infty}B_{-} N_{c,t}B_-$ is of measure zero, we conclude that for all $t\in A_n$, the limit $\lim_{c\to \infty}I_c(W,W';\psi_t)$ exists and is given by:
\begin{align}\label{eqn limit Ic alt}
\lim_{c\to \infty}I_c(W,W';\psi_t)=|\delta_n(t)|\alpha(W_t,W'_t).
\end{align}

There is another way to evaluate the limit. We first interchange the order of integration in (\ref{eqn def Ic}) and rewrite it as
\begin{align}\label{eqn int Ic1}
I_c(W,W';\psi_t)=\int_{N_c}\Phi_{\pi_n(\varphi)W_n,\pi_n(\varphi')W_{n}'}(u)\psi_t(u)\,du.
\end{align}
This integral is the same as (cf. (\ref{eqn def CF}))
\begin{align}
I_c(W,W';\psi_t)=\int_{N^{ab}_c}\CF_{\pi_n(\varphi)W_n,\pi_n(\varphi')W_{n}'}(u)\psi_t(u)\,du.
\end{align} Note that $\CF_{\pi_n(\varphi)W_n,\pi_n(\varphi')W_{n}'}\in L^2(N^{ab})$ by Lemma \ref{lem Wald} (i). We now view $I_c(W,W';\cdot)$ as a function of $\psi_t\in \wh{N^{ab}}$. It follows that $\lim_{c\to \infty}I_c(W,W';\cdot)$ converges in $L^2( \wh{N^{ab}})$ to $\wh{\CF}_{\pi_n(\varphi)W_n,\pi_n(\varphi')W_{n}'}$.
But we have proved that $\lim_{c\to \infty}I_c(W,W';\cdot)$ converges pointwise (for regular characters) almost everywhere. Therefore, for almost all (i.e., except a measure zero set) $t\in A_n$, the point wise limit is the same as the Fourier transform (cf. \cite[Theorem 1.1.11]{Gra}):
$$
\lim_{c\to \infty}I_c(W,W';\psi_t)=\wh{\CF}_{\pi_n(\varphi)W_n,\pi_n(\varphi')W_{n}'}(\psi_t).
$$
By (ii) of Lemma \ref{lem Wald},  the right hand side is equal to 
\begin{align}\label{eqn lim Ic W(t)}
|\delta(t)|^{-1}\pi_n(\varphi)W_n(t)\ov{\pi_n(\varphi')W_{n}'}(t)=|\delta_n(t)|\lambda(W_t)\ov{\lambda(W_t')},
\end{align}
where the equality follows from (\ref{eqn pin var W(t)}). Therefore we have for almost all $t\in A_n$
\begin{align}\label{eqn lim Ic Wt}
\lim_{c\to \infty}I_c(W,W';\psi_t)=|\delta_n(t)|\lambda(W_t)\ov{\lambda(W_t')}.
\end{align}
 
Comparing (\ref{eqn lim Ic Wt}) with (\ref{eqn limit Ic alt}), we have for almost all $t\in A_n$:
\begin{align}\label{eqn alp=lam}\alpha(W_t,W'_t)=
\lambda(W_t)\ov{\lambda(W_t')}.
\end{align} 
In particular, in any small open neighborhood of $1$ in $A_n$, there exists $t$ so that the equality (\ref{eqn alp=lam}) holds.

Finally, it remains to verify that for some choice of $W_n,W_n'$ and $\varphi,\varphi'$, the local period $\lambda(W_t)\ov{\lambda(W'_t)}$ does not vanish for $t$ in a small  open neighborhood of $1$ in $A_n$. We choose $W_n,W_n'$ such that $W_{n}(1)\neq 0,W_n'(1)\neq 0$. Since $W_n|_{B-}$ is a continuous function, there exists $\varphi\in \sC_c^\infty(B_-)$ so that 
$$
\int_{B_-}W_n(b)\varphi(b)db\neq 0.
$$ It is easy to see that $t\mapsto \int_{B_-}W_n(b)\varphi(tb)db$ is continuous. Hence for $t$ in a small open neighborhood of $1$ in $A_n$, the integral $\int_{B_-}W_n(b)\varphi(tb)db\neq 0$, or equivalently, by (\ref{eqn pin var W(t)}), $\lambda(W_t)\neq 0$ for $W_t=W_n\otimes W_{\varphi_t}$. Similarly we may achieve $\lambda(W'_t)\neq 0$ for $t$ in a small open neighborhood of $1$. This completes the proof.
 \end{proof}
 \begin{cor}
\label{cor thm local}
The case (1) (i.e., for a split $v$) of Theorem \ref{thm local} holds.
 \end{cor}
 \begin{proof}
 This follows from Lemma \ref{lem split all} and Prop. \ref{prop a=a'}.
 \end{proof}

 \section{The totally definite case}
We now prove the part $(2)$ of Theorem \ref{thm main}, assuming Theorem  \ref{thm local}. We hence assume that $G(F_\infty)$ is compact. Equivalently, $F$ is a totally real field,  $E$ is a CM extension and the Hermitian spaces $W,V$ are positive definite at every archimedean place $v$ of $F$. As in the proof of the part $(1)$ of Theorem \ref{thm main}, we may assume that the local invariant form  $\alpha_v\neq 0$ for all $v$. We may further assume that the global period $\sP\neq 0$ so that the global spherical character $J_\pi$ does not vanish (otherwise $\sL(1/2,\pi)=0$ and the result holds trivially). We then have  
\begin{align}\label{eqn 5.1}
J_\pi(f)=\sC_\pi \prod_v J^\nat_{\pi_v}(f_v)
\end{align}
for a non-zero constant $\sC_\pi$.

Let us recall that in the disjoint union of (\ref{eqn orbit match}) we take all isomorphism classes of n-dimensional Hermitian spaces $W_v$. When $F_v\simeq \BR$ is non-archimedean and $E_v\simeq \BC $, the isomorphism classes of such $W_v$  are indexed by the signature $(p,q)$ of $W_v$. We denote them by $W_{(p,q),v}$. Then the two definite (positive or negative) spaces correspond to $(p,q)=(n,0), (0,n)$. Only when $W_v=W_{(n,0),v}$ is the positive definite one, the space $V_v$ is also positive definite (by (\ref{eqn V=W+e})), or equivalently the group $G_{W_v}(F_v)$ is compact. Let $G'(F_v)_{rs, (n,0)}$ be the open subset of the regular semisimple locus $G'(F_v)_{rs}$ corresponding to the positive definite one $G_{W_{(n,0),v}}(F_v)_{rs}$ in the disjoint union (\ref{eqn orbit match}). In our case, our $G(F_v)$ is isomorphic to   $G_{W_{(n,0),v}}(F_v)$ for all $v|\infty$.

 For every $v|\infty$, we now choose a test function $f_v$ supported in the regular semisimple locus $G_{W_{(n,0),v}}(F_v)_{rs}$. Then there exists a smooth transfer $f'_v$ supported in $G'(F_v)_{rs, (n,0)}$.  Since the representation $\pi_v$ must be finite dimensional and we are assuming that $\alpha_v\neq 0$, our choice of $f_v$ can be made so that $J^\nat_{\pi_v}(f_v)\neq0$.
 
For non-archimedean places $v$, we choose $f_v$ and its smooth transfer $f_v'$ as in the proof of case (1). Particularly, $J^\nat_{\pi_v}(f_v)\neq 0$ for all non-archimedean $v$. Then
for such test functions $f=\otimes f_v$ and $f'=\otimes f'_v$, we again have, by Theorem \ref{thm tr id}
 $$
J_\pi(f)=2^{-2}L(1,\eta)^{-2}I_{\pi_E}(f').
$$
By Prop. \ref{prop char GL}, the right hand side is equal to
$$c\cdot  \sL(1/2,\pi)\prod_v I^\nat_{\pi_{E,v}}(f'_v),
$$
for some constant $c$ independent of $\pi$. Now fix an arbitrary $v_0|\infty$ and we further assume that $J_{\pi_v}^\nat(f_v)\neq 0$ for $v\neq v_0$. By comparison with (\ref{eqn 5.1}),  there exists a constant $b_{\pi_{v_0}}\neq 0$, such that for all $f_{v_0}$ with regular semisimple support and its smooth transfer $f_{v_0}'$ supported in $G'(F_{v_0})_{rs, (n,0)}$, we have
$$
I_{\Pi_{v_0}}(f'_{v_0})=b_{\pi_{v_0}}J_{\pi_{v_0}}(f_{v_0}).
$$

Now we don't know how to evaluate $b_{\pi_{v}}$ for $v|\infty$.  Nevertheless the same argument as the proof of Prop. \ref{prop imp} shows that 
\begin{align*}\frac{|\sP(\phi)|^2}{\langle
\phi,\phi \rangle_{Pet}}=c_{\pi_\infty}2^{-2}\sL\left(\frac{1}{2},\pi\right)
\prod_{v}\frac{\alpha^{\nat}_v(\phi_v,\phi_v)}{\pair{\phi_v,\phi_v}_v},
\end{align*}
where the constant $c_{\pi_\infty}=\prod_{v|\infty}c_{\pi_v}$ and
$$
c_{\pi_v}=b_{\pi_v}\kappa_v^{-1}L(1,\eta_v),
$$
where $\kappa_v$ is the constant in Conjecture \ref{conj local}.

\part{Local theory}

In the rest of the paper, we prove the remaining parts of Theorem \ref{thm local}.

\section{Harmonic analysis on Lie algebra}
We establish some basic results to prove the identity between local characters, Theorem \ref{thm local}  for a non-split non-archimedean place $v$.

\subsection{Relative regular nilpotent elements in $M_{n+1}$.}Let $F$ be any field.
The group $H_n$, viewed as a subgroup of $H_{n+1}$, acts on $M_{n+1}$ by conjugation.
Write $$
X=\left(\begin{array}{cc}A & u\\
v &  w
\end{array}\right)\in M_{n+1}.
$$
The ring of invariants for this action is freely generated by either
\begin{align}
 (-1)^{i-1}\tr \wedge ^i X,\quad e_{n+1} X^je_{n+1}^\ast,\quad 1\leq i\leq n+1, 1\leq j\leq n;
\end{align}
or 
\begin{align}\label{gen 2}
 (-1)^{i-1}\tr \wedge ^i A,\quad vA^ju,\quad w,\quad 1\leq i\leq n, 0\leq j\leq n-1.
\end{align}
We define a matrix
\begin{align}
\delta_+(X):=(A^{n-1}u,A^{n-2}u,...,u)\in M_{n}(F),
\end{align}
and its determinant
\begin{align}
\Delta_+(X)=\det(\delta_+(X)).
\end{align}
Similarly, we define
$$
\delta_-(X):=(v,vA,...,vA^{n-1})\in M_{n}(F),\quad \Delta_-(X)=\det(\delta_-(X)),
$$
and
$$\quad \Delta:=\Delta_+\Delta_-.
$$
Clearly we have for $X\in M_{n+1}(F)$ and $ h\in \GL_{n}(F)$:
\begin{align}
\label{eqn delta+} \delta_+(hXh^{-1})=h\delta_+(X), \quad \delta_-(hXh^{-1})=\delta_-(X)h^{-1}.
\end{align}

The $H_n$-nilpotent cone $\CN$ is defined to be the zeros of all of the above invariant functions on $M_{n+1}$. An element in $M_{n+1}$ is called {\em $H_n$-regular} (or {\em regular} if no confusion arises) if its stabilizer is trivial. 
Denote
\begin{align}
\xi_{n+1,+}=
\left(\begin{array}{cccc}0  &1 &0 & 0\\
 0 &  0&1&0\\ 0&...&0 &1
\\
0&...&0&0\end{array}\right)\in M_{n+1}(F),
\end{align}
and $\xi_{n+1,-}$ its transpose. If no confusion arises, we simply denote them by $\xi_\pm$. Clearly $\xi_\pm$ are regular nilpotent.

\begin{lem}\label{lem 6.1}
Let $X\in \CN$. The following statements are equivalent:
\begin{enumerate}
\item $X$ is regular nilpotent.
\item
$\xi$ is $H_n$-equivalent to $\xi_{+}$ or $\xi_{-}$.
\item $\Delta_+(X)\neq 0$ or $\Delta_-(X)\neq 0$.
\end{enumerate}
In particular, the orbit of $\xi_{+}$ is open in $\CN$.\end{lem}

\begin{proof} 

\label{proof noncompete}
Let $X=\left(\begin{array}{cc}A & u\\
v &  0
\end{array}\right)
$ be an $H_n$-nilpotent element. We then have $A^n=0$ and $vA^iu=0$, for $i=0,1,...,n-1$. It follows that $vA^iu=0$ for all $i\in \BZ_{\geq 0}$. Let $r$ be the dimension of subspace of the ($n\times 1$) column vectors spanned by $A^iu,i=0,1,...,n-1$, and similarly $r'$ the dimension of the subspace spanned by $vA^i,i=0,1,...,n-1$. Clearly, we have an inequality $r+r'\leq n$.

 $1\Rightarrow 2$. It suffices to show that, if $X$ is regular unipotent, then either $r$ or $r'$ is equal to $n$. Indeed, for example if $r=n$, then $r'=0$ (i.e., $v=0$) and the column vectors $A^iu,i=0,1,...,n-1$ form a basis of the n-dimension space of column vectors. Then $\{e^\ast,X e^\ast,...,X^{n-1}e^\ast\}$ form a basis of the $(n+1)$-dimensional column vectors (recall that $e^\ast$ is the transpose of $e=(0,...,0,1)\in M_{1,n+1}(F)$). In term of this new basis we see that $X$ becomes $\xi_+$. 

Now suppose that $r,r'<n$.  Clearly if $r=r'=0$, $X$ must have positive dimensional stabilizer hence non-regular. We now assume that $0<r<n$. Let $L$ be the subspace spanned by $A^iu,i=0,1,...,r-1$. It is easy to see that this is the same as the space spanned by $A^iu,i=0,1,...,n-1$. We write the column vector spaces $F^n=L\oplus L'$ for a subspace $L'$. Then in term of the basis of $L$ given by $A^iu,i=0,1,...,r-1$, we may write $u$ as $(0,0,...,1,0,0,...,0)^t$ where only the $r$-th entry is nonzero and may be assumed to be equal to one, and
$$
A=\left(\begin{array}{cc}Y & B\\
0 &  Z
\end{array}\right),\quad Y= \left(\begin{array}{cccc}0  &1 &0 & 0\\
 0 &  0&1&0\\ 0&...&0 &1
\\
0&...&0&0\end{array}\right) \in M_{r}(F).
$$
Then $A^iu=(0,0,...,1,0,0,...,0)^t$ where only the $(r-i)$-th entry is one, $i=0,1,...,r-1$. Hence the conditions $vA^iu=0$ ($0\leq i\leq n-1$) implies that $v$ is of the form $(0,0,...,0,\ast,...,\ast)$ where the first $r$-entries are all zero.

Now we consider 
$$
h=\left(\begin{array}{cc}1_{r}& Q\\
0 &  1_{n-r}
\end{array}\right)\in \GL_n(F).
$$
Clearly the matrix $h^{-1}Xh$ to an element of the same form with $B$ replaced by $B+YQ-QZ$. Hence the stabilizer of $X$ at least contains all $h$ with $Q$ satisfying
$YQ-QZ=0$. Define $\varphi\in \End(M_{r,n-r}(F))$ by $Q\mapsto YQ-QZ $. We claim that the dimension of the kernel $\Ker(\varphi)$ is positive. Clearly the dimension of $\Ker(\varphi)$ depends only on the conjugacy class of $Z$ in $M_{n-r}(F)$.  Since $A$ is nilpotent, so is $Z$. We thus can assume that $Z$ is a Jordan canonical form.  The endomorphism $\varphi$ cannot be surjective since $M_{r,n-r}(F)\neq 0$ ($0<r<n$) and every $YQ-QZ$ must have zero as its lower left entry. This proves the claim and hence the stabilizer of such $X$ cannot be trivial.

$2\Rightarrow 3$. This is clear since we have $\delta_+(hXh^{-1})=h\delta_+(X)$ and  $\delta_-(hXh^{-1})=\delta_-(X)h^{-1}$ by (\ref{eqn delta+}).

$3\Rightarrow1$. Note that $\Delta_+(X)\neq 0$ is equivalent to $\delta_+(X)\in \GL_n(F)$. The latter property implies that the stabilizer (under the $H_n$-action) of any $X\in M_{n+1,+}$ must be trivial. Indeed, if $hXh^{-1}=X$, we have $\delta_+(X)=\delta_+(hXh^{-1})=h\delta_+(X)$ and hence $h=1$. Similarly for $\Delta_-(X)\neq 0$.
\end{proof}

\subsection{A regular section.} 
Denote 
$$
M_{n+1,+}:=\{X\in M_{n+1}|\Delta_+(X)\neq 0.\}
$$
Note that every element in $M_{n+1,+}$ is regular (cf. the proof of ``$3\Rightarrow1$" of Lemma \ref{lem 6.1}).
We shall write $ \sX=\BA^{n}\times \BA^{n+1}$, \footnote{We use $\BA$ in this section only to denote the affine line.} the affine space of dimension $2n+1$. Then the second set of generators (\ref{gen 2}) defines a morphism that is constant on $H_n$-orbits
\begin{align*}
\pi:M_{n+1}&\lra \sX=\BA^n\times \BA^{n+1},\\
\left(\begin{array}{cc}A & u\\
v &  w
\end{array}\right)&\mapsto(a,b),
\end{align*}
where $a=(a_1,...,a_n), b=(b_0,...,b_n)$, $a_{i}=(-1)^{i-1}\tr\wedge ^{i}A$, $b_{0}=w$ and $b_i=vA^{i-1}u$ for $1\leq i\leq n$. We say that $x\in \sX$ is {\em regular semisimple} if one element (and hence all) in $\pi^{-1}(x)$ is $H_n$-regular semisimple.

Now we define a section of the morphism $\pi:M_{n+1}\to \sX$:
\begin{align*}
\sigma:  \sX&\lra M_{n+1}\\
(a,b)&\mapsto \left(
\begin{array}{ccccc}a_1 &1 &0&0 &0 \\
a_2 &   0&1&0&0\\ ...&0&0&1 &0
\\ a_n&0&0&0&1\\
b_n&...&...&b_1&b_0\end{array}\right).
\end{align*} We note that $\xi_+$ is precisely the image of $0\in \sX$ under $\sigma$.

\[
\xymatrix@=40pt@C=30pt{
& M_{n+1} \ar@{->}[d] ^{\pi}\\
& \sX=M_{n+1//H_n} \ar@{->}@/^2pc/[u] ^{\sigma}
}
\]

\begin{lem}
The morphism $\sigma$ is a section of $\pi$, i.e.:
$$
\sigma\circ \pi=id.
$$
The image of $\sigma$ lies in $M_{n+1,+}$ (in particular, $\sigma$ is a {\em regular} section, in the sense that the image $\sigma(a,b)$ is always $H_n$-regular).
\end{lem}

\begin{proof}
It is easy to check that 
$$
\det\left( T\cdot 1_{n} +\left(
\begin{array}{cccc}a_1 &1 &0&0 \\
a_2 &   0&1&0\\ ...&0&0&1 
\\ a_n&0&0&0\end{array}\right)\right)=T^n+a_1T^{n-1}-a_2T^{n-2}+...+(-1)^{n-1}a_n,
$$
and the $b$-invariants of $\sigma(a,b)$ are $(b_0,b_1,b_2,...,b_n)$. This shows that $\sigma$ is a section of $\pi$.
To see that the image of $\sigma$ lies in $M_{n+1,+}$, we note that for any $(a,b)\in \sX$ we have
\begin{align}\label{eqn delta+=1}
\delta_+(\sigma(a,b))=1_n.
\end{align}

\end{proof}

\begin{prop}
\label{prop sec}
We have an $H_n$-equivariant morphism \begin{align*}
\iota:\GL_n\times \sX &\to M_{n+1,+}\\
(h, (a,b))&\mapsto h\sigma(a,b)h^{-1},\end{align*}
where the group $H_n$ acts on the left hand side by left translation on the first factor, and trivially on the second factor.
Moreover, the morphism $\iota$ is an isomorphism with its inverse given by $(\delta_{+},\pi|_{M_{n+1,+}})$.
\end{prop}

\begin{proof}
It suffices to prove that $ (\delta_{+},\pi)\circ \iota=id$ and $\iota\circ (\delta_{+},\pi)=id$. To show the first identity we note that the invariants of $h\sigma(a,b)h^{-1} $ (being the same as $\sigma(a,b)$) are $(a,b)$. Hence it is enough to show that $\delta_+(\iota(h,(a,b))=h$. This follows from the fact that $\delta_+(\sigma(a,b))=1_n$ (cf. (\ref{eqn delta+=1}))and $\delta_+(hXh^{-1})=h\delta_+(X)$ by (\ref{eqn delta+}). 

Now we show the second identity. Let $X=\left(\begin{array}{cc}A & u\\
v &  w
\end{array}\right)\in M_{n+1,+}$. Let $(a,b)=\pi(X)$ and $h=\delta_+(X)=(A^{n-1}u,A^{n-2}u,...,u)\in H_n$. Denote $\iota\circ (\delta_{+},\pi)(X)=\left(\begin{array}{cc}A' & u'\\
v' &  w'
\end{array}\right)$. Clearly $w=w'$. By the first identity, the elements $\iota\circ (\delta_{+},\pi)(X)$ and $X$ have the same invariants. In particular, $$\det(T\cdot 1_n+A)=T^n+\sum_{i=1}^n(-1)^{i-1}a_iT^{n-i},$$ 
and therefore
$$
A^n=\sum_{i=1}^n a_iA^{n-i}.
$$
This implies that
$$
A\delta_+(X)=(A^nu,A^{n-1}u,...,Au)=\delta_+(X)\left(
\begin{array}{cccc}a_1 &1 &0&0 \\
a_2 &   0&1&0\\ ...&0&0&1 
\\ a_n&0&0&0\end{array}\right) .
$$
Since $\delta_+(X)=h$ is invertible, we obtain $$ A=h\left(
\begin{array}{cccc}a_1 &1 &0&0 \\
a_2 &   0&1&0\\ ...&0&0&1 
\\ a_n&0&0&0\end{array}\right)h^{-1}=A'.
$$
Obviously we have $u= \delta_+(X)e_n^\ast=he_n^\ast=u'$ ($e_n^\ast=(0,0,...,0,1)^t$). Finally since $b_i=vA^{i-1}u$, $(b_n,b_{n-1},...,b_1)=v\delta_+(X)$, we have $$v=(b_n,b_{n-1},...,b_1)\delta_+(X)^{-1}=(b_n,b_{n-1},...,b_1)h^{-1}=v'.$$ This completes the proof of the second identity.
\end{proof}
Similarly we define a variant $\sigma':\sX=\BA^n\times \BA^{n+1}\to M_{n+1}$ by
\begin{align}
\sigma'(a,b)=\left(
\begin{array}{ccccc}0 &1&   0 & \cdots&0\\ 0&0& 1&\cdots&0\\ \vdots&\vdots&\vdots& \vdots&\vdots\\
a_n &a_{n-1} &\cdots&a_1& 1\\
b_n &b_{n-1} &\cdots&b_1& b_0
\end{array}\right).
\end{align}

To facilitate the exposition, we introduce:
 \begin{defn}
Consider a morphism between two affine spaces:
 $$\phi:\BA^m=\Spec F[x_1,...,x_m]\longrightarrow \BA^m=\Spec F[y_1,...,y_m]$$ with induced morphism $\phi^\ast: F[y_1,...,y_m]\lra F[x_1,...,x_m]$. We say that $\phi$ is {\em triangular} if we have, possibly after reordering the coordinates:
 $$\phi^\ast(y_i)=\pm x_i+\varphi_i(x_1,...,x_{i-1}),\quad1\leq i\leq m,
 $$ where $\varphi_i(x_1,...,x_{i-1})\in F[x_1,...,x_{i-1}]$ is a polynomial of $x_1,...,x_{i-1}$.  \end{defn}
It is easy to see that if $\phi$ is triangular, then it is an isomorphism and its inverse is triangular, too. Moreover, the Jacobian factor of a triangular morphism is equal to $\pm 1$. 

\begin{cor}\label{cor reg sec}
The following morphism is an isomorphism:
\begin{align*}
\iota':H_n\times \sX &\to M_{n+1,+}\\
(h, (a,b))&\mapsto h\sigma'(a,b)h^{-1}.
\end{align*}
Moreover, the induced morphism $\pi\circ \sigma': \sX\to \sX$ is triangular, and in particular an isomorphism.
\end{cor}

\begin{proof}  The proof of the first part follows the same line as the previous one: it suffices to show that for an arbitrarily $X=\left(\begin{array}{cc}A & u\\
v &  w
\end{array}\right)\in M_{n+1,+}$, we may solve for $(a,b)$ and $h$ uniquely in terms of polynomials of entries of $X$:
\begin{align}\label{eqn sigma'}
h\sigma'(a,b)h^{-1}=X.
\end{align}
We proceed in three steps:
\begin{enumerate}\item[]{\em Step 1.} For the $a$-component of $\sigma'(a,b)$, we have $a_i=(-1)^{i-1}\tr \wedge^i A$.  
\item[]{\em Step 2.}
By (\ref{eqn delta+}), we have $\delta_+(X)=h\delta_+(\sigma'(a,b))$.
Note that the matrix $\delta_+(\sigma'(a,b))$ lies in $N_{n,-}$ and it depends only on $a$ (but not on $b$).  Combined with {\em Step 1}, we see that can be expressed in terms of $X$:  \begin{align}\label{eqn solve h}
h=\delta_+(X)\delta_+(\sigma'(a,b))^{-1}.
\end{align}
\item[]{\em Step 3.} In (\ref{eqn sigma'}), the last row of $\sigma'(a,b)$, i.e. $(b_n,...,b_1,b_0)$, is equal to $(vh,w)$. Combining with {Step 2} we complete the proof. 
\end{enumerate}

To show the second part, by {\em Step 1} we may write $\pi\circ \sigma'(a,b')=(a,b)$. By computing the $b$-invariants of $\sigma'(a,b')$, we say that $b_0=b_0'$, and for each $i\geq 1$, $b_i-b_i'$ is a polynomial of $a,b'_1,...,b'_{i-1}$. This also shows that $b_i$ is a polynomial of  $a,b'_1,...,b'_{i-1}$. Therefore $\pi\circ \sigma':\sX\to \sX$ is a triangular morphism. 
\end{proof}

We will need to consider the restriction of $\iota'$ to some closed subvarieties. We denote by $\CW$ the subvariety of of $M_{n+1}$ consisting of matrices $X$ of the following form:
\begin{align}
X=\left(
\begin{array}{ccccc}\ast&1&0&0&0\\\ast& \ast&1&0&0\\
\ast& \ast& \ast&1&0\\
\ast& \ast& \ast&\ast &1\\
\ast & \ast &   \ast & \ast&\ast
\end{array}\right)\in M_{n+1}.
\end{align}
Denote by $\CV$ the subvariety of $\CW$ consisting of $X$ of the same form but with the last row identically zero.
Then we have a natural projection $p: \CW\to \CV$.

\begin{lem}\label{lem a sect}
\begin{enumerate}
\item[(1)]
The variety $\CW$ is a subvariety of $M_{n+1,+}$ and the preimgae of $\CW$ under $\iota'$ is the product $N_{n,-}\times \sX$. 
\item[(2)]For every $(a,b)\in \sX$,
we define a morphism $\nu_{(a,b)}: N_{n,-}\lra \CW$ by
\begin{align}
\nu_{(a,b)}(u)=u\sigma'(a,b)u^{-1},\quad u\in N_{n,-}.
\end{align}  Then the composition $\nu'_{(a,b)}:=p\circ \nu_{(a.b)}:N_{n,-}\longrightarrow \CV$ is an isomorphism with Jacobian equal to $\pm 1$.
\end{enumerate}

\end{lem}
\begin{proof}Let $X=\left(\begin{array}{cc}A & u\\
v &  w
\end{array}\right)$ be in $\CW$. It is easy to verify the following properties about $\delta_+(X)$:
\begin{itemize}
\item[(i)]
$\delta_+(X)\in N_{n,-}$.
\item[(ii)] $\delta_+(X)$ depends only on the last $n-1$ columns of $A$, but not on the first column.
\item[(iii)] For each $i$, $0\leq i\leq n-1$, the $(n-i)$-th column of $\delta_+(X)$ is equal to the sum of the $(n-i+1)$-th column of $A$ plus a column vector whose entries are polynomials depending only the last $(i-1)$-columns of $A$.
\end{itemize} 
By (i), such $X$ lies in $M_{n+1,+}$ and hence $\CW\subset M_{n+1}$. Setting $X=\sigma'(a,b)$ shows that $\delta_+(\sigma'(a,b))$ lies in  $N_{n,-}$ and depends only on $a$. Let $(h,(a,b))$ be the preimage $\iota'^{-1}(X)$. By (\ref{eqn solve h}) in the proof of Prop. \ref{prop sec}, we have 
$$
h=\delta_+(X)\delta_+(\sigma'(a,b))^{-1}\in N_{n,-}.
$$ 
Hence the preimage of $\CW$ is contained in $ N_{n,-}\times \sX$. Since $\CW$ is preserved under the conjugation by $N_{n,-}$ and contains the image of $\sigma$, it follows that the preimage of $\CW$ is exactly $N_{n,-}\times \sX$. This proves part (1) of the lemma. Alternatively, we may identify $\CW$ with the variety consisting of $X\in M_{n+1}$ such that $\delta_+(X)\in N_{n,-}$.

To show part (2),
we denote by $\CW_{(a,b)}$ the image of $N_{n,-}\times \{(a,b)\}$ under $\iota'$.
Consider an auxiliary subvariety $\CV'$ of $\CW$ with the first column and the last row both being zero. Let $p': \CW\to \CV$ be the natural projection. By the property (iii) above, the composition $p'\circ \nu:N_{n,-}\to \CV'$ is a triangular morphism. Thus the restriction of the projection $p'$ to $\CW_{(a,b)}$ induces an isomorphism $p'_{(a,b)}:\CW_{(a,b)}\to \CV'$. To prove part (2), it remains to show that the morphism $p|_{\CW_{(a,b)}}\circ (p_{(a,b)}')^{-1}:\CV'\to \CV$ is triangular.
  
Now we let $X=\left(\begin{array}{cc}A & u\\
v &  w
\end{array}\right)\in \CW_{(a,b)}$ with
$$
A=\left(
\begin{array}{ccccc}\alpha_1&1&0&0&0\\\alpha_2& \ast&1&0&0\\
\ast & \ast& \ast&1&0\\
\ast& \ast& \ast&\ast &1\\
\alpha_n=\beta_n& \beta_{n-1} &   \ast & \ast&\beta_1
\end{array}\right)\in M_n.
$$
We denote by $\wt{A}$ the square matrix obtained by deleting the first and the last row/column of $A$. By computing the coefficients of the characteristic polynomial of $A$, we have the following:

{\em $(\spade)$ For each $i$, $1\leq i\leq n$, the sum $\alpha_i+\beta_{i}$ is a polynomial of $\alpha_{1},...,\alpha_{i-1}$, $\beta_1,...,\beta_{i-1}$, $a_1,...,a_{i}$ and the entries of $\wt{A}$.}

By induction on $i$, $\alpha_i$ is a polynomial of $\beta_1,...,\beta_{i}$, $a_1,...,a_{i}$ and the entries of $\wt{A}$. The same statement holds if we replace $\alpha$ by $\beta$ everywhere. Note that the last row of $X\in \CW_{(a,b)}$ is also determined by the entries $\alpha_1,...,\alpha_n,a_1,...,a_n$ and $\wt{A}$.
It follows that the morphism $p|_{\CW_{(a,b)}}\circ (p_{(a,b)}')^{-1}:\CV'\to \CV$ is triangular.  This completes the proof.

\end{proof}

A byproduct of the proof is:
\begin{cor}\label{cor bdd}
Let $X=\left(\begin{array}{cc}A & u\\
v &  w
\end{array}\right)\in \CW$. Then every entry of the last row of $A$ is a polynomial of the first $n-1$ rows of $A$ and the coefficients of the characteristic polynomial of $A$.
\end{cor}

\begin{proof}
In the proof of the previous Lemma \ref{lem a sect}, the $\beta_i$'s are polynomials of the first $n-1$ rows of $A$, and the coefficients $a_i$'s of the characteristic polynomial of $A$.
\end{proof}

\label{combine with last lem?}

We also have an easier statement about the upper unipotent $N_{n}$ acting on $\xi_{n+1,+}$.
\begin{lem}
\label{lem N+}
Denote by $\CV_+$ the subvariety of $M_{n+1}$ consisting of $X$ of the following form:
$$
X=\left(
\begin{array}{cccc}0  &1 &\ast & \ast \\
0 &   0&1&\ast\\ ...&...&0 &1
\\
0&...&0&0\end{array}\right)\in M_{n+1}.
$$ Define a morphism 
\begin{align*}
\nu_+:N_n&\to \CV_+\\
u&\mapsto u \xi_{n+1,+ }u^{-1}.
\end{align*} Then $\nu_+$ is triangular. 
\end{lem}

\begin{proof}
Similar to the previous one. We omit the detail.
\end{proof}

For later use in \S 8, we take the transpose of the morphism $\sigma'$ and denote it by $\varrho$:
\begin{align}
\label{eqn varrho}
\varrho(a,b)=\sigma'(a,b)^t.
\end{align}

\subsection{Regular nilpotent orbital integral.} We now assume that $F$ is a {\em p-adic local field.}
We now define the $(H_n,\eta)$-orbital integral of a regular nilpotent orbit.
Since the orbits of $\xi_\pm$ are not closed, we need to regularize the orbital integral. We consider the following integral
for $s\in \BC$, $X\in M_{n+1}(F)$:
\begin{align}\label{def O(X,f,s)}
O(X,f,s)=\int_{H_n(F)} f(X^h)\eta(h)|\det(h)|^s\, dh,\quad f\in \sC^\infty_c(M_{n+1}(F)).
\end{align}
It is absolutely convergent for all $s\in \BC$ if $X$ is regular semisimple in which case we denote \begin{align}\label{def O(X,f)}
O(X,f)=O(X,f,0).
\end{align}
\begin{lem}
The integral $O(\xi_\pm,f,s)$ converges absolutely when $\Re(s)>1-\frac{1}{n}$ and extends to a meromorphic function in $s\in \BC$ with at most simple poles at $$s=1-\frac{1}{\ell} +\frac{2\pi i}{\log q}\BZ ,$$ for even integers $\ell$ with $1<\ell\leq n$. Here $q$ is the cardinality of the residue field  $\CO_F/(\varpi)$.
\end{lem}
\begin{proof}
We use the Iwasawa decomposition of $H_n(F)=KAN$.
Define $$f_K(X)=\int_{K}f(kXk^{-1})\,dk.$$ By Iwasawa decomposition on $H_n(F)$, we have
$$
\int_{A(F)}\int_{N(F)}f_K(au\xi_+ u^{-1}a^{-1})\eta(a)|a|^s|\delta(a)|\,du\,da,
$$
where $\delta$ is the modular character 
$$
\delta(a)=a_1^{n-1}a_2^{n-3}...a_{n}^{-(n-1)},\quad a=diag[a_1,...,a_{n}].
$$
By Lemma \ref{lem N+}, this is
$$
\int_{A(F)}\int_{\CV_+(F)}f_K(axa^{-1})\eta(a)|a|^s\,dx\,da.
$$
Replacing $x_{\ell j}$ by $x_{\ell j}a_ja_\ell^{-1}$, $1\leq \ell\leq j-2\leq (n+1)-2$ and setting $a_{n+1}=1$,  we may partially cancel the factor $|\delta(a)|$:
$$
\int_{A(F)}\int_{\CV_+(F)}f_K\left(
\begin{array}{ccccc}0  &\frac{a_1}{a_2} &x_{13} & x_{14} &\ast \\
0 &   0&\frac{a_2}{a_3}&\ast&\ast\\0&0&0&...&... \\ ...&...&0&0 &\frac{a_{n}}{a_{n+1}}
\\
0&...&0&0&0\end{array}\right)\eta(a)|a|^s |a_2a_3...a_n|^{-1}\,dx\,da.
$$
Substitute $b_\ell:=a_\ell/a_{\ell+1},1\leq \ell\leq n$:
$$
\int_{A(F)}\int_{\CV_+(F)}f_K\left(
\begin{array}{ccccc}0  &b_1 &x_{13} & x_{14} &\ast \\
0 &   0&b_2&\ast&\ast\\0&0&0&...&... \\ ...&...&0&0 &b_n
\\
0&...&0&0&0\end{array}\right)\eta(b_1b_3...)|\prod_{\ell=1}^nb_\ell^\ell|^{-1+s}\,\prod_{\ell=1}^n\,db_\ell \,dx.
$$
(Note $db_\ell$ is the additive Haar measure, cf. \S2.)
Now it is clear that the integral converges absolutely if $\Re(\ell(-1+s))>-1$ for all $\ell=1,2,...,n$, or equivalently $\Re(s)>1-\frac{1}{n}$. By Tate's local zeta integral, the integral extends meromorphically to $s\in\BC$ with at most simple poles at those $s$  modulo $\frac{2\pi i}{\log q}\BZ$ satisfying one of the following:
$$
\ell(-1+s)=-1,\quad \ell=2,4,...,2[n/2].
$$
Namely $s=1-\frac{1}{\ell}+\frac{2\pi i}{\log q}\BZ$, for even $\ell$ with $1<\ell\leq n$.

\end{proof}

\begin{defn}\label{def orb uni}
For $f\in \sC_c^\infty(M_{n+1}(F))$, we define the regular nilpotent orbital integral $O(\xi_\pm,f)$, also denoted by $\mu_{\xi_\pm}(f)$ as
$$
\mu_{\xi_\pm}(f)=O(\xi_\pm,f):=O(\xi_\pm, f,0)
.$$ This defines an $(H_n,\eta)$-invariant distribution on $M_{n+1}(F).$
\end{defn}

 The two propositions below will not be used later on.  They are interesting in their own right and provide heuristics for the admissible functions in the remaining sections of this paper. 
\begin{prop}

The intersection $M_{n+1,+}(F)\cap \CN$ is equal to the $H_n$-orbit of $\xi_+$. 
In particular, for a function $f\in \sC_c^\infty(M_{n+1}(F))$ supported on $M_{n+1,+}(F)$, any distribution on $M_{n+1}(F)$ supported in the closed subset $\CN\setminus
( H_n\cdot \xi_+)$ of $M_{n+1}(F)$  vanishes on $f$.
\end{prop}

\begin{proof}Since $\xi_+$ is precisely the image of $0\in \sX$ under $\sigma$,  the $H_n$-orbit of $\xi_+$ is then the image of $H_n\times \{0\}$ under $\iota$.
We also have $M_{n+1,+}\cap \CN=(\pi|_{M_{n+1,+}})^{-1}(0)$, the fiber of $0\in\sX$ under $\pi|_{M_{n+1,+}}$. By Proposition \ref{prop sec}, the fiber $(\pi|_{M_{n+1,+}})^{-1}(0)$ is precisely the image of  $H_n\times \{0\}$ under $\iota$. This proves the first assertion.
The ``In particular" part is clear from the definition of the support of a distribution.
\end{proof}

\begin{prop}\label{cor local const}
For any $f\in \sC_c^\infty(M_{n+1,+}(F))\subset \sC_c^\infty(M_{n+1}(F))$, the orbital integral 
$$
\phi_f(x):=O(\sigma(x), f)
$$
defined for regular semisimple $x\in \sX$ (cf. (\ref{def O(X,f)})) extends to a locally constant function with compact support on $\sX$ (i.e., $\phi_f\in \sC_c^\infty(\sX)$). Conversely, given any function $\phi$ in $\sC_c^\infty(\sX)$, there exists $f\in \sC_c^\infty(M_{n+1,+}(F))$ such that $O(\sigma(x), f)=\phi(x)$ for all regular semisimple $x$.

\end{prop}

\begin{proof}
The orbital integral (\ref{def O(X,f)}) is given by
 $$O(\sigma(x), f)=\int_{H_n}f(h\sigma(x)h^{-1})\eta(h)\,dh.$$
By the $H_n$-equivariant isomorphism  $\iota: H_n(F)\times \sX(F)\to M_{n+1,+}(F)$,  corresponding to $f$ we have an element denoted by $f'$ in $\sC_c^\infty(H_n\times \sX)$, defined by 
$$
f'(h,x)=f(h\sigma(x)h^{-1}),\quad h\in H_n,x\in \sX,
$$with the property that
$$
O(\sigma(x), f)=\int_{H_n}f'(h,x)\eta(h)\, dh.
$$
The integral on the right hand side is clearly absolutely convergent for all $x\in \sX$ and defines an element in $\sC_c^\infty(\sX)$. The converse is clearly now by the isomorphism $\iota$.
\end{proof}

\begin{remark}The proof also shows that if $f$ is supported on $M_{n+1,+}(F)$, 
the integral $O(\xi_+,f,s)$ (cf. (\ref{def O(X,f,s)})) converges absolutely for all $s\in \BC$.
\end{remark}

\subsection{Orbital integrals on $\fks_{n+1}$.} We will need to consider the induced $H_n$-action on the tangent space $\fks_{n+1}$ at $1$ of the symmetric space $S_{n+1}$:
$$
\fks_{n+1}(F)=\{X\in M_{n+1}(E)|X+\ov{X}=0\}.
$$
Fixing a choice of non-zero $\tau\in E^{-}$, we have an isomorphism 
\begin{align}\label{eqn iso M fks}
M_{n+1}(F)\simeq \fks_{n+1}(F),
\end{align}
defined by $X\mapsto \tau X$. 
In particular, we will abuse the notation $\xi_\pm$ to denote $\tau\xi_{\pm}$ if we want to consider the regular unipotent orbit on $\fks_{n+1}$. We may extend the definitions of $\sigma$, $\varrho$ and the orbital integrals to the setting of $\fks_{n+1}$ via the isomorphism (\ref{eqn iso M fks}). Then it is clear how to extend the results from the setting of $M_{n+1}(F) $ to the setting of $\fks_{n+1}$.

\section{Smoothing local periods}

\subsection{Convolutions.}\label{para conv}
We introduce some abstract notions for the use of this and the next section.
Let $F$ be a $p$-adic local field and $G$ the $F$-points of some reductive group. We consider the space of $\sC^\infty_c(G)$ with an (anti-) involution $\ast$ defined by:
\begin{align}
\label{eqn def inv ast}
f^\ast(g):=\ov{f(g^{-1})},\quad f\in \sC^\infty_c(G).
\end{align}
We will also use the other (anti-) involution defined by 
\begin{align}
\label{eqn def inv vee}
f^\vee(g)=f(g^{-1}).
\end{align}
Let $dg$ be a Haar measure on $G$. Let $H$ be a unimodular (closed) subgroup of $G$ and $dh$ a Haar measure $H$. We define a left and a right action of $\sC^\infty_c(H)$ on $\sC^\infty_c(G)$ as follows: for $f\in \sC^\infty_c(G)$ and $\phi\in \sC^\infty_c(H)$,  we define {\em convolutions} $f\ast \phi$ and $\phi\ast f$ both in $\sC^\infty_c(G)$:
\begin{align}
\label{eqn def ast}
(f\ast \phi)(g)=\int_{H}f(gh^{-1})\phi(h)\,dh
\end{align}
and
$$
(\phi\ast f)(g)=\int_{H}\phi(h)f(h^{-1}g)\,dh=\int_{H}\phi(h^{-1})f(hg)\,dh.
$$
This also applies to the case $H=G$. Then we have
$$
(f\ast \phi)^\ast=\phi^\ast\ast f^\ast,\quad (\phi\ast f)^\ast=f^\ast\ast \phi^\ast.
$$
If we have two closed unimodular subgroups $H_1,H_2$, we could iterate the definition: for example, for $f\in \sC^\infty_c(G)$ and $\phi_i\in \sC^\infty_c(H_i)$, we define $$\phi_1\ast\phi_2\ast f:=\phi_1\ast(\phi_2\ast f)\in \sC^\infty_c(G).$$
Now if we have a smooth representation $\pi$ of $G$ (hence its restriction to $H$ is a smooth representation as well), as usual we define $\pi(f)$ and $\pi(\phi)$ to be the endomorphisms of $\pi$:
$$
\pi(f)=\int_{G}f(g)\pi(g)\,dg,\quad \pi(\phi)=\int_{H}\phi(h)\pi(h)\,dh.
$$
Then we have $\pi(f\ast \phi)=\pi(f)\pi(\phi)\in \End(\pi)$ and so on. 
If $\pi$ has a $G$-invaraint inner product $\pair{\cdot,\cdot}$, we then have
$$
\pair{\pi(f)u,u'}=\pair{u,\pi(f^\ast)u'},\quad u,u'\in \pi.
$$

The questions addressed in this section can be abstracted as follows (cf. \cite[\S2]{J2}). Let $\pi$ be a smooth admissible representation of $G$. The algebraic dual space $\pi^\ast:=\Hom(\pi,\BC)$ is usually much larger than the congragredient $\wt{\pi}$ (the subspace of $\pi^\ast$ consisting of smooth linear functional, i.e., $K$-finite vectors in $\pi^\ast$ for some open compact $K$). Very often we will be interested in some distinguished element called $\ell$ in $\pi^\ast\setminus \wt{\pi}$ (the set of non-smooth linear functionals). Then the question is to find some $\phi\in \sC^\infty_c(H)$, for suitable subgroup $H$ (smaller than $G$ in order to be useful), such that $\pi^\ast(\phi)\ell$ is nonzero and smooth (i.e., in $\wt{\pi}$). 
In this section we will study this question for the local Flicker-Rallis period, and the local Rankin-Selberg period. 

\subsection{A compactness lemma.}
Now we return to our setting. Let $E/F$ be a quadratic extension of non-archimedean local fields of characteristic zero with residue characteristic $p$. We denote by $\eta$ the quadratic character associated to $E/F$, and set
\begin{align}\label{eqn def eta}
\eta_n=\eta^{n-1}.
\end{align}
Let $\varphi_{n-1}\in \sC_c^\infty(H_{n-1}(E))$ and $\phi_{n-1}\in \sC_c^\infty(M_{n-1,1}(E))$. We consider the Fourier transform of $\phi_{n-1}$ as a function on $M_{1,n-1}(E)$ by 
$$\wh{\phi}_{n-1}(X)=\int_{M_{1,n-1}(E)}\phi_{n-1}(Y)\psi_E(\tr(XY))\,dY.
$$
To the pair $(\varphi_{n-1},\phi_{n-1})$ we associate a new function on $H_{n-1}(E)$ by
\begin{align}\label{eqn W_n}
\wt{W}_{\varphi_{n-1},\phi_{n-1}}(g):=\wh{\phi}_{n-1}(-e_{n-1}g) 
\int_{} \int\varphi_{n-1}(g^{-1}u\epsilon_{n-1} h)\ov{\psi}_E(u)\eta_n(h)\,du\,dh,
\end{align}
where $u\in N_{n-1}(E), h\in N_{n-1}(F)\bs H_{n-1}(F)$ and the integral is iterated. Note that the integral converges absolutely. Clearly we have $$W_{\varphi_{n-1},\phi_{n-1}}(ug)=\ov{\psi}_E(u)W_{\varphi_{n-1},\phi_{n-1}}(g)$$ for $u\in N_{n-1}(E)$.

We would like to obtain a function with compact support modulo $N_{n-1}(E)$ by imposing suitable conditions on $(\varphi_{n-1},\phi_{n-1})$. To simplify the notation, we will denote, when $p>2$
\begin{align*}
\Lambda=\CO_E.
\end{align*}
It decomposes as $\Lambda=\Lambda^+\oplus \Lambda^-$ where
$\Lambda^\pm=\CO_{E^\pm}$. If $p=2$,  in the rest of the paper we define $\CO_E$ as $\CO_{E^+}\oplus \CO_{E^-}$, which may be a non-maximal order of $E$.  
Then the Fourier transform of $1_\Lambda\in \sC_c^\infty(E)$ is a non-zero multiple of $1_{\Lambda^\ast}$ for a lattice (i.e., an $\CO_E$-module) $\Lambda^\ast\subset E$ (depending on $\psi$). Let $\varpi$ be a uniformizer of $F$. For an integer $m>0$, we naturally view $\BC[\varpi^{m}\Lambda/\varpi^{2m} \Lambda]$ as the subspace of $\sC_c^\infty(E)$ consisting of functions supported in $\varpi^{m}\Lambda$ and invariant by $\varpi^{2m}\Lambda$. 

\begin{defn}\label{def dag 11}
We define a {\em ``dagger"} space of level $m$, denoted by $\BC[\varpi^{m}\Lambda/\varpi^{2m} \Lambda]^\dagger$ or $\sC_c^\infty(E)^{\dagger}_{m}$, as the subspace of $\BC[\varpi^{m}\Lambda/\varpi^{2m} \Lambda]$ spanned by functions $\theta=\theta^+\otimes\theta^-,\theta^\pm\in \sC_c^\infty(E^\pm)$, satisfying: 
\begin{itemize}
\item $\theta^+$ is a multiple of $1_{\varpi^m\Lambda^+}$.
\item The Fourier transform $\wh{\theta}\in \BC[\varpi^{-2m} \Lambda^\ast/\varpi^{-m} \Lambda^\ast]$ is supported in $\varpi^{-2m} \Lambda^\ast-\varpi^{-2m+1} \Lambda^\ast$.  With the condition on $\theta^+$, this is equivalent to that the function $\wh{\theta^-}$ is supported in $\varpi^{-2m} \Lambda^{-\ast}-\varpi^{-2m+1} \Lambda^{-\ast}$ where $\Lambda^{-\ast}=\Lambda^\ast\cap E^-$.\end{itemize} 
\end{defn}
In particular, every element in $\BC[\varpi^{m}\Lambda/\varpi^{2m} \Lambda]^\dagger$  is invariant under multiplication by $1+\varpi^m\CO_E$.
Heuristically, such $\theta$ has a constant real part $\theta^+$ but a highly oscillating imaginary part $\theta^-$.
\begin{defn}\label{def dag n1}
We denote by $\sC_c^\infty(M_{n-1,1}(E))^{\dagger}_{m}$ the space spanned by functions on $M_{n-1,1}(E)$ of the form $\phi_{n-1}=\bigotimes_{1\leq i\leq n-1}\phi^{(i)}$ in the way that $\phi_{n-1}(x)=\prod_{i}\phi^{(i)}(x_i)$ if $x=(x_1,...,x_{n-1})^t\in M_{n-1,1}(E)$ satisfying
 \begin{itemize} 
\item When $1\leq i\leq n-2$, $\phi^{(i)}$ is the characteristic function of $\varpi^m\Lambda$; $\phi^{(n-1)}$ is an element of $\BC[\varpi^{m}\Lambda/\varpi^{2m} \Lambda]^\dagger$.
\end{itemize} 
\end{defn}
\begin{defn}\label{def dag nn}
We denote by $\sC_c^\infty(H_{n-1}(E))^{\dagger}_{m}$ the space spanned by functions on $H_{n-1}(E)$ of the form $\varphi_{n-1}=\bigotimes_{1\leq i,j\leq n-1}\varphi^{(ij)}$ in the way that $\varphi_{n-1}(g)=\prod_{i,j}\varphi^{(ij)}(g_{ij})$ if $g=(g_{ij})$ satisfying:
\begin{itemize}

\item When $1\leq j<i\leq n-1$, $\varphi^{(ij)}$ is the characteristic function of $\varpi^m\CO_E$.
\item  When $1\leq i=j\leq n-1$, $\varphi^{(ij)}$ is the characteristic function of $1+\varpi^m\CO_E$. \item When $1\leq i<j\leq n-1$, $j-i\neq1$, $\varphi^{(ij)}$ is the characteristic function of $\varpi^m\Lambda$.
\item When $1\leq i=j-1\leq n-2$, $\varphi^{(ij)}$ is an element of $\BC[\varpi^{m}\Lambda/\varpi^{2m} \Lambda]^\dagger$.
\end{itemize} 
\end{defn}
\begin{remark}
A function $\varphi_{n-1}\in \sC_c^\infty(H_{n-1}(E))^{\dagger}_{m}$ has the following property:
$$
\varphi_{n-1}(u_{n-2}...u_1av_1...v_{n-2})=
\varphi_{n-1}(u_{\sigma'(n-2)}...u_{\sigma'(1)}av_{\sigma'(1)}...v_{\sigma'(n-2)}),
$$
where $a\in A_{n-1}, u_i\in N_{n-1}$ ($v_i\in N_{n-1,-}$, resp.), $u_i-1$ ($v_i-1$, resp.) has nonzero entries only in the $(i+1)$-th column (row, resp.), and $\sigma,\sigma'$ are any permutations.
\end{remark}
\begin{defn}\label{def m-adm}
We say that the pair $(\varphi_{n-1},\phi_{n-1})$ is {\em $m$-admissible} if $\phi_{n-1}\in\sC_c^\infty(M_{n-1,1}(E))^{\dagger}_{m}$, and $\varphi_{n-1}\in \sC_c^\infty(H_{n-1}(E))^{\dagger}_{m}$. 
\end{defn}
 
\begin{remark}
The pair $(\varphi_{n-1},\phi_{n-1})$ defines a function denoted by $\varphi_{n-1}\otimes\phi_{n-1}$ on the mirabolic subgroup $P_n$ of $H_n(E)$
$$
\varphi_{n-1}\otimes\phi_{n-1}\left[\left(\begin{array}{cc}x & \\
 &  1
\end{array}\right)\left(\begin{array}{cc}1_{n-1} &u \\
 &  1
\end{array}\right)\right]=\varphi_{n-1}(x)\phi_{n-1}(u).
$$
\end{remark}

For {\em $m$-admissible}  $(\varphi_{n-1},\phi_{n-1})$ as above, we define recursively $\varphi_i, \phi_i,\phi_{i+1}'$ for $i=n-2,...,1$, such that
\begin{align}\label{eqn decomp}
\varphi_{i+1}=\varphi_{i}\otimes \phi_{i}\otimes \phi'_{i+1},
\end{align} where $$\varphi_{i}\in \sC_c^\infty(M_{i}(E)),\quad \phi_{i}\in\sC_c^\infty(M_{i,1}(E)),\quad \phi'_{i+1}\in \sC_c^\infty(M_{1,i+1}(E)).$$
Here the function $\varphi_i$ is viewed as a function on $M_i(E)$ (though it is supported in $H_i(E)$). The tensor product is understood as 
$$
\varphi_{i+1}(X_{i+1})=\varphi_{i}(X_i)\phi_{i}(u_i)\phi'_{i+1}(v_{i+1}),
$$
where $X_{i+1}=\left(\begin{array}{ccc}X_{i} & &u_i \\ && \\
 &v _{i+1} &
\end{array}\right)\in M_{i+1}(E),X_i\in M_i(E), u_i\in M_{i,1}(E), v_{i+1}\in M_{1,i+1}(E)$.
Set $\phi_1'=\varphi_1$ so that we have the following decomposition of $\varphi_{n-1}\otimes \phi_{n-1}$:
\begin{center}
\[\begin{tabular}{| l  ccc|r| }
  \hline   
   \multicolumn{1}{|r}{$\phi^{'}_1$}  &     \multicolumn{1}{|r|}{$\phi_1$} & \multicolumn{1}{r|}{$\phi_2$} & $\cdots$&  \multicolumn{1}{r|}{$\phi_{n-1}^{}$} \\ \cline{1-2}
   \multicolumn{1}{|r}{$\phi_2^{'}$} &   \multicolumn{1}{r|}{} & \multicolumn{1}{r|}{} & &  \\ \cline{1-3}
  $\vdots$&& \multicolumn{1}{r|}{}& &  \\  \cline{1-4}
 $\phi'_{n-1}$  &&&&\\  \hline
\end{tabular}\]
\end{center}
To facilitate the exposition, we list the properties of admissible functions that will be used later in our proof.

\begin{prop}\label{prop pro adm}
Let $(\varphi_{n-1}, \phi_{n-1})$ be $m$-admissible (Definition \ref{def m-adm}) and we decompose it according to (\ref{eqn decomp}). Then we have the following properties:
\begin{itemize}
\item[(i)] The function $\phi_i'$ is the characteristic function of $(0,...,0,1)+\varpi^mM_{1,i}(\CO_E)$, and $\phi_i\in \sC_c^\infty(M_{i,1}(E))^{\dagger}_{m}$.

\item[(ii)] The function  $\varphi_{n-1}$ is left and right invariant under $N_{n-1,-}(\varpi^m\CO_E)=1+\varpi^m\frak{n}_{n-1,-}(\CO_E)$.
\item[(iii)]  With respect to the decomposition $M_{n-1}(E)=M_{n-1}(F)\oplus M_{n-1}(E^-)$, the function  $\varphi_{n-1}=\varphi_{n-1}^+\otimes \varphi_{n-1}^-$ is decomposable and the ``real" part $\varphi_{n-1}^+$ is a multiple of the characteristic function of $1+\varpi^mM_{n-1}(\CO_F)$.
\item[(iv)] The function  $\varphi_{n-1}$ is left and right invariant under the compact open subgroup $1+\varpi^mM_{n-1}(\CO_F)$ (i.e., the support of the ``real" part $\varphi_{n-1}^+$ of $\varphi_{n-1}$).
\end{itemize}

\end{prop}
\begin{proof}
They all follow from Definition \ref{def dag n1}, \ref{def dag nn}, and \ref{def m-adm}.
\end{proof}
Property (iii) and (iv) of admissible functions will not be used until the next section.
Our key result of this section is the following compactness lemma.
\begin{lem}
\label{lem W_n fml}
Assume that $(\varphi_{n-1},\phi_{n-1})$ is {\em $m$-admissible} for some $m>0$ and we have the derived functions $\phi_i,\phi_i'$ as above.
\begin{itemize}
\item[(1)] Then the support of the function $\wt{W}_{\varphi_{n-1},\phi_{n-1}}$ is compact modulo $N_{n-1}(E)$, i.e., it defines an element in
$$
\sC_c^\infty(N_{n-1}(E)\bs H_{n-1}(E),\ov{\psi}_E).
$$
Furthermore, $\wt{W}_{\varphi_{n-1},\phi_{n-1}}(\epsilon_{n-1} g)$ is nonzero only when 
$$
g\in H'_{n-1}(E)=N_{n-1}(E)A_{n-1}(E)N_{n-1,-}(E).$$ 

\item[(2)] 
View $\phi':=\otimes_{i=1}^{n-1}\phi_i'$ as a function on $B_{n-1,-}(E)$ or its Lie algebra $\fkb_{n-1,1}(E)$ (this is possible due to the special feature of the function $\phi'$). Denote by $d_n={n\choose 3}$ so that 
$$
\tau^{d_n}=\delta_{n-1}(\epsilon_{n-1})=\det(Ad(\epsilon_{n-1}): N_{n-1}(E)).$$ 
Then the value of $\wt{W}_{\varphi_{n-1},\phi_{n-1}}(\epsilon_{n-1} g)$ at $g=yv\in  A_{n-1}N_{n-1,-}(F)$:
$$
y=\left(\begin{array}{cccc}y_1y_2...y_{n-1} & &&\\&\ddots&&\\
& & y_1y_2&\\& & & y_1
\end{array}\right)\in A_{n-1}(F),
$$and
$$
v=\prod_{i=1}^{n-2}\left(\begin{array}{cc}1_{i} & \\
v_i &  1
\end{array}\right)\in N_{n-1,-}(F),\quad v_i\in M_{1,i}(F),
$$
is given by the product of the constant
\begin{align}
|\tau|_E^{d_n}\int_{B_{n-1,-}(F)}\phi'( b)\,db 
\end{align}
and
\begin{align}\eta_n(y)|\delta_{n-1}(y)|_F \prod_{i=1}^{n-1}\wh{\phi_{n-i}}(- y_i(v_{n-i-1},1)\tau).
\end{align}
\end{itemize}
\end{lem}
Here the measure $db$ on $B_{n-1,-}(F)$ is either the left or the right invariant one; they give the same value to the  integral since the support of $\phi'$ is contained in $1+\varpi\fkb_{n-1,-}(\CO_E)$.

\begin{proof}

It suffices to consider the following absolutely convergent integral:
\begin{align}\label{eqn w(g)}
w(g)=
\int_{B_{n-1,-}(F)} \int_{N_{n-1}(E)}\varphi_{n-1}(g^{-1}u h)\ov{\psi}_\tau(u)\eta_n(h)\,du\,dh,
\end{align}
where we have replaced $N_{n-1}(F)\bs H_{n-1}(F)$ by $B_{n-1,-}$ (with the right invariant measure) and
$$
\psi_{\tau}(u)=\psi_{E,\tau}(u)=\psi_E( \epsilon_{n-1}u\epsilon_{n-1}^{-1}).
$$
Indeed, by a suitable substitution we have 
$$
\wt{W}_{\varphi_{n-1},\phi_{n-1}}(\epsilon_{n-1}g)=|\tau|_E^{d_n}
\wh{\phi}_{n-1}(-e_{n-1}\tau g)  w(g).
$$

By the condition on the support of $\wh{\phi}_{n-1}$, we know that $\wh{\phi}_{n-1}(e_{n-1}g)$ is zero unless the $(n-1,n-1)$-th entry of $g\in H_{n-1}(E)$ is nonzero. Up to the left translation by $N_{n-1}(E)$,  such $g\in H_{n-1}(E)$ is of the form:
$$
g=y_1\left(\begin{array}{cc}x_{n-2} & \\
 &  1
\end{array}\right)\left(\begin{array}{cc}1_{n-2} & \\
 v_{n-2}&  1
\end{array}\right),y_1\in E^\times,x_{n-2}\in H_{n-2}(E),v_{n-2}\in M_{1,n-2}(E).
$$
By the support condition on $\wh{\phi}_{n-1}$ (noting that $\phi_{n-1}\in\sC_c^\infty(M_{n-1,1}(E))^{\dagger}_{m} $, cf. Definition \ref{def dag n1}),  for $w(g)$ in (\ref{eqn w(g)}) to be nonzero, $y_1$ must lie in a compact set of $E^\times$ and $v_{n-2}\in \varpi^mM_{1,n-2}(\CO_E)$.  By the property (ii) in Prop. \ref{prop pro adm}, the function  $\varphi_{n-1}$ is invariant under left multiplication by such $\left(\begin{array}{cc}1_{n-2} & \\
 v_{n-2}&  1
\end{array}\right)$.  Hence we have
\begin{align}
\label{eqn W=w}
\wt{W}_{\varphi_{n-1},\phi_{n-1}}(\epsilon_{n-1}g)=|\tau|_E^{d_n}
\wh{\phi}_{n-1}(\tau y_1 (v_{n-2},1))  w\left[y_1\left(\begin{array}{cc}x_{n-2} & \\
 &  1
\end{array}\right)\right].
\end{align}
Therefore it is enough to consider $w(g)$ when $g=y_1\left(\begin{array}{cc}x_{n-2} & \\
 &  1
\end{array}\right)$ for $x_{n-2}\in H_{n-2}(E)$.

We write $h\in B_{n-1,-}(F)=A_{n-1}(F)N_{n-1,-}(F)$ as
$$
h=b_1 \left(\begin{array}{cc}a_{n-2} & \\
 &  1
\end{array}\right)\left(\begin{array}{cc}1_{n-2} & \\
 c_{n-2}&  1
\end{array}\right),
$$
where $b_1\in F^\times,a_{n-2}\in B_{n-2,-}(F),c_{n-2}\in M_{1,n-2}(F).$ The measure can be chosen as
$$
|a_{n-2}|^{-1}|b_1|^{-1}\,db_1\,dc_{n-2} \,da_{n-2},
$$
where $da_{n-2}$ is the right invariant measure on $B_{n-2,-}(F)$.
For the integration over $u\in N_{n-1}(E)$, we write 
$$
u=\left(\begin{array}{cc}1_{n-2} &u'_{n-2} \\
&  1
\end{array}\right)\left(\begin{array}{cc}u_{n-2} & \\
&  1
\end{array}\right)\in N_{n-1}(E).
$$

Then the product
$g^{-1}uh$ is equal to
\begin{align}\label{eqn prod guh}
b_1y_1^{-1}
\left(\begin{array}{cc}x^{-1}_{n-2} & \\
 &  1
\end{array}\right)\left(\begin{array}{cc}1_{n-2} &u'_{n-2} \\
&  1
\end{array}\right)\left(\begin{array}{cc}u_{n-2} & \\
&  1
\end{array}\right)\left(\begin{array}{cc}a_{n-2} & \\
 &  1
\end{array}\right)\left(\begin{array}{cc}1_{n-2} & \\
 c_{n-2}&  1
\end{array}\right).
\end{align}
Then the last row of the product (\ref{eqn prod guh}) is equal to $y_1^{-1}b_1(c_{n-2},1)\in M_{1,n-1}(E)$. By the condition on the support of $\phi_{n-1}'$  (cf. Property (i) of Prop. \ref{prop pro adm}), we can assume that
$$c_{n-2}\in \varpi^mM_{1,n-2}(\CO_E),
$$
so that $\varphi_{n-1}$ is invariant under right translation by such $\left(\begin{array}{cc}1_{n-2} & \\
 c_{n-2}&  1
\end{array}\right)$  (cf. Property (ii) of Prop. \ref{prop pro adm}).
 The product
of the first four matrices in (\ref{eqn prod guh}) is then equal to
 $$
y_1^{-1}b_1 \left(\begin{array}{cc}x_{n-2}^{-1}u_{n-2}a_{n-2} & x^{-1}_{n-2}u'_{n-2}\\
 &  1
\end{array}\right).
$$
The integrations on $c_{n-2}$ and $u'_{n-2}$ yield respectively
\begin{align*}
 \int_{M_{n-2,1}(F)}\phi_{n-1}'(y_1^{-1}b_1(c_{n-2},1))\,dc_{n-2},
\end{align*}
\begin{align*}
&\int_{M_{n-2,1}(E)}\phi_{n-2}(b_1y_1^{-1}x^{-1}_{n-2}u'_{n-2})\ov{\psi}_\tau(u'_{n-2})\,du'_{n-2}\\
=&|b_1^{-1}y_1|_E^{n-2}|x_{n-2}|_{E}\wh{\phi}_{n-2}(-e_{n-2}\tau y_1 b_1^{-1} x_{n-2}).
\end{align*}
(Here we note that the Fourier transform of $\phi_{n-2}$ is defined by the character $\psi_E$.)
Therefore $w(g)$ is equal to the integration of the function of $b_1\in F^\times$ given by the product of the above two terms and 
$$\int_{B_{n-2,-}(F)} \int_{N_{n-2}(E)}\varphi_{n-2}(y_1^{-1}b_1 x_{n-2}^{-1}u_{n-2}a_{n-2})\ov{\psi}_\tau(u_{n-2})|a_{n-2}|^{-1}\eta_n(h)\,du_{n-2}\,da_{n-2}
$$
with respect to the measure $|b_1|^{n-2}db_1$.
We may repeat the process and hence may assume that the function on $H_{n-2}(E)$ defined by
\begin{align*}
 g_{n-2}\mapsto  & \wh{\phi}_{n-2}(-e_{n-2}g_{n-2}) \int_{B_{n-2,-}(F)} \int_{N_{n-2}(E)}\varphi_{n-2}(g_{n-2}^{-1}u_{n-2}a_{n-2})\\ & \times\ov{\psi}_\tau(u_{n-2})
 |a_{n-2}|^{-1}\eta_n(a_{n-2})\,du_{n-2}\,da_{n-2}
\end{align*}
is zero unless $g_2\in H_{n-2}'(E)$ and its support is compact modulo $N_{n-2}(E)$. By the support condition of $\phi_{n-1}'$ (cf. Property (i) of Prop. \ref{prop pro adm}), we know that $y_1^{-1}b_1\in 1+\varpi^m\CO_E$. Since $y_1$ is in a compact region of $E^\times$, the integration of $b_1$ must also be in a compact region. This implies that $w(g)\neq 0$ only when $x_{n-2}\in H_{n-2}'(E)$ and in a region compact modulo $N_{n-2}(E)$. By the boundedness of $v_{n-2}$ as shown in the equation (\ref{eqn W=w}), we complete the proof of the part $(1)$.

To show part $(2)$, we need to keep track of the computation above. Since we are now assuming that $g\in H_{n-1}(F)$, we have $y_1\in F^\times$ and hence we may substitute $b_1\mapsto b_1 y_1$. Then for $\phi_{n-1}'(b_1(c_{n-2},1))$ to be nonzero, we must have $b_1\in 1+\varpi^m\CO_F$  (cf. Property (i) of Prop. \ref{prop pro adm}). Note that $\wh{\phi_{n-2}}$ and $\varphi_{n-2}$ are invariant under multiplication by scalars in $1+\varpi^m\CO_F$. We see that $w(g)$ is given by the product of
$$
\eta_n(y_1^{n-1}) \int_{M_{n-1,1}(F)}\phi_{n-1}'(b_1(c_{n-2},1))dc_{n-2}\eta(b_1^{n-1})|b_1|^{-1} \,d b_1,
 $$
$$
|x_{n-2}|_{E}\wh{\phi}_{n-2}(-e_{n-2} \tau x_{n-2}),
$$
and
$$\int_{B_{n-2,-}(F)} \int_{N_{n-2}(E)}\varphi_{n-2}( x_{n-2}^{-1}u_{n-2}a_{n-2})\ov{\psi}_\tau(u_{n-2})|a_{n-2}|^{-1}\eta_n(a_{n-2})\,du_{n-2}\,da_{n-2}.
$$
When $y=y_1\cdot diag(x_{n-2},1), x_{n-2}\in B_{n-2,-}(F)$,
we have 
\begin{align}\label{eqn delta rel}
\quad \delta_{n-1}(y)=\delta_{n-2}(x_{n-2})\det(x_{n-2}).
\end{align}
Note that
$$
|x_{n-2}|_{E}=|x_{n-2}|_{F}^2.
$$
Now we may repeat this process to complete the proof.
\end{proof}

For admissible $\varphi_{n-1}\otimes\phi_{n-1}\in  \sC_c^\infty(H_{n-1}(E)\times M_{n,1}(E))$, the funciton  $\wt{W}_{\varphi_{n-1},\phi_{n-1}}$ lies in $\sC_c^\infty(N_{n-1}(E)\bs H_{n-1}(E),\psi_E)$ and therefore determines a unique element in $\CW$, denoted by $W_{\varphi_{n-1},\phi_{n-1}}$, characterized by
\begin{align}
\label{eqn def W_n}
W_{\varphi_{n-1},\phi_{n-1}}\left(\begin{array}{cc}g& \\
 &  1
\end{array}\right)=\wt{W}_{\varphi_{n-1},\phi_{n-1}}(g),\quad g\in H_{n-1}(E).
\end{align}

\subsection{Smoothing local Flicker--Rallis period $\beta_n$.}

Let $\Pi_n$ be an irreducible unitary generic representation of $H_{n}(E)$. We now consider the local period Flicker--Rallis $\beta_n$ (cf. \S3, (\ref{eqn defn FR}), (\ref{eqn defn FR even})). We let $\CW=\CW(\Pi_n,\ov{\psi}_E)$ be the Whittaker model of $\Pi_n$ with respect to the complex conjugate of $\psi_E$ for later convenience. As earlier we have endowed $\CW$ with a non-degenerate positive definite invariant Hermitian structure (cf. (\ref{eqn defn vartheta})):
$$
\pair{W,W'}=\vartheta(W,W')=\int_{N_{n-1}(E)\bs H_{n-1}(E)}W\left(\begin{array}{cc}g & \\
 &  1
\end{array}\right)\ov{W'\left(\begin{array}{cc}g & \\
 &  1
\end{array}\right)}\,dg.
$$
We also consider its Kirillov model denoted by $\CK=\CK(\Pi_n,\ov{\psi}_E)$, which is a certain subspace of smooth functions $\sC^\infty(N_{n-1}(E)\bs H_{n-1}(E),\ov{\psi}_E)$. Moreover, it is well-known that the Kirillov model always contains the subspace $\sC_c^\infty(N_{n-1}(E)\bs H_{n-1}(E),\ov{\psi}_E)$ of smooth compactly-supported functions.

Let $\CW^\ast$  be the (conjugate) algebraic dual space of $\CW$ and $\CH$ be the Hilbert space underlying the unitary representation $\Pi_n$. Then $\CW$ is the space of smooth vectors in $\CH$ and we have inclusions:
$$
\CW\subset \CH\subset \CW^\ast.
$$
The Hermitian pairing on $\CW\times \CW$ extends to $\CH\times \CH$ and $\CW\times \CW^\ast$. A similar discussion also appears in \cite[\S2.1]{J}. We still denote by $\Pi_n$ the representation of $H_{n}(E)$ on $\CW^\ast$ so for any $W\in\CW,W'\in\CW^\ast$:
$$
\pair{\Pi_n(g)W,W'}=\pair{W,\Pi_n(g^{-1})W'}.
$$
Then the local Flicker--Rallis period $\beta_n$ is an element in $\CW^\ast$ defined by (\ref{eqn defn FR}).
To ease notation, we write this as
\begin{align}\label{eqn def beta1}
\beta_n(W)=\int_{N_{n-1}(F)\bs H_{n-1}(F)}W\left(\epsilon_{n-1}h\right)\eta_{n}(h)\,dh,
\end{align}
where $\eta_n$ is as in (\ref{eqn def eta}). We also write this as
$$
\beta_n(W)=\pair{W,\beta_n}.$$
It is $(H_n(F),\eta_{n})$-invariant (\cite{GJaR}):
$$
\beta_n\in \Hom_{H_n(F)}(\CW,\BC_{\eta_{n}}).
$$

We would like to smoothen the local period $\beta_n$ by applying some sort of ``mollifier".  
\begin{prop}\label{prop sm FR}
Let $\Pi_n$ be an irreducible unitary generic representation of $H_{n}(E)$. 
Assume that the pair $\varphi_{n-1}\in  \sC_c^\infty(H_{n-1}(E)), \phi_{n-1}\in \sC_c^\infty(M_{n,1}(E))$ is $m$-admissible for $m>0$. Let $W_{\varphi_{n-1},\phi_{n-1}}\in\CW$ be the element determined by (\ref{eqn def W_n}). Then for every $W\in W(\Pi_n,\ov{\psi}_E)$, we have
$$
\pair{\Pi_n(\varphi_{n-1}^\ast)\Pi_n(\phi^\ast_{n-1})W,\beta_n}=\pair{W,W_{\varphi_{n-1},\phi_{n-1}}}.
$$
In other words, the linear functional $\Pi_n(\phi_{n-1})\Pi_n(\varphi_{n-1})\beta_n$, a priori only in $\CW^\ast$, is indeed a smooth vector and is represented by $W_{\varphi_{n-1},\phi_{n-1}}\in \CW(\Pi_n,\ov{\psi}_E)$. 
\end{prop}

\begin{proof}
The left hand side is given by
$$
\int_{N_{n-1}(F)\bs H_{n-1}(F)}\int_{H_{n-1}(E)}\Pi_n(\phi^\ast_{n-1}) W(\epsilon_{n-1} hg_{n-1})\ov{\varphi}_{n-1}(g^{-1}_{n-1})dg_{n-1}\eta_n(h)\,dh.
$$
Substitute $g_{n-1}\mapsto h^{-1}\epsilon_{n-1}^{-1}g_{n-1}$,
$$
\int_{N_{n-1}(F)\bs H_{n-1}(F)}\int_{H_{n-1}(E)}\Pi_n(\phi^\ast_{n-1}) W(g_{n-1})\ov{\varphi}_{n-1}(g_{n-1}^{-1}\epsilon_{n-1} h)dg_{n-1}\eta_n(h)\,dh.
$$
Since  $W(ug_{n-1})=\ov{\psi}_E(u)W(g_{n-1})$ for $u\in N_{n-1}(E)$, we may rewrite the integral as$$
\int \Pi_n(\phi^\ast_{n-1}) W( g_{n-1})\left(\int_{N_{n-1}(E)}\ov{\varphi}_{n-1}(g_{n-1}^{-1}u^{-1}\epsilon_{n-1} h)\ov{\psi}_E(u)du\right)\eta_n(h)\,dh\,dg_{n-1},
$$
where the outer integral is over
$$
h\in N_{n-1}(F)\bs H_{n-1}(F),\quad g_{n-1}\in N_{n-1}(E)\bs H_{n-1}(E).
$$
Now we also note that 
\begin{align*}
 \Pi_n(\phi_{n-1}^\ast) W( g_{n-1})&=\int_{M_{n-1,1}(E)}W\left[g_{n-1}\left(\begin{array}{cc}1_{n-1}&u \\
 &  1
\end{array}\right)\right]\phi_{n-1}^\ast(u)\,du
\\&=W\left(g_{n-1}\right)\int_{M_{n-1,1}(E)}\ov{\phi}_{n-1}(-u)\ov{\psi}_E(e_{n-1}^\ast g_{n-1}u)\,du
\\&=W\left(g_{n-1}\right)\ov{ \wh{\phi}}_{n-1}(-e_{n-1}^\ast g_{n-1}).
\end{align*}
This completes the proof.
\end{proof}

\subsection{Smoothing local Rankin-Selberg period $\lambda$.} 
Now let $\Pi=\Pi_n\otimes \Pi_{n+1}$ be an irreducible unitary generic representation of $G'(F)=H_{n}(E)\times H_{n+1}(E)$.
We now need another compactness lemma. For an integer $m'>0$, we consider $\phi_n\in \sC_c^\infty(M_{n,1}(E))_{m'}^\dagger$ (cf. Definition \ref{def dag n1}). Note that  by definition, the Fourier transform $\wh{\phi}_n$ is supported in the domain
\begin{align}\label{eqn supp cond phi n}
\{(x_1,...,x_n)\in M_{1,n}(E)| |x_i/x_n|\leq|\varpi|^{m'}, i=1,2,...,n-1. \}
\end{align}
\begin{lem}\label{lem W_n+1}
Let $\phi_{n}\in \sC_c^\infty(M_{n,1}(E))_{m'}^\dagger$. Let $W\in\CW(\Pi_n,\ov{\psi}_E)$ be such that its restriction to $ H_{n-1}(E)$ has support compact modulo $N_{n-1}(E)$. 
If $m'$ is sufficiently large (depending on the vector $W$), 
the map 
$$
H_n(E)\ni g\mapsto \wh{\phi}_n(e_n  g)W(g)
$$ defines an element in
$\sC_c^\infty(N_n(E)\bs H_{n}(E),\ov{\psi}_E)$.  \end{lem}
\begin{proof}Clearly $\wh{\phi}_n(e_n  g)$ is $N_{n}(E)$-invariant and smooth.  Hence the product defines an element in $\sC^\infty(N_n(E)\bs H_{n}(E),\ov{\psi}_E)$. It remains to show that the product has support compact modulo $N_{n}(E)$. By the support condition of $ \wh{\phi}_n(e_n g)$ (cf. (\ref{eqn supp cond phi n})), we may assume that the lower right entry of $g$ is non-zero. Therefore we may write
$g=xu\left(\begin{array}{cc}h & \\
 &  1
\end{array}\right)\left(\begin{array}{cc}1_{n-1} & \\
v &  1
\end{array}\right)$ where $h\in H_{n-1}(E),u\in N_n(E),v\in M_{1,n-1}(E)$ and $x\in E^\times$. Again by the assumption on $\wh{\phi}_n$, we may assume that $||v||<|\varpi|^{m'}$ (otherwise $ \wh{\phi}_n(e_n g)$ vanishes). We may choose $m'$ sufficiently large so that $W$ is invariant under right multiplication by $1+\varpi^{m'}M_{n}(\CO_E)$ (such $m'$ exists since $W\in \CW$ is smooth).  We now have
\begin{align*}
\wh{\phi}_n(e_n  g)W(g)&=\psi_E(u)\wh{\phi}_n(x(v,1))W\left[x\left(\begin{array}{cc}h & \\
 &  1
\end{array}\right)\right]
\\&=\psi_E(u)\omega_{\Pi_n}(x)\wh{\phi}_n(x(v,1))W\left[\left(\begin{array}{cc}h & \\
 &  1
\end{array}\right)\right],
\end{align*}
where $\omega_{\Pi_n}$ is the central character of $\Pi_n$. Since the last factor has support compact modulo $N_{n-1}(E)$, and $x$ lies in a compact region, the desired compactness follows.
\end{proof}

Since the Kirillow model $\CK(\Pi_{n+1},\ov{\psi}_E)$ contains $\sC_c^\infty(N_n(E)\bs H_{n}(E),\ov{\psi}_E)$ as a subspace, we may view the product $\wh{\phi}_n(e_n  g)W(g)$ as an element in $\CK(\Pi_{n+1},\ov{\psi}_E)$. This determines uniquely an element in the Whittaker model $\CW(\Pi_{n+1},\ov{\psi}_E)$, denoted by $W_{\phi_n}$. In other words, to each $W\in\CW(\Pi_n,\ov{\psi}_E)$ whose restriction to $H_{n-1}(E)$ lies in $\sC_c^\infty(N_{n-1}(E)\bs H_{n-1}(E),\ov{\psi}_E)$ and a $\phi_n\in \sC_c^\infty(M_{n,1}(E))^\dagger_{m'}$ for $m'$ sufficiently large, we associate an element $W_{\phi_{n}}\in \CW(\Pi_{n+1},\ov{\psi}_E)$ characterized by
\begin{align}
\label{eqn W_n+1}
W_{\phi_n}\left(\begin{array}{cc}g & \\
 &  1
\end{array}\right)=\wh{\phi}_n(e_n  g)W(g),\quad g\in H_{n}(E).
\end{align}
Its complex conjugate $\ov{W}_{\phi_n}$ defines an element in $\CW(\Pi_{n+1},\psi_E)$.

Now we recall that the local Rankin-Selberg period is defined by (cf. (\ref{def local RS}):
$$
\lambda(s,W\otimes W')=\int_{N_{n}(E)\bs H_{n}(E)}W(g)W'\left(\begin{array}{cc}g& \\
 &  1
\end{array}\right)|g|^s\,dg,
$$
where $W\in \CW(\Pi_n,\ov{\psi}_E),W'\in  \CW(\Pi_{n+1},\psi_E)$. It has a meromorphic continuation to $s\in \BC$ (with possible poles at those poles of the local Rankin-Selberg L-factor $L(s+1/2,\Pi_v)$). For $\phi_{n}\in\sC_c^\infty(M_{n,1}(E))$, we have an action on $\CW(\Pi_{n+1},\psi_E)$ by:
$$
\Pi_{n+1}(\phi_n)W'(g)=\int_{M_{n-1,1}(E)}W'\left[g\left(\begin{array}{cc}1_{n-1} & u \\
 &  1
\end{array}\right)\right]\phi_n(u)\,du,\quad g\in H_{n+1}(E).
$$

\begin{prop}\label{prop sm RS}Let $\Pi=\Pi_n\otimes\Pi_{n+1}$ be an irreducible unitary generic representation of $H_n(E)\times H_{n+1}(E)$. 
Assume that the restriction of $W\in \CW(\Pi_n,\ov{\psi}_E)$ to $H_{n-1}(E)$ lies in $\sC_c^\infty(N_{n-1}(E)\bs H_{n-1}(E),\ov{\psi}_E)$ and that $\wh{\phi}_n\in \sC_c^\infty(M_{n,1}(E))^\dagger_{m'}$ for $m'$ sufficiently large (depending on $W$). Then for every $W'\in \CW(\Pi_{n+1},\psi_E)$,  $\lambda(s,W\otimes \Pi_{n+1}(\phi_n)W')$ is an entire function in $s\in \BC$ and we have
$$
\lambda(0,W\otimes \Pi_{n+1}(\phi_n)W')=\pair{W',\ov{W}_{\phi_n}}.
$$
\end{prop}

\begin{proof}For $g\in H_{n}(E)$, we have
\begin{align*}
\Pi_{n+1}(\phi_n)W'\left(\begin{array}{cc}g & \\
 &  1
\end{array}\right)=&\int_{M_{n-1,1}(E)}W'\left[\left(\begin{array}{cc}g & \\
 &  1
\end{array}\right)\left(\begin{array}{cc}1_{n-1} & u \\
 &  1
\end{array}\right) \right]\phi_n(u)\,du
\\=&
\int_{M_{n-1,1}(E)}W'\left[\left(\begin{array}{cc}g & \\
 &  1
\end{array}\right)\right]\psi_E(e_n gu)\phi(u)\,du
\\=&\wh{\phi}_n(e_n g)W'\left[\left(\begin{array}{cc}g & \\
 &  1
\end{array}\right)\right].
\end{align*}
Return to the local Rankin-Selberg period:
\begin{align}\label{eqn RSs}
\lambda_{}(s,W\otimes \Pi_{n+1}(\phi_n)W')=
\int_{N_{n}(E)\bs H_{n}(E)}W(g)\wh{\phi}_n(e_n g) W'\left(\begin{array}{cc}g& \\
 &  1
\end{array}\right)|g|^s\,dg.
\end{align}
By Lemma \ref{lem W_n+1}, $W(g)\wh{\phi}_n(e_n g)$ has compact support modulo $N_{n}(E)$. It follows that the last integral converges absolutely for all $s\in \BC$ and hence defines an entire function in $s\in \BC$. Moreover the value at $s=0$ is given by (cf. \ref{eqn W_n+1})
\begin{align*}
\lambda_{}(0,W\otimes \Pi_{n+1}(\phi_n)W')=\pair{W',\ov{W}_{\phi_n}}.
\end{align*}
This completes the proof.
\end{proof}

\section{Local character expansion in the general linear group case}

In this section we prove a ``limit" formula for the (local) spherical character in the general linear group case. We will choose a subspace of test functions supported in a small neighborhood of the origin of the symmetric space $S_{n+1}$.  Then we may treat them as  functions on the ``Lie algebra" $\mathfrak{s}_{n+1}$ of $S_{n+1}$, the tangent space at the origin. The key property for  these functions is that their Fourier transforms have vanishing unipotent orbital integrals except for the regular unipotent element. The intermediate steps are messy and somehow ugly; but the final outcome seems to be miraculously neat.

\subsection{A ``limit" formula for the spherical character $I_\Pi(f)$.} Now let $\Pi=\Pi_n\otimes \Pi_{n+1}$  be an irreducible unitary generic representation of $H_n(E)\times H_{n+1}(E)$. Assume that the central characters of both $\Pi_n$ and $\Pi_{n+1}$ are trivial on $F^\times$. Let  $I_{\Pi,s}(f)$ be the (unnormalized) local spherical character defined by (\ref{def unnorm char G'}). We now combine Prop. \ref{prop sm FR} and \ref{prop sm RS} to obtain a formula of $I_{\Pi,s}(f)$ for a special class of test functions $f$.

Consider $f_n\in \sC_c^\infty(H_n(E))$ and $f_{n+1}\in \sC_c^\infty(H_{n+1}(E))$.
Let $\phi_n\in \sC_c^\infty(M_{n,1}(E))$. We consider a {\em perturbation} of $f_{n+1}$ by $\phi_n$
$$
f_{n+1}^{\phi_n}(g)=\int_{M_{n,1}(E)}f_{n+1}\left[\left(\begin{array}{cc}1_n & u\\
 &  1
\end{array}\right)g \right]\phi_n(u)\,du,\quad g\in H_{n+1}(E).
$$
This is the same as $\phi_n\ast f_{n+1}$
---the convolution introduced in \S\ref{para conv}, (\ref{eqn def ast})---where we view $M_{n,1}(E)$ as a subgroup of $H_{n+1}(E)$.
Similarly, for $\varphi_{n-1}\in  \sC_c^\infty(H_{n-1}(E)),\phi_{n-1}\in \sC^\infty_c (M_{n-1,1}(E))$, we define  a {\em perturbation} of $f_{n}$: for $g\in H_n(E)$,
$$
f_{n}^{\varphi_{n-1},\phi_{n-1}}(g):=\int_{M_{n-1,1}(E)}\int_{H_{n-1}(E)}f_n\left[g\left(\begin{array}{cc}x^{-1} & \\
 &  1
\end{array}\right) \left(\begin{array}{cc}1 & -u\\
 &  1
\end{array}\right) \right] \varphi_{n-1}(x)\phi_{n-1}(u)\,dx\,du.
$$
Equivalently $$f_{n}^{\varphi_{n-1},\phi_{n-1}}=f_{n}\ast \phi_{n-1}\ast \varphi_{n-1}\in \sC^\infty_c(H_n(E)).$$
We will consider functions of the following form:
$$
f=f_{n}^{\varphi_{n-1},\phi_{n-1}}\otimes f^{\phi_{n}}_{n+1}\in \sC_c^\infty(H_{n}(E)\times H_{n+1}(E)).
$$

\begin{defn}\label{def adm}
Fix $\Pi$. Let $(m,m',r)$ be positive integers with $r>m'>m>0$. We say that $f=f_{n}^{\varphi_{n-1},\phi_{n-1}}\otimes f^{\phi_{n}}_{n+1}$ is {\em $(m,m',r)$-admissible} or {\em admissible} for $\Pi$ if it satisfies the following:
\begin{itemize}
\item The function $\varphi_{n-1}\otimes \phi_{n-1}$ is $m$-admissible. Hence it determines an element $\wt{W}_{\varphi_{n-1},\ov{\phi}_{n-1}}\in \sC_c^\infty(N_{n-1}\bs H_{n-1}(E),\psi_E)$ (cf. Lemma \ref{lem W_n fml}) and $W_{\varphi_{n-1},\phi_{n-1}}\in\CW(\Pi_n,\psi_E)$. 
\item The function $\phi_n\in \sC_c^\infty(M_{n,1}(E))^\dagger_{m'}$ for sufficiently large $m'$ (depending on $\Pi_n$, $\varphi_{n-1}\otimes \phi_{n-1}$ and hence on the integer $m$). More precisely, $m'$ is large enough such that Lemma \ref{lem W_n+1} holds for $W=W_{\varphi_{n-1},\phi_{n-1}}$. Let $W_{\varphi_{n-1},\phi_{n-1},\phi_n}\in \CW(\Pi_{n+1},\ov{\psi}_E)$ be the function characterized by the equation (\ref{eqn W_n+1}) for the choice $W=W_{\varphi_{n-1},\phi_{n-1}}$.
\item The function 
 $f_n$  ($f_{n+1}$, resp.) is a multiple of the characteristic function of $1+\varpi^rM_{n}(\CO_E)$ (of $1+\varpi^{r}M_{n+1}(\CO_E)$ , resp.). We normalize $f_n,f_{n+1}$ by
\begin{align}\label{eqn norm f}
 \int_{H_n(E)}f_n(g)\,dg  = \int_{H_{n+1}(E)}f_{n+1}(g)\,dg =1.
\end{align}
We require that $r$ is sufficiently large (depending on $\Pi$, $\varphi_{n-1},\phi_{n-1},\phi_n$ and hence on $m,m'$) so that 
$$
\Pi_n(f_n)W_{\varphi_{n-1},\phi_{n-1}}=W_{\varphi_{n-1},\phi_{n-1}}
$$
and
$$
\Pi_{n+1}(f_{n+1})W_{\varphi_{n-1},\phi_{n-1},\phi_n}=W_{\varphi_{n-1},\phi_{n-1},\phi_n}.
$$\end{itemize}
\end{defn}

\begin{prop}Fix $\Pi$ and assume that $f=f_{n}^{\varphi_{n-1},\phi_{n-1}}\otimes f^{\phi_{n}}_{n+1}$ is $(m,m',r)$-admissible for $\Pi$. Then we have
\begin{align*}
I_{\Pi,s}(f)=\beta_{n+1}(W_{\varphi_{n-1},\phi_{n-1},\phi_n}).
\end{align*}
In particular it is independent of $s\in \BC$.
\end{prop}

\begin{proof}
First we have
\begin{align*}
I_{\Pi,s}(f_n\otimes f_{n+1})=&\sum_{W,W'}\lambda(s,\Pi_n(f_n)W\otimes \Pi_{n+1}(f_{n+1})W')\ov{\beta_n(W)} \cdot \ov{\beta_{n+1}(W')}
\\=&\sum_{W,W'}\lambda(s,W\otimes \Pi_{n+1}(f_{n+1})W')\ov{\beta_n(\Pi_n(f_n^\ast)W)}\cdot \ov{\beta_{n+1}(W')},
\end{align*}
where the sum of $W$ ($W'$, resp.) runs over an orthonormal basis of $\Pi_n$ ($\Pi_{n+1}$, resp.).

For simplicity, we write $\phi$ for $\phi_{n-1}$ and $\varphi$ for $\varphi_{n-1}$.
Now we replace $f_n$ by $f^{\varphi_{},\phi_{}}_n=f\ast \phi\ast \varphi$. Note that
$$
\Pi_n((f^{\varphi,\phi}_n)^\ast)=\Pi_n(\varphi^\ast)\Pi_n(\phi^\ast)\Pi_n(f^\ast)
$$
By Prop. \ref{prop sm FR}, we have for all $W\in \CW(\Pi_n,\ov{\psi}_E)$
$$
\beta_n(\Pi_n((f^{\varphi,\phi}_n)^\ast)W)=\pair{\Pi_n(f_n^\ast)W,W_{\varphi,\phi}}=\pair{W,\Pi_n(f_n)W_{\varphi,\phi}}.
$$
For any $W_0\in \CW(\Pi_n,\ov{\psi}_E)$, we have an orthonormal expansion
$$
W_0=\sum_W \pair{W_0,W}W,
$$
where the sum of $W$ runs over an orthonormal basis of $\Pi_n$.
Hence we may fold the sum over $W$ to obtain
\begin{align*}
I_{\Pi,s}(f^{\varphi,\phi}_n\otimes f_{n+1})=&\sum_{W'}\lambda(s,\Pi_n(f_n)W_{\varphi,\phi} \otimes \Pi_{n+1}(f_{n+1})W') \ov{\beta_{n+1}(W')}.
\end{align*}
Now we further replace $ f_{n+1}$ by $ f_{n+1}^{\phi_n}=\phi_n\ast f_{n+1}$ and assume that $f=f_{n}^{\varphi,\phi_{}}\otimes f^{\phi_{n}}_{n+1}$ is admissible. By the admissibility, $f_n$ and $f_{n+1}$ have small support so that we have$$
\Pi_n(f_n)W_{\varphi,\phi}=W_{\varphi,\phi}.
$$and
$$
\Pi_{n+1}(f^\ast_{n+1})\ov{W}_{\varphi,\phi,\phi_n}=\ov{W}_{\varphi,\phi,\phi_n}.
$$  Now note that the function $W_{\varphi,\phi}(g)\wh{\phi}_n(e_n g)$ is supported in $N_n(E)H_n(\CO_E)$. Hence in (\ref{eqn RSs}), we have $|g|^s=1$ independent of $s\in \BC$.  Now by Prop. \ref{prop sm RS} and $\int_{H_n(E)}f_n(g)dg=\int_{H_{n+1}(E)}f_{n+1}(g)dg=1$, we have for all $W'\in \CW(\Pi_{n+1},\psi_E)$:
\begin{align*}
\lambda(s,\Pi_n(f_n)W_{\varphi,\phi} \otimes \Pi_{n+1}(f_{n+1}^{\phi_n})W')
=&\pair{\Pi_{n+1}(f_{n+1})W',\ov{W}_{\varphi,\phi,\phi_n} }
\\=&\pair{W',\Pi_{n+1}(f^\ast_{n+1})\ov{W}_{\varphi,\phi,\phi_n} }
\\=&\pair{W',\ov{W}_{\varphi,\phi,\phi_n} }.
\end{align*}
Then we may fold the sum over $W'$ to obtain
\begin{align*}
I_{\Pi,s}(f^{\varphi,\phi}_n\otimes f^{\phi_n}_{n+1})
=\ov{\beta_{n+1}(\ov{W}_{\varphi,\phi,\phi_n})}=\beta_{n+1}(W_{\varphi,\phi,\phi_n}).
\end{align*}
This completes the proof.
\end{proof}

We need to simplify the formula. Define for $a\in A_n(F)$:
\begin{align}
\label{def del bar}
\ov{\delta}_n(a):=\det(Ad(a):\fkn_{n}))/\det(Ad(a):\fkn_{n-1}).
\end{align}
{\em To simplify the exposition from now on we redefine $f_{n+1}^{\phi_n}$ as $\phi_{n}^\vee\ast f_{n+1}$.}
\begin{prop}\label{prop fml I Pi}
Fix $\Pi$ and assume that $f=f_{n}^{\varphi_{n-1},\phi_{n-1}}\otimes f^{\phi_{n}}_{n+1}$ is $(m,m',r)$-admissible for $\Pi$. Use notations as in Lemma \ref{lem W_n fml}. Then we have
\begin{align*}
I_{\Pi,s}(f)=\omega_{\Pi_n}(\tau)|\tau|_E^{d_n}\left(\int_{B_{n-1,-}(F)}\phi'( b)\,db \right)\int \prod_{i=0}^{n-1}\wh{\phi_{n-i}}(- y_i(v_{n-i-1},1)\tau)|\ov{\delta}_{n}(y)|^{-1}\eta(y)\,d^\ast y\,dv,
\end{align*}
where the integral of $d^\ast ydv$ is over $y\in A_n(F),v\in N_{n,-}(F)$:
\begin{align}\label{eqn for y00}
y=\left(\begin{array}{cccc}y_0y_1y_2...y_{n-1} & &&\\&\ddots&&\\
& & y_0 y_1&\\& & & y_0
\end{array}\right)\in A_{n}(F),
\end{align}
and
\begin{align}\label{eqn for v00}
v=\prod_{i=1}^{n-1}\left(\begin{array}{cc}1_{i} & \\
v_i &  1
\end{array}\right)\in N_{n,-}(F),\quad v_i\in M_{1,i}(F).
\end{align}
\end{prop}

\begin{proof}
By (\ref{eqn W_n+1}), we have (cf. (\ref{eqn def beta1}), and note to replace $\phi_n$ by $\phi_n^\vee$)
\begin{align*}
\beta_{n+1}(W_{\varphi,\phi,\phi_n})
=&\int_{N_{n}(F)\bs H_n(F)}W_{\varphi,\phi}(\epsilon_n h)\wh{\phi}_n(-e_n  \epsilon_n h)\eta_{n+1}(h)\,dh.
\end{align*}
Let $h=yv$ where
$$
y=y_0\left(\begin{array}{cc}y' & \\
&  1
\end{array}\right) \in A_{n}(F),\quad y'\in A_{n-1}(F),
$$ and
$$v=\left(\begin{array}{cc}v' & \\
 &  1
\end{array}\right)\left(\begin{array}{cc}1_{n-1}& \\
v_{n-1} &  1
\end{array}\right),\quad v'\in N_{n-1,-}(F).
$$
Then we may replace the integral on the quotient $N_{n}(F)\bs H_n(F)$ by the integral over $y,v$ as in (\ref{eqn for y00}) and (\ref{eqn for v00}) for the measure:
$$
|\delta_{n}(y)|^{-1}\prod_{i=0}^{n-1}\,d^\ast y_i \prod_{j=1}^{n-1}\,dv_j.
$$
Since we have $\delta_n(y)=\delta_{n-1}(y')\det(y')=\delta_{n-1}(y')\ov{\delta}_n(y')$ (cf. (\ref{eqn delta rel}), (\ref{def del bar})),
we may also write this as a product 
$$
\left(|\delta_{n-1}(y')|^{-1}|y'|^{-1} \,d^\ast y'\,dv'\right) \,d^\ast y_0 \,dv_{n-1},
$$
where $d^\ast y'=\prod_{i=0}^{n-2}\,d^\ast y_i$ and $dv'= \prod_{j=1}^{n-2}\,dv_j$.
Since the central character of $\Pi_n$ is trivial on $F^\times$, we have
$$
W_{\varphi,\phi}(\epsilon_n y_0h)=W_{\varphi,\phi}(\epsilon_n h).
$$
By the admissibility, the value $\wh{\phi}_n(-e_n \epsilon_n h)=\wh{\phi}_n(- y_0(v_{n-1},1)\tau)$ is non-zero only if $\|v_{n-1}\|\leq |\varpi^{m'}|$ is very small so that $\left(\begin{array}{cc}1_{n-1}& \\
v_{n-1} &  1
\end{array}\right)$ acts trivially on $W_{\varphi,\phi}$.
This allows us to write the integral as the product of
$$
\int_{}\wh{\phi}_n(- y_0(v_{n-1},1)\tau)\eta_{n+1}(y_0^{n})\,d^\ast y_0\,dv_{n-1}
$$ 
and
$$
\omega_{\Pi_n}(\tau)\int_{A_{n-1}(F)}\int_{N_{n-1,-}(F)}W_{\varphi,\phi}\left(\begin{array}{cc}\epsilon_{n-1} y'v' & \\
 &  1
\end{array}\right) |\delta_{n-1}(y')|^{-1}|y'|^{-1} \eta_{n+1}(y')\,d^\ast y'\,dv',
$$
where we have used the equality
$$
\epsilon_{n}=\tau\left(\begin{array}{cc} \epsilon_{n-1} & \\
 &  1
\end{array}\right).
$$
By definition $W_{\varphi,\phi}\left(\begin{array}{cc}\epsilon_{n-1} y'v' & \\
 &  1
\end{array}\right) =\wt{W}_{\varphi,\phi}\left(\epsilon_{n-1} y'v' \right) $ (cf. (\ref{eqn def W_n})),
we may apply the formula of the latter by Lemma \ref{lem W_n fml}:
\begin{align*}W_{\varphi,\phi}\left(\begin{array}{cc}\epsilon_{n-1} y'v' & \\
 &  1
\end{array}\right)=
|\tau|_E^{d_n}\left(\int_{B_{n-1,-}(F)}\phi'( b)\,db \right)\eta_n(y')|\delta_{n-1}(y')|_F \prod_{i=1}^{n-1}\wh{\phi_{n-i}}( y_i(v_{n-i-1},1)\tau).
\end{align*}
Finally we note $\eta_n\eta_{n+1}=\eta$ and $\eta_{n+1}(y_0^n)=\eta(y_0^{n^2})=\eta(y_0^{n})$ (cf. (\ref{eqn def eta})).
This completes the proof.

\end{proof}

\subsection{Truncated local expansion of the spherical character $I_\Pi$.} We are now ready to deduce a truncated local expansion of $I_\Pi$ around the origin. As we have alluded to in the Introduction, this expansion is the relative version of a theorem of Harish-Chandra.  His result is a local expansion of the character of an admissible representations of a $p$-adic reductive group in terms of Fourier transform of nilpotent orbital integrals. Here we obtain a truncated expansion that only involves the regular unipotent element. 

First we need to associate to $f\in \sC_c^\infty(H_{n}(E)\times H_{n+1}(E))$ with small support around $1$ a function on the Lie algebra $\fks$ with small support around $0$. To $f=f_n\otimes f_{n+1}\in \sC_c^\infty(H_{n}(E)\times H_{n+1}(E))$ we have associated a function $\wt{f}\in  \sC_c^\infty(H_{n+1}(E))$ by (\ref{eqn def wtf}) and $\wt{\wt f}\in \sC_c^\infty(S_{n+1}(F))$ by (\ref{eqn def nuf even}), (\ref{eqn def nuf odd}). It is easy to see that $\wt{f}=f^\vee_n\ast f_{n+1}$ (cf. (\ref{eqn def inv vee})).
The Cayley map (cf. (\ref{def cay ugp})) defines a local homeomorphism near a neighborhood of $0\in \fks$
\begin{align*}
\fkc=\fkc_{n+1}:\fks&\to S_{n+1}\\
         X& \mapsto (1+X)(1-X)^{-1},
\end{align*} 
and its inverse is given by
$$
\fkc^{-1}(x)=-(1-x)(1+x)^{-1}.
$$
In particular, for a function $\Phi \in \sC_c^\infty(S_n )$ ($\phi \in \sC_c^\infty(\fks )$, resp.) with support in a small neighborhood of $1\in S_{n+1}$ ($0\in \fks$, resp.), we may consider it as a function on $\fks$ ($S_{n+1}$, resp.) denoted by $\fkc^{-1}(\Phi)$ ($\fkc(\phi))$, resp.). 
We also have a morphism:
\begin{align*}
\iota=\iota_{n+1}: \fks &\to H_{n+1}(E)\\
X&\mapsto 1+X,
\end{align*}
such that, wherever they are all defined, we have $\nu\circ \iota=\fkc:$
$$
\xymatrix{   & H_{n+1}(E) \ar[d]^{\nu}\\
\fks_{}\ar[ur]^{\iota}\ar[r]^{\fkc} & S_{n+1}
 }
$$
where $\nu$ is as in (\ref{eqn def nu}).
\begin{defn}\label{def f_nat}
We associate to $f\in  \sC_c^\infty(H_{n}(E)\times H_{n+1}(E))$ a function on $\fks$ denoted by $f_\nat$:
\begin{align}\label{def f_nat}
f_\nat(X):=\int_{H_n(F)}\wt{f}((1+X)h)\,dh,\quad X\in \fks,
\end{align}
if $n$ is even, and 
\begin{align}\label{def f_nat odd}
f_\nat(X):=\int_{H_n(F)}\wt{f}((1+X)h)\eta'((1+X)h)\,dh,\quad X\in \fks,
\end{align}
if $n$ is odd,
where $\wt{f}$ is defined by (\ref{eqn def wtf}).
\end{defn}
Then we have when $\det(1-X)\neq 0$
$$
\fkc^{-1}(\wt{\wt{f}})(X)=f_\nat(X),\quad X\in \fks.
$$
{\em From now on, all the test functions at hand will be supported in suitable neighbourhoods of the obvious distinguished points of $H_{n}(E)\times H_{n+1}(E),S_{n+1}(F)$ and $\fks(F)$ so that $\fkc$ is well-defined}. In particular, we have $f_\nat\in \sC^\infty_c(\fks)$.

We then consider $\fks(F)$ as a subspace of $M_{n+1} (E)$. On $M_{n+1} (E)$ we have a bilinear pairing $\pair{X,Y}:=\tr(XY)$, under which the decomposition $M_{n+1} (E)=M_{n+1} (F)\oplus \fks(F)$ is orthogonal. We then define the Fourier transform on $\fks(F)$ with respect to the restriction of the above pairing:
$$
\wh{f_\nat}(X):=\int_{\fks}f_\nat(Y)\psi(\tr(XY))\,dY.
$$
Here the measure is normalized so that $\wh{\wh{f_\nat}}(X)=f_\nat(-X)$ (cf. \S2 (\ref{eqn measure})). We will consider the orbital integral (Definition \ref{def orb uni})  of the regular unipotent element
\begin{align}\label{eqn xi- s_n}
\xi_-=\xi_{n+1,-}=\tau\left(
\begin{array}{cccc}0  &0 &... & 0\\
 1&  0&0&...\\ ...&1&0 &0
\\
0&...&1&0\end{array}\right)\in \fks(F).
\end{align}

\begin{thm}\label{thm germ gl}
Let $\Pi$ be an irreducible unitary generic representation. Then for any small neighborhood $\Omega$ of $1\in G'(F)$, there exists admissible $f\in \sC_c^{\infty}(\Omega)$ such that for all $s\in \BC$:
$$
I_{\Pi,s}(f)=|\tau|_E^{(d_n+d_{n+1})/2}\omega_{\Pi_n}(\tau)\cdot \mu_{\xi_-} (\wh{f_\nat}),
$$where
$\omega_{\Pi_n}$ is the central character of $\Pi_n$, the constant $d_n={n \choose 3}$ is as in (\ref{eqn def d_n}).
\end{thm}
\begin{remark}
Note that $\mu_{\xi_-} (\wh{f_\nat})$ depends on the choice of $\tau$, but $\eta'(\Delta_-(\xi_-))\mu_{\xi_-} (\wh{f_\nat})$ does not .
\end{remark}
The proof will occupy the rest of this section by several steps.

\subsection{Determine $f_\nat$.}
We consider admissible functions of the form $f_{n}^{\varphi_{n-1},\phi_{n-1}}\otimes f^{\phi_{n}^\vee}_{n+1}$---{\em note the change to $\phi_n^\vee$ to simplify our later exposition (cf. (\ref{eqn def inv vee}))}. Let $P$---not to be confused with the mirabolic $P_n$---be the subgroup $P$ of $H_{n+1}(E)$ consisting of elements 
$$
\left(\begin{array}{cccc}* & *&*&*\\ *&...&*&*\\
0&0 & 1&*\\0&0 &0 & 1
\end{array}\right)\in H_{n+1}(E)\subset M_{n+1}(E).
$$
The triple $(\varphi_{n-1},\phi_{n-1},\phi_{n})$ then defines a function $\Psi$ on $P$ by
\begin{align}
\Psi\left[\left(\begin{array}{cc} \begin{array}{cc} x& u\\ &1
\end{array}& u'\\ &1
\end{array}\right)
\right]:=\varphi_{n-1}(x)\phi_{n-1}(u)\phi_n(u'),
\end{align}
where $x\in H_{n-1}(E), u\in M_{n-1,1}(E),u'\in M_{n,1}(E)$.
For simplicity, we write 
\begin{align}\label{eqn def f^psi}
f^\Psi=f_{n}^{\varphi_{n-1},\phi_{n-1}}\otimes f^{\phi_{n}^\vee}_{n+1}.
\end{align}
We also consider the Lie algebra $\fkp$ of $P$, consisting of:
$$
\left(\begin{array}{cccc}* & *&*&*\\ *&...&*&*\\
0&0 & 0&*\\0&0 &0 & 0
\end{array}\right)\in M_{n+1}(E).
$$Both $\fkp$ and $P$ are considered as subsets of $M_{n+1}(E)$.
The map $p\mapsto 1+p$ from $\fkp$ to $P$ defines a local homeomorphism from $\varpi \fkp(\CO_E)$ to its image. The function $\Psi$ is supported in the subgroup of $P(E)$:
$$
P(\varpi\CO)\oplus \fkp(\varpi\CO_E^-)=1+\fkp(\varpi\CO_E).
$$
Note that $1+\fkp(\varpi\CO_E)$ is a compact subgroup of $M_{n+1}(E)$ and hence unimodular. So both the left and the right invariant measure on $P(E)$ restrict to a Haar measure on it. 
 Let $(\varphi_{n-1},\phi_{n-1},\phi_{n})$ be $(m,m')$-admissible. We may write 
$\phi_i=\phi^+\otimes \phi^-_i$ ($i=n-1,n$) according to the decomposition $M_{i,1}(E)=M_{i,1}(F)\oplus M_{i,1}(E^-) $
and similarly for $\varphi_{n-1}$ viewed as a function on $M_{n-1}(E)$. Write $\fkp=\fkp^+\oplus \fkp^-$ for $\fkp^\pm=\fkp\cap M_{n+1}(E^\pm)$. Then we define $\Psi^+$ and $\Psi^-$ as functions on $1+\fkp^+(\CO_F)$ and $\fkp^-$, respectively:
$$
\Psi^?=\varphi_{n-1}^?\otimes \phi_{n-1}^?\otimes \phi_n^{?},\quad ?=\pm,$$
as follows:
\begin{align}
\Psi^+\left[\left(\begin{array}{cc} \begin{array}{cc} x& u\\ &1
\end{array}& u'\\ &1
\end{array}\right)
\right]:=\varphi_{n-1}^+(x)\phi_{n-1}^+(u)\phi_n^+(u'),
\end{align}
where $x\in H_{n-1}(F), u\in M_{n-1,1}(F),u'\in M_{n,1}(F)$, and
\begin{align}
\Psi^-\left[\left(\begin{array}{cc} \begin{array}{cc} x& u\\ &0
\end{array}& u'\\ &0
\end{array}\right)
\right]:=\varphi_{n-1}^-(x)\phi_{n-1}^-(u)\phi_n^-(u'),
\end{align}
where $x\in H_{n-1}(E^-), u\in M_{n-1,1}(E^-),u'\in M_{n,1}(E^-)$. 	We have
\begin{align}\label{eqn Psi =+-}
\Psi(p_++p_-)=\Psi^+(p_+)\Psi^-(p_-), \quad p_+\in 1+\fkp^+,p_-\in \fkp_-.
\end{align}

\begin{lem}Assume that
$f^\Psi=f_{n}^{\varphi_{n-1},\phi_{n-1}}\otimes f^{\phi_{n}^\vee}_{n+1}$ is admissible. Then the function $\Psi^-$ has the following invariance property
\begin{align}\label{eqn inv Psi}
\Psi^-(p_+p_-)=\Psi^-(p_-)
\end{align}
 whenever $p_+\in supp(\Psi^+)$. 
 \end{lem}
 \begin{proof}
 By the admissibility (\ref{def adm}), we may and will assume that $\Psi^+$ is a multiple of the characteristic function of the subset of $1+\fkp^+(\CO_F)$ consisting of the matrices $(x_{ij})$ where $x_{ij}\equiv \delta_{ij} ({\rm mod}\,\varpi^m)$ when $j\leq n$ and  $x_{i,n+1}\equiv \delta_{i,n+1} ({\rm mod}\,\varpi^{m'})$. 
Any $p^+\in supp(\Psi^+)$ is of the form
$$ \left(\begin{array}{cc} \begin{array}{cc} x& u\\ &1
\end{array}& u'\\ &1
\end{array}\right),\quad x\in 1+\varpi^mM_{n-1}(\CO_F), u\in \varpi^mM_{n-1,1}(\CO_F). $$
To show (\ref{eqn inv Psi}), we may assume that $p_-\in supp(\Psi^-)$. 
Now we note that
 $$
p_+p_-= \left(\begin{array}{cc} \begin{array}{cc} x& u\\ &1
\end{array}& u'\\ &1
\end{array}\right) \left(\begin{array}{ccc}  x_-& u_-& v_-\\
&0& v'_-\\ &&0
\end{array}\right)= \left(\begin{array}{ccc}  xx_-& x u_-& xv_-+uv'_-\\
&0& v'_-\\ &&0
\end{array}\right).
 $$
 Now the invariance follows from Definition \ref{def dag n1} (for $\phi_n$), and Property (iii), (iv) of $\varphi_{n-1}\otimes \phi_{n-1}$ in Prop. \ref{prop pro adm}.
  \end{proof}

\begin{lem}\label{lem f_nat}
Fix $\Pi$ and assume that
$f^\Psi=f_{n}^{\varphi_{n-1},\phi_{n-1}}\otimes f^{\phi_{n}^\vee}_{n+1}$ is admissible for $\Pi$. Then we have
$$
f^\Psi_\nat(X)=c(\Psi^+)\int_{\fkp^-}f_\nat (X+p)\Psi^-(p)\,dp,
$$
where $f_\nat$ is the function on $\fks$ associated to $f_n\otimes f_{n+1}$ and 
\begin{align}\label{eqn con c}
c(\Psi^+)=\int_{1+\fkp^+(\varpi\CO)}\Psi^+(p)\,dp
\end{align}
is a constant.\end{lem}

\begin{proof}We consider the case when $n$ is odd; the case $n$ even is similar and only require us to change notations in several places.
By definition we have 
$$\wt{f^\Psi}=(f_n\ast \phi_{n-1}\ast \varphi_{n-1})^\vee \ast f_{n+1}^{\phi^\vee_n}= \varphi_{n-1}^\vee\ast \phi_{n-1}^\vee\ast f_n^\vee \ast f_{n+1}^{\phi^\vee_n}.$$
Since $f_i$ ($i=n,n+1$) is a multiple of $1_{1+\varpi^r_EM_{i} (\CO_E)}$ for some and $\int f(g)dg=1$, we may assume that $f_{n}^\vee\ast f^{\phi_{n}}_{n+1}$ is the same as $f^{\phi_{n}}_{n+1}$:
$$\wt{f^\Psi}=\varphi_{n-1}^\vee\ast \phi_{n-1}^\vee\ast f_{n+1}^{\phi^\vee_n}=\varphi_{n-1}^\vee\ast \phi_{n-1}^\vee\ast \phi_n^\vee\ast f_{n+1}.$$
Explicitly, this reads
\begin{align*}
\wt{f^\Psi}(g)=&\int \varphi_{n-1}(x)\phi_{n-1}(u)f^{\phi_{n}^\vee}_{n+1}(uxg)\,dx\,du
=\int_{P}\Psi(p)f_{n+1}(pg)\,dp.\end{align*}
Let's denote the right hand side by $f^\Psi_{n+1}(g)$. By our choice of $f_{n+1}$ also implies that $f_{n+1}$ is conjugate invariant under $1+\varpi M_{n+1}(\CO_E)$:
\begin{align}\label{eqn inv fn1}
f_{n+1}(hgh^{-1})=f_{n+1}(g),\quad h\in 1+\varpi M_{n+1}(\CO_E).
\end{align}
By the support condition,  $f^\Psi_\nat(X)$ vanishes unless $X\in \fks(\CO_F)$. We thus assume that $X\in \fks(\CO_F)$. Then by definition and the support condition of $\Psi$, we have:
\begin{align*}
f^\Psi_\nat(X)=&\int_{H_{n+1}(F)}\int_{P(E)}\Psi(p)f_{n+1}(p(1+X)h)\,dp\,dh
\\=&\int_{H_{n+1}(F)}\int_{\fkp^-}\int_{P(F)}\Psi(p_+(1+p_-))f_{n+1}(p_+(1+p_-)(1+X)h)\,dp_+\,dp_-\,dh.
\end{align*}
Note that $\Psi(p_+(1+p_-))=\Psi^+(p_+)\Psi^-(p_+p_-)$ (cf. (\ref{eqn Psi =+-})) and $\Psi^-(p_+p_-)=\Psi^-(p_-)$ (cf. (\ref{eqn inv Psi})). Together with the conjugate invariance (\ref{eqn inv fn1}), we have
\begin{align*}
f^\Psi_\nat(X)=&\int_{H_{n+1}(F)}\int_{\fkp^-}\int_{P(F)}\Psi^+(p_+)\Psi^-(p_-)f_{n+1}((1+p_-)(1+X)hp_+)\,dp_+\,dp_-\,dh
\\=&\int_{H_{n+1}(F)}\int_{\fkp^-}\int_{P(F)}\Psi^+(p_+)\Psi^-(p_-)f_{n+1}((1+p_-)(1+X)h)\,dp_+\,dp_-\,dh
\\=&c(\Psi^+)\int_{H_{n+1}(F)}\int_{\fkp^-}\Psi^-(p_-)f_{n+1}((1+p_-)(1+X)h)\,dp_-\,dh.
\end{align*}
Note that
$$
(1+p_-)(1+X)=(1+p_-X)(1+(1+p_-X)^{-1}(p_-+X))
$$
and $(1+p_-X)\in 1+\varpi M_{n+1} (\CO_F)\subset H_{n+1}(F)$. We have
\begin{align*}
&\int_{H_{n+1}(F)}\int_{\fkp^-}\Psi^-(p_-)f_{n+1}((1+p_-)(1+X)h)\,dp_-\,dh
\\=&\int_{H_{n+1}(F)}\int_{\fkp^-}\Psi^-(p_-)f_{n+1}((1+(1+p_-X)^{-1}(p_-+X))h)\,dp_-\,dh.
\end{align*}
Compared to the definition of $f_\nat$ for $f=f_n\otimes f_{n+1}$, we obtain\begin{align*}
f^\Psi_\nat(X)=&c(\Psi^+)\int_{\fkp^-}\Psi^-(p_-)f_{\nat}((1+p_-X)^{-1}(p_-+X))\,dp_-.
\end{align*}
Finally note that $f_{\nat}(X)$ is a multiple of $1_{\varpi^r\fks(\CO_F)}$ for some  $r>1$. It follows that $f_{\nat}(X)=f_{\nat}(hX)$  for any $h\in 1+\varpi M_n(\CO_F)$. Since $p_-\in supp(\Psi^-)\subset \fkp^-(\varpi\CO)$ and $X\in \fks(\CO)$, we therefore obtain
\begin{align*}
f^\Psi_\nat(X)=c(\Psi^+)\int_{\fkp^-}\Psi^-(p_-)f_{\nat}(p_-+X)\,dp_-.
\end{align*}
This completes the proof.
\end{proof}

\subsection{Local constancy of the orbital integral and a formula for the regular nilpotent orbital integral.}
To compare with the unitary group case  in the next section, we need to understand the orbital integral of $\wh{f^\Psi_\nat}$ in Lemma \ref{lem f_nat}, at least around zero. We show that the orbital integral is locally constant around zero. This constant is essentially given by the regular unipotent orbital integral (for $\xi_-$ defined by (\ref{eqn xi- s_n})) of Fourier transform $\wh{f^\Psi_\nat}$ of $f^\Psi_\nat$ on $\fks$. 

\begin{lem}\label{lem fml mu f}

Fix any $m$-admissible function $(\varphi_{n-1},\phi_{n-1})$. 
\begin{itemize}
\item[(1)]
For an arbitrarily large compact neighborhood $\sZ$ of $0\in \fks$, there exists large enough $(m,m',r)$ and an $(m,m',r)$-admissible function $f^\Psi=f_{n}^{\varphi_{n-1}, \phi_{n-1}}\otimes f^{\phi_{n}^\vee}_{n+1}$, such that the orbital integral $\eta'(\Delta_-(X))O(X,\wh{f^\Psi_\nat})$ is a (nonzero) constant for regular semisimple $X\in \sZ$ and and this constant is equal to $\eta'(\Delta_-(\xi_-))\mu_{\xi_{-}}(\wh{f^\Psi_\nat})$. 

\item[(2)]Let $f^\Psi$ be as in (1).
Then the regular unipotent orbital integral $\mu_{\xi_{-}}(\wh{f^\Psi_\nat})$ is equal to
\begin{align*}
&c(\Psi^+)|\ov{\delta}_{n,E}(\epsilon_n)|^{-1/2} \prod_{i=1}^{n-1}\phi^{'-}_{i}(0)
\\\times& \int_{A_n(F)N_{n,-}(F)} \prod_{i=1}^{n} \wh{\phi^-_{i}}(-a_i(v_{i-1},1)\tau)|\ov{\delta}_n(a)|^{-1} \eta(a)\,d^\ast a\,dv,
\end{align*}
where $v_{i}\in M_{i,1}(F)$, and $\ov{\delta}_n(a)$ is defined by (\ref{def del bar}).
\end{itemize}

\end{lem}
\begin{proof}Without loss of generality 
we may assume $c(\Psi^+)=1$. Assume that $f^\Psi$ is $(m,m',r)$-admissible.  By Lemma \ref{lem f_nat}, when $r$ is large enough, the function $f^\Psi_\nat$ on $\fks_{n+1}$ is of the form $\Psi^-\otimes \phi^{'-}_n\otimes \phi^{-}_{n+1}$ corresponding to the decomposition
$$
\fks_{n+1}=\fkp^-\oplus M_{1,n}(E^-) \oplus M_{1,n+1}(E^-).
$$
More precisely we may write the function on $\fks$ defined by
$$X\mapsto f^\Psi_\nat(-X)
$$
in the following form:
\begin{center}
\[\begin{tabular}{| l  ccc|r| }
  \hline   
   \multicolumn{1}{|r}{$\phi^{'-}_1$}  &     \multicolumn{1}{|r|}{$\phi^-_1$} & \multicolumn{1}{r|}{$\phi_2^-$} & $\cdots$&  \multicolumn{1}{r|}{$\phi_n^{-}$} \\ \cline{1-2}
   \multicolumn{1}{|r}{$\phi_2^{'-}$} &   \multicolumn{1}{r|}{} & \multicolumn{1}{r|}{} & &  \\ \cline{1-3}
  $\vdots$&& \multicolumn{1}{r|}{}& &  \\  \cline{1-4}
 $\vdots$  &&$\ddots$&&\\  \hline
    \multicolumn{1}{|r}{ $\phi_{n+1}^{'-}$}&\multicolumn{1}{r}{} &\multicolumn{1}{r}{}&\multicolumn{1}{r}{}& \multicolumn{1}{r|}{} \\ \hline
\end{tabular}\]
\end{center}
We also write $f^\Psi_\nat$ as the tensor product 
$$\varphi^-_n\otimes \phi^-_n\otimes \phi_{n+1}^{'-},
$$
where $\varphi^-_n\in \sC_c^\infty(\fks_n)$. We recall their key properties for our computation below:
\begin{itemize}
\item[(i)] For $i=1,...,n-1$, $\phi_{i}^{'-}$ is a nonzero multiple of the characteristic function of $\varpi^m\CO_{E^-}$; each of $\phi_{n}^{'-}$ and $\phi_{n+1}^{'-}$ is a  nonzero multiple of the characteristic function of $\varpi^r\CO_{E^-}$, normalized so that $\wh{\phi^{'-}_{n}}(0)=\wh{\phi^{'-}_{n+1}}(0)=1$.

\item[(ii)] 
For $i=1,...,n-1$, $\phi_{i}$ is in $\sC^\infty_c(M_{i,1}(E))^\dagger_{m}$;  $\phi_{n}$ is in $\sC^\infty_c(M_{i,1}(E))^\dagger_{m'}$.

\item[(iii)] the function $\wh{\varphi^-_n}$ on $\fks_n$ is invariant under left and right multiplication by $1+\varpi^m M_{n-1}(\CO_F)$ (as a subgroup of $H_{n-1}(F)$, and hence of $H_n(F)$).
\end{itemize}

Back to the orbital integral, we may fix an arbitrarily large compact neighborhood $\sZ$ of $0$ in the quotient of $\fks_{n+1}$ by $H_n$. We use the following regular elements (cf. \S6, (\ref{eqn varrho})):
\begin{align}
\varrho(x,y)=\tau\left(
\begin{array}{ccccc}0 &...&0 &x_{n}  &y_n \\
1 &   0&...&x_{n-1}& y_{n-1}\\ 0&1&0&... &y_{n-2}
\\ ...&0&1&x_1&...\\
0&...&0&1&y_0\end{array}\right)\in \fks_{n+1},
\end{align}
where $(x,y)\in F^{2n+1}$ and we may assume that $\sZ$ is a compact neighborhood of $0$ in $F^{2n+1}$.
We also denote
\begin{align}
\varrho(x)=\tau\left(
\begin{array}{ccccc}0 &\cdots &0 &x_{n}   \\
1 &   0&\vdots &x_{n-1}\\ 0&1&0&\vdots 
\\ \vdots &0&1&x_1\end{array}\right)\in \fks_{n}.
\end{align}
It suffices to show that, when we increase $m,m',r$ suitably, the orbital integral of $\varrho(x,y)$ is a constant when $(x,y)$ lies in the fixed compact set $\sZ$.
We proceed in three steps. 

{\em Step 1.}
To ease notation, we denote $$f'=\wh{f^\Psi_\nat}.
$$ By definition we have
\begin{align}\label{eqn orb varrho 1}
O(\varrho(x,y),\wh{f^\Psi_\nat})=\int_{H_n(F)} \wh{f^\Psi_\nat}(h^{-1}\varrho(x,y) h)\eta(h)\,dh=\int_{H_n(F)} f'(h^{-1}\varrho(x,y) h)\eta(h)\,dh.
\end{align}
We may write $h\in H_n'(F)$ as
$$
h= a_n k v, \quad k\in  N_{n}H_{n-1},v=\left(\begin{array}{cc}1_{n-1} &\\ v_{n-1}& 1
\end{array}\right),a_n\in F^{\times}.
$$
Then the last row of $h^{-1}\xi_- h$ is of the form
$$
( a_n(v_{n-1},1),y_{0})\tau.
$$
For the integrand $ f'(h^{-1}\varrho(x,y) h)$ to be nonzero,  $a_n(v_{n-1},1)$ must be  in the support of $\wh{\phi^-_n}$. Since $\phi_n\in \sC^\infty_c(M_{n,1}(E))^\dagger_{m'}$), we may assume that (cf. (\ref{eqn supp cond phi n}))
$$\|v_{n-1}\|\leq |\varpi^{m'}|.
$$

Now we {\em claim that if $h^{-1}\varrho(x)h$ lies in the support of $\wh{\varphi_{n}^-}$, then the last column of $h^{-1}\varrho(x) h\in \fks_n$ is bounded by a polynomial of the norms of $x_i$ ($1\leq i\leq n$) and $|\varpi|^{-m}$, independent of $r$.} To show this, suppose that $g=h^{-1}\varrho(x)h$ lies in the support of $\wh{\varphi_{n}^-}$. The same argument as in the previous paragraph shows that we may find $v=\left(\begin{array}{cc}1_{n-2} &\\ v_{n-2}& 1
\end{array}\right)\in N_{n-1,-}\in N_{n,-}$ with $\|v_{n-2}\|\leq |\varpi|^m$, such that 
$v^{-1}gv$ is of the form
$$
\left(
\begin{array}{ccccc}\ast & \ast &   \ast & \ast&\ast\\ \ast&\ast& \ast&\ast&\ast\\ \ast&\ast&\ast& \ast&\ast\\
\ast& \ast&\ast&\ast& \ast\\
0 &0 &0&\alpha_{n-1}& \ast
\end{array}\right).
$$Since $\wh{\varphi_{n}^-}$ is invariant under both left and right multiplication by $1+\varpi^mM_{n-1}(\CO_F)$, the above  $v^{-1}gv$ again lies in the support of $\wh{\varphi_{n}^-}$.
Continuing this process, we may find an element $v$ of $N_{n-1,-}(F)\subset N_n(F)$ whose off-diagonal entries all lie in $\varpi^{m}\CO_F$, such that $v^{-1}gv$ is of the form
$$
\left(
\begin{array}{ccccc}\ast & \ast &   \ast & \ast&\ast\\ \alpha_1&\ast& \ast&\ast&\ast\\ 0&\alpha_2&\ast& \ast&\ast\\
0 &0 &\ddots&\ast& \ast\\
0 &0 &0&\alpha_{n-1}& \ast
\end{array}\right)
$$
and remains in the support of $\wh{\varphi_{n}^-}$.
Since the functions $\phi_{i}$, $1\leq i\leq n-1$, lie in $\sC^{\infty}_c(M_{i,1}(E))_m^\dagger$, all $|\alpha_i|$ ($1\leq i\leq n-1$) are equal to some constant multiple of $|\varpi|^{-2m}$. Note that the $x_i$'s are the coefficients of the characteristic polynomial of $g$ and hence of $v^{-1}gv$. 
Now by Corollary \ref{cor bdd} (taking transpose), it is easy to see that each entry of the last column of $v^{-1}gv$ is a polynomial of the first $n-1$ columns of $v^{-1}gv$, $\alpha_i,\alpha^{-1}_i$'s and the $x_i$'s. This shows that the claim holds for the last column of $v^{-1}gv$ and hence also for $g$ itself since the off-diagonal entries of $v$ is bounded by $|\varpi|^{m}$.  This proves the claim.

Now by the claim, there exists large enough $m'_0$ and $r_0$ such that once $m'\geq m'_0$ and $r\geq r_0$, we have the following invariance property when $(x,y)\in \sZ$:
\begin{align}
f'(h^{-1}\varrho(x,y)h)=\wh{\varphi^-_n}(-k^{-1}\varrho(x)k)\wh{\phi_n^-}(-a_n(v_{n-1},1)\tau)\wh{\phi^{'-}_{n+1}}(h^{-1}y\tau)
\end{align}
if the left hand side is not zero (i.e., at least $h^{-1}\varrho(x,y)h$ lies in the support of $f'$).

{\em Step 2.}
We now may repeat the process and utilize the property (iii) (i.e., the invariance of $\wh{\varphi^-_n}$ under $1+\varpi^m M_{n-1}(\CO_F)$). We write
$$
h= uav,  a= \left(\begin{array}{ccc}a_1a_2\cdots a_n& &\\&\ddots&\\
& &  a_n
\end{array}\right), u\in N_{n}(F),
$$
and write $v\in N_{n-}(F)$ as the product of $\left(\begin{array}{cc}1_{n-i} &\\ v_{n-i}& 1
\end{array}\right)\in N_{n-i+1,-}(F)\subset N_{n-}(F)$ for $1\leq i\leq n-1$. In Step 1 we have seen that $\|v_{n-1}\|\leq |\varpi|^m{'}.$
Since the functions $\phi_{i}$, $1\leq i\leq n-1$, lie in $\sC^{\infty}_c(M_{i,1}(E))_m^\dagger$, and by the property (iii), we may inductively show that:
$$
\|v_{i}\|\leq |\varpi|^m,\quad 1\leq i\leq n-2.
$$
We now view $\bigotimes_{i=1}^{n+1} \wh{\phi^{'-}_{i}}$ as a function on the set of upper triangular elements and then consider it as a function on $\fks_{n+1}$ via the natural projection from $\fks_{n+1}$ to the upper triangular elements. Then, when $(x,y)\in \sZ$, the orbital integral $O(\varrho(x,y),\wh{f^\Psi_\nat})$ is equal to (cf. (\ref{eqn orb varrho 1})):
\begin{align}\label{eqn orb xy}
\int_{}\left(\bigotimes_{i=1}^{n+1} \wh{\phi^{'-}_{i}} \right)(-(ua)^{-1}\varrho(x,y) ua)\,du \prod_{i=1}^{n} \wh{\phi^-_{i}}(-a_i(v_{i-1},1)\tau)|\delta_n(a)|^{-1}\eta(a)\,d^\ast a\,dv,
\end{align}
where $u\in N_n(F),v\in N_{n,-}(F)$.

{\em Step 3.}
 Note that $u^{-1}\varrho(x,y)u$ is of the form
$$
u^{-1}\varrho(x,y)u=\tau\left(
\begin{array}{ccccc}\ast & \ast &   \ast & \ast&\ast\\ 1&\ast& \ast&\ast&\ast\\ 0&1&\ast& \ast&\ast\\
0 &0 &1&\ast& \ast\\
0 &0 &0&1& \ast
\end{array}\right).
$$
By Lemma \ref{lem a sect}, we may make a substitution  
to replace the integral over $u$ in (\ref{eqn orb xy}) by an integral over $u'$ of elements of the form
\begin{align}
u'=\left(
\begin{array}{cccc}\ast & \ast &   \ast &0\\ 1&\ast&\ast&0\\ 0&1&\ast&0\\
0 &0 &1&0
\end{array}\right)\in H_n(F),
\end{align}
and the measure $du'$ is induced by $du$:
\begin{align*}&
\int_{}\left(\bigotimes_{i=1}^{n+1} \wh{\phi^{'-}_{i}} \right)((ua)^{-1}\varrho(x,y)ua)\,du
\\
=&\int_{}\left(\bigotimes_{i=1}^{n+1} \wh{\phi^{'-}_{i}} \right)\left(a^{-1}\left(
\begin{array}{cccc} u' &   \ast &\ast\\ 1&\ast&\ast\\ 0&1&\ast\end{array}\right)
a\tau\right)\,du',
\end{align*}
where the last two columns are polynomials of entries of $u'$, $x_i,y_j$ and $n$, by Lemma \ref{lem a sect} (taking transpose). 
Now we may increase $r$ suitably (i.e., increase the support of $\wh{\phi^{'-}_{n}}$ and $\wh{\phi^{'-}_{n+1}}$) so that the constraints of the last columns on $u'$ is superfluous. In particular,  there exists $r_1>r_0$ large enough (depending on $m,m'$ and $\sZ$) such that when $r>r_1$ we have
\begin{align*}&
\int_{}\left(\bigotimes_{i=1}^{n+1} \wh{\phi^{'-}_{i}} \right)((ua)^{-1}\varrho(x,y)ua)\,du
\\
=&\wh{\phi^{'-}_{n+1}}(0)\wh{\phi^{'-}_{n}}(0)\int_{}\left(\bigotimes_{i=1}^{n-1} \wh{\phi^{'-}_{i}} \right)\left(a^{-1}u' a\tau\right)\,du'.
\end{align*}
This is independent of $(x,y)\in \sZ$. By (\ref{eqn orb xy}), we conclude that  the orbital integral $O(\varrho(x,y),\wh{f^\Psi_\nat})$ is a constant and the constant is equal to the regular unipotent orbital integral $O(\varrho(0,0),\wh{f^\Psi_\nat})=\mu_{\xi_-}(\wh{f^\Psi_\nat})$. This finishes the proof the first part of the lemma.

To prove the second part of the lemma, it remains to evaluate the  regular unipotent orbital integral $\mu_{\xi_-}(\wh{f^\Psi_\nat})$. First we note that $\wh{\phi^{'-}_{n}}(0)=\wh{\phi^{'-}_{n+1}}(0)=1$ by our normalization. We make the substitution $u'\mapsto au'a^{-1}=Ad(a)u'$ (equivalently, $u'_{ij}\mapsto u'_{ij}\prod_{i<\ell\leq j}a_{\ell}$ for $1\leq i\leq n-1$).
This yields
\begin{align*}
&\int_{}\left(\bigotimes_{i=1}^{n-1} \wh{\phi^{'-}_{i}} \right) (a^{-1}u'a\tau )\,du'\\
=&\det(Ad(a):\fkn_{n-1})\int_{}\left(\bigotimes_{i=1}^{n-1} \wh{\phi^{'-}_{i}} \right) ( u' \tau)\,du'\\
=&\det(Ad(a):\fkn_{n-1})|\tau|_E^{-(1+2+...+(n-2))/2} \prod_{i=1}^{n-1}\phi^{'-}_{i}(0).
\end{align*}
(Or more explicitly $\det(Ad(a):\fkn_{n-1})=\prod_{i=1}^{n-1}|a_i|^{(i-1)(n-i)}$.) Here the factor
$$
|\tau|_E^{(1+2+...+(n-2))/2}=|\ov{\delta}_{n,E}(\epsilon_n)|^{1/2}
$$
is caused by the difference between the measures on $F\tau$ and $E^-$. We thus proved that  $\mu_{\CO_0}(\wh{f_\nat})$ is equal to
$$|\ov{\delta}_{n,E}(\epsilon_n)|^{1/2} \prod_{i=1}^{n-1}\phi^{'-}_{i}(0)
\int_{A_n(F)N_{n,-}(F)} \prod_{i=1}^{n} \wh{\phi^-_{i}}(-a_i(v_{i-1},1)\tau) |\ov{\delta}_n(a)|^{-1}\eta(a)\,d^\ast a\,dv.
$$
(Or more explicitly, $\ov{\delta}_n(a)=\det(Ad(a):\fkn_{n}))/\det(Ad(a):\fkn_{n-1})=\prod_{i=1}^{n-1}|a_i|^{n-i}$.)

\end{proof}

\subsection{Proof of Theorem \ref{thm germ gl}.}
We choose an admissible function $f^\Psi=f_{n}^{\varphi_{n-1}, \phi_{n-1}}\otimes f^{\phi_{n}^\vee}_{n+1}$ for $\Pi$ so that it also verifies the conditions Lemma \ref{lem fml mu f}. We decompose $\varphi_{n-1},\phi_{n-1}$ as in Lemma \ref{lem W_n fml} and similarly $\phi_n$. Note that \begin{align*}
&\int_{} \prod_{i=0}^{n-1}\wh{\phi_{n-i}}(- y_i(v_{n-i-1},1)\tau)|\ov{\delta}_{n}(y)|^{-1}\eta(y)\,d^\ast y\,dv,
\\=&\left(\prod_{i=0}^{n-1}\wh{\phi_{n-i}^+}(0)\right)\int \prod_{i=0}^{n-1}\wh{\phi^-_{n-i}}(- y_i(v_{n-i-1},1)\tau)|\ov{\delta}_{n}(y)|^{-1}\eta(y)\,d^\ast y\,dv,
\\=&\left(\prod_{i=1}^{n}\int_{M_{i,1}(F)} \phi_{i}^{+}(x_i)\,dx_i \right)\int \prod_{i=0}^{n-1}\wh{\phi^-_{n-i}}(- y_i(v_{n-i-1},1)\tau)|\ov{\delta}_{n}(y)|^{-1}\eta(y)\,d^\ast y\,dv,
\end{align*}
where $y\in A_n(F), v \in N_{n,-}(F)$ are as in (\ref{eqn for y00}) and (\ref{eqn for v00}).
Moreover,
$$
\int_{B_{n-1,-}(F)}\phi'( b)\,db=\prod_{i=1}^{n-1}\phi_{i}^{'-}(0)\prod_{i=1}^{n-1}\int_{M_{1,i}(F)}\phi_{i}^{'+}(b_i)\,db_i.
$$
It is clear that the constant in (\ref{eqn con c}) is given by
$$
c(\Psi^+)=\prod_{i=1}^{n}\int _{M_{i,1}(F)}\phi_{i}^{+}(x_i)\,dx_i\prod_{i=1}^{n-1}\int_{M_{1,i}(F)}\phi_{i}^{'+}(b_i)\,db_i.
$$ 
Note that $$
|\tau|_E^{(d_n+d_{n+1})/2}=|\delta_{n-1,E}(\epsilon_{n-1})|^{1/2}|\delta_{n,E}(\epsilon_{n})|^{1/2}=|\delta_{n-1,E}(\epsilon_{n-1})||\ov{\delta}_{n,E}(\epsilon_n)|^{1/2}.
$$
Then the identity in Theorem \ref{thm germ gl} follows by comparing Prop.
\ref{prop fml I Pi} with Lemma \ref{lem fml mu f}. Moreover, we may choose such admissible function with arbitrarily small support by increasing $(m,m',r)$ and such that $\mu_{\xi_-}(f^\Psi_\nat)\neq 0$.

\section{Local character expansion in the unitary group case }
\subsection{Three ingredients from \cite{Zh2}.}
Let $F$ be a non-archimedean local field of characteristic zero. We need to recall the main local results of \cite{Zh2}. There are two isomorphism classes of Hermitian spaces $W_1,W_2$ of dimension $n$. Denote by $H_{W_i}=U(W_i)$ the unitary group. We let $V_i=W_i\oplus Ee$ be the orthogonal sum of $W_i$ and a one-dimensional space $Ee$ with norm $\pair{e,e}=1$. Denote by $\fku_i$ the Lie algebra of $U(V_i)$. We have a bijection of regular semisimple orbits (cf. (\ref{eqn orbit match}) and \cite[\S3.1]{Zh2})
\begin{align*}
H_n(F)\bs \fks(F)_{rs}\simeq H_{W_1}(F)\bs \fku_{1}(F)_{rs}\coprod H_{W_2}(F)\bs \fku_{2}(F)_{rs}.
\end{align*}
A regular semisimple $X\in \fks$ matches some $Y\in \fku_i$ if and only if 
\begin{align}\label{eqn match con}
\eta(\Delta(X/ \tau))=\eta(\disc(W_i)),
\end{align} where $disc(W_i)\in F^\times/\RN E^\times$ is the discriminant of $W_i$. For an $f'\in  \sC^\infty_c(\fks)$ and a pair $(f_1,f_2)$, $f_i\in \sC^\infty_c(\fku_i)$, we say that $f'$ matches $(f_1,f_2)$, if for all matching regular semisimple $X\in \fks, Y\in\fku_i$, we have (cf. (\ref{def O(X,f,s)}))
\begin{align}\label{def sm mat lie}
 \eta'(\Delta_+(X)) O(X,f')=O(Y,f_i),
\end{align} where $\eta'$ is a fixed choice of character $E^\times\to\BC^\times$ with restriction $\eta'|_{F^\times}=\eta$.

Let $W\in\{W_1,W_2\}$. Analogous to the general linear group case (cf. Definition \ref{def f_nat}), to a function in a small neighborhood of $1\in G=U(W)\times U(V)$, we associate a function on the Lie algebra $\fku$ of $U(V)$. For $f=f_n\otimes f_{n+1}\in \sC^\infty_c(G)$, we let $\wt{f}$ be the function on $U(V)$ defined by
$$
\wt{f}(g):=\int_{U(W)}f_n(h)f_{n+1}(hg)\,dh,\quad g\in U(V).
$$
If $f$ is supported in a small neighborhood $\Omega$ of $1$ in $G$, then $\wt{f}$ is supported in a small neighborhood $\wt{\Omega}$ of $1$ in $U(V)$. Since the Cayley map $\fkc:\fku\to U(V)$ is a local homeomorphism around $0\in \fku$, we may denote $\omega =\fkc^{-1}(\wt{\Omega})\simeq \wt{\Omega}$ and to $\wt{f}$ we associate a function denoted by $f_\nat=\fkc^{-1}(\wt{f})$ on $\fku$ supported in $\omega$.
To connect the smooth transfer on the groups to the one on Lie algebras, we need:
\begin{lem}
Let $f'\in  \sC^\infty_c(G'(F))$ and  $f_i\in \sC^\infty_c(U(W_i)\times U(V_i))$ ($i=1,2$) be matching functions (in the sense of \S4, (\ref{eqn mat grp})) with support in a neighborhood of the identity where the Cayley map is well-defined. Then the functions
$f'_\nat \in \sC^\infty_c(\fks)$ and $f_{i,\nat}\in \sC^\infty_c(\fku_i)$ ($i=1,2$) match.
\end{lem}
\begin{proof}The support condition ensures that the associated functions  $f'_\nat, f_{i,\nat}$ are well-defined. Then it remains to show the transfer factor are compatible:
$$
\eta'(\Delta_+(X)) O(X,f'_{\nat})=\Omega(g) O(g,f'),
$$
where $\nu(g)=\fkc(X)\in S_{n+1}(F)$.
This follows from the proof of \cite[Lemma 3.5]{Zh2}.
\end{proof}

Now we only consider one Hermitian space $W\in\{W_1,W_2\}$, with the corresponding groups $U(W),U(V)$, and the Lie algebra $\fku$. We say that $f\in \sC^\infty_c(\fku)$ and $f'\in  \sC^\infty_c(\fks)$ match if the equality (\ref{def sm mat lie}) holds for all regular semisimple $X$ matching $Y\in \fku$.

In \cite[Theorem 2.6]{Zh2} the following result is proved:
\begin{thm}
For any $f\in \sC^\infty_c(\fku)$ there exists a matching $f'\in  \sC^\infty_c(\fks)$ and conversely. 
\end{thm}

Moreover we have \cite[Theorem 4.17]{Zh2}:
\begin{thm}\label{thm comp F}
If the functions $f$ and $f'$ match, then so do $\epsilon(1/2,\eta,\psi)^{n(1+n)/2}\wh{ f}$ and $\wh{f'}$.
\end{thm}

An important ingredient of the proof of both theorems above is a local relative trace formula on Lie algebra (\cite[Theorem 4.6]{Zh2}). Now we only need the one in the unitary group case .

\begin{thm}\label{thm local TF}
For $f_1,f_2\in \sC_c^\infty(\fku)$, we have
$$
\int_{\fku}f_1(X)O(X,\wh{f_2})\,dX=\int_{\fku}O(X,\wh{f_1})f_2(X)\,dX,
$$
where the integrals are absolutely convergent.
\end{thm}

\subsection{A hypothesis.} We now return to the local spherical character in the unitary group case . Let $\pi$ be an irreducible admissible representation of $G=U(V)\times U(W)$. We use the measure on $U(V)$ determined by the self-dual measure on $\fku$ via the Cayley map. We call a subset $\Omega\subset G$ a {\em $U(W)\times U(W)$-domain} (associated to $\omega$) if  there is an open and closed subset $\omega$ in the $F$-points $(H\bs \fku)(F)$ of categorical quotient $H\bs \fku$\footnote{In our case, the categorical quotient  $H\bs \fku:=\Spec\, \CO_{\fku}^H$ is an affine space and the natural morphism $\fku\to H\bs \fku$ induces a continuous map on the $F$-points: $\fku(F)\to  (H\bs \fku)(F)$.} such that 
\begin{itemize}
\item the Cayley map is defined on  the preimage of $\omega$  in $\fku$ and takes    the preimage of $\omega$  to $\Omega'\subset U(V)$.
\item $\Omega$ is the preimage of $\Omega'$ under the contraction map $U(W)\times U(V)\to U(V)$ (given by $(g_n,g_{n+1})\mapsto g_ng_{n+1}$).
\end{itemize}
In particular, $\Omega$ is $U(W)\times U(W)$-invariant, open and closed.

We consider the following:

{\em Hypothesis $(\star)$ for $\pi$}: there exist a neighborhood  $\Omega\subset G$ of $1\in G$ that is a $U(W)\times U(W)$-domain, and a function $\Phi\in \sC_c^\infty(\fku)$, such that 
\begin{align}
\Phi(0)=1,
\end{align} and for all $f\in \sC_c^\infty(\Omega)\subset \sC_c^\infty(G)$,
$$
J_\pi(f)=\int_{\fku}f_\nat(X)O(X,\Phi)\,dX.
$$
 \begin{thm}
\label{thm char u}
Assume that $\pi$ is tempered and $\Hom_{H}(\pi,\BC)\neq 0$.  
Let $\phi$ be a matrix coefficient of $\pi$ such that $$\int_{H}\phi(h)\,dh=1.$$
Then the distribution $J_\pi$ is represented by the orbital integral of $\phi$, as  a function on $G$:
$$
G \ni g\mapsto O(g,\phi).
$$Moreover, the orbital integral $g\mapsto O(g,\phi)$ is a bi-$H$-invariant function that is locally $L^1$ on $G$.
\end{thm}

\begin{proof}
When $\pi$ is tempered and $\Hom_{H}(\pi,\BC)\neq 0$, we have $\alpha\neq 0$, (cf. Property (iii) after (\ref{eqn def alpha})). Hence there exists $\phi$ such that $\int_{H}\phi(h)dh\neq 0$. Up to a scalar multiplication, we may assume that $\int_{H}\phi(h)dh=1.$ Then the theorem is proved  in \cite{IZ}.
\end{proof}

\begin{prop}\label{prop hypo}
Assume that $\Hom_{H}(\pi,\BC)\neq 0$. If the group $H=U(W)$ is compact or $\pi$ is supercuspidal, then
 $\pi$ verifies  Hypothesis $(\star)$.
\end{prop}

\begin{proof}

Assume first that $\pi$ is supercuspidal.  We choose an open and closed neighborhood $\omega$ of $0$ in the categorical quotient of $\fku$. Clearly we may choose such $\omega$ so that the Cayley map is defined on the preimage of $\omega$ in $\fku$. Then the associated $U(W)\times U(W)$-domain $\Omega$ is open and closed neighborhood of $1\in G$.  Let $\phi$ be a matrix coefficient as in Theorem \ref{thm char u}.  We consider $\phi_\Omega=\phi\cdot 1_{\Omega}$ where $1_{\Omega}$ is the characteristic function of $\Omega$. Since $\phi\in \sC^\infty_c(G)$ and $\Omega$ is open and closed, the function $\phi_\Omega$ also lies in  $\sC^\infty_c(G)$.  We now consider the function $\wt{\phi_\Omega}$ which lies $ \sC^\infty_c(U(V))$ and let $\Phi=\phi_{\Omega,\nat}\in  \sC^\infty_c(\fku)$ be the corresponding function on $\fku$ via the Cayley map. It is important to note that we still have
\begin{align}
\Phi(0)=1.
\end{align}
Moreover the measure on $\fku$ is transferred to the measure on $U(V)$. Then $\Phi\in \sC_c^\infty(\fku)$ and for all function $f\in  \sC_c^\infty(G)$ with small support around $1$, 
$$
J_\pi(f)=\int_{\fkg} f_\nat(X)O(X,\Phi)\,dX.
$$

If $H$ is compact (so that $dim\, W\leq 2$ or $E/F=\BC/\BR$),  then there is a non-zero vector $\phi_0\in \pi$ fixed by $H$. Then for all $f$, we have
$$
J_\pi(f)=\vol(H)\int_{G}f(g)\pair{\pi(g) \phi_0, \phi_0}\,dg$$
for a norm one $\phi_0\in \pi^H$. Set $\Phi(g)=\vol(H)^{-1}\pair{\pi(g)\phi_0,\phi_0}$. Then the same truncation as above completes the proof.

\end{proof}

We say that $f$ is admissible if there is an admissible $f'$ matching $f$.
\begin{thm}
\label{thm germ u}Assume that $\pi$ verifies  Hypothesis $(\star)$.
Then there exist an admissible functions $f\in \sC^\infty_c(G(F))$ and a matching function $f'\in \sC^\infty_c(G'(F))$ such that
$$
J_\pi(f)=(\eta'(\tau)/\epsilon(1/2,\eta,\psi))^{n(n+1)/2} \eta(\disc(W))\wh{\mu_{\xi_-}}(f'_\nat)\neq 0.
$$
\end{thm}
\begin{proof}
Suppose that in Hypothesis $(\ast)$ we have a $U(W)\times U(W)$-domain $\Omega$ associated to $\omega$ in the categorical quotient of $\fku$. Since the categorical quotient  of $\fks$ and that of $\fku$ are isomorphic we also view $\omega$ as an open and closed set in the quotient of $\fks$. 
Let $f'\in\sC^\infty_c(G')$ be an $(m,m',r)$-admissible function and $f\in \sC^\infty_c(G)$ be a matching function. We claim that we may choose an $f$ supported in $\Omega$. Indeed we may choose $(m,m',r)$ very large so that the support of $f'$ is very small, say, so that the image of the support of $f'_\nat$ in the categorical quotient of $\fks$ is contained in $\omega$. Now we choose any $f_0$ that matches $f'$. Then we set $f=f_0\cdot 1_{\Omega}$. Clearly the function $f$ has the same orbital integral as $f_0$ and is supported in $\Omega$. This proves the claim.

Now we apply the local trace formula (Theorem \ref{thm local TF})
\begin{align*}
\int_{\fku} f_\nat(Y)O(Y,\Phi)\,dY=
\int_{\fku}O(Y, \wh{f_\nat})\wc{\Phi}(Y)\,dY,
\end{align*}
where $\wc{\Phi}$ is the inverse of the Fourier transform. 
By the compatibility between Fourier transform and smooth transfer (Theorem \ref{thm comp F}) and (\ref{def sm mat lie}), we have
\begin{align}
\epsilon(1/2,\eta,\psi)^{n(n+1)/2}O(Y, \wh{f_\nat})=\eta'(\Delta_+(X))O(X, \wh{f'_\nat})
\end{align}
for matching regular semisimple $X$ and $Y$.
Since $\wc{\Phi}$ has compact support, we may choose a compact neighborhood $\sZ$ of $0\in \fks$ so that the image of $\sZ$ in the quotient $H_n\bs \fks(F)\simeq H\bs \fku(F)$ contains the image of $supp(\wc{\Phi})$. By Lemma \ref{lem fml mu f}, we may choose an admissible functions $f'$ such that $\eta'(\Delta_-(X))O(X, \wh{f'_\nat})$ is equal to a non-zero constant $\eta'(\Delta_-(\xi_-))O(\xi_-, \wh{f'_\nat})$ when $X\in \sZ$.  Thus for regular semisimple $Y\in supp(\wc{\Phi})$ we have:
$$O(Y, \wh{f_\nat})=\epsilon(1/2,\eta,\psi)^{-n(n+1)/2} \eta'(\Delta_+(X)/\Delta_-(X)) \eta'(\Delta_-(\xi_-)) \mu_{\xi_-}(\wh{f'_\nat})\neq 0.
$$
By (cf. (\ref{eqn match con})), we know that
$$\eta(\Delta_+(X)/\Delta_-(X))=\eta(\Delta_+(X/\tau)/\Delta_-(X/\tau))=\eta(\Delta(X/\tau))=\eta(\disc(W)).$$
We note that $\eta'(\Delta_-(\xi_-)) =\eta'(\tau)^{n(n+1)/2}.$
Therefore for all regular semisimple $Y\in supp(\wc{\Phi})$, we have
$$O(Y, \wh{f_\nat})=(\eta'(\tau)/\epsilon(1/2,\eta,\psi))^{n(n+1)/2} \eta(\disc(W))\mu_{\xi_-}(\wh{f'_\nat})\neq 0.
$$
We obtain
\begin{align*}
J_\pi(f)=&\int_{\fku}O(X, \wh{f_\nat})\wc{\Phi}(X)\,dX\\
=&(\eta'(\tau)/\epsilon(1/2,\eta,\psi))^{n(n+1)/2} \eta(\disc(W)) \mu_{\xi_-}(\wh{f'_\nat})\cdot \int_{\fku}\wc{\Phi}(X)\,dX\\
=& (\eta'(\tau)/\epsilon(1/2,\eta,\psi))^{n(n+1)/2} \eta(\disc(W))\mu_{\xi_-}(\wh{f'_\nat}) \cdot \Phi(0).
\end{align*}
By Hypothesis $(\star)$ we have
$$
\Phi(0)=1.
$$
The theorem now follows.
\end{proof}

\subsection{Completion of the proof of Theorem \ref{thm local}: Case (2)-(ii) and (2)-(iii).} It remains to prove case (2)-(ii) and (2)-(iii), i.e.,  when $v$ is non-split.
If we choose a suitable admissible function $f'$ and a smooth transfer $f$, then the equality holds for $f,f'$ by Theorem \ref{thm germ gl}, Prop. \ref{prop hypo} and Theorem \ref{thm germ u}.

\subsection{Concluding remarks}

Note that we only deal with $\pi_v$ which appears as a local component of  a global $\pi$. But we expect Conjecture \ref{conj local} to hold in general (as long as $\Pi_v$ is generic in order to define $I_{\Pi_v}$).

We conclude with
\begin{conj}
The spherical character $ I_\Pi$ and $J_\pi$ are  representable by a locally $L^1$ function which is smooth (locally constant in the non-archimedean case) on an open subset.\end{conj}
There should be a more complete analogue of the Harish--Chandra local character expansion in our relative setting.
Moreover, it seems that the spherical character (if non-zero) should contain as much information as the usual character of the representation.

\section*{Acknowledgement}  The influence of Herv\' e Jacquet on this paper should be obvious to the reader. The author is grateful to S. Zhang and B. Gross for their constant support, to Z. Mao for a question that leads to an improvement of the proof, to M. Harris, A. Ichino, D. Jiang, E. Lapid, C-P. Mok, Y. Tian, L. Xiao, X. Yuan for helpful comments, to the anonymous referees for their careful reading of this paper and making useful suggestions.
 The author thanks  the Morningside Center of Mathematics at the Chinese Academy of Sciences for their hospitality.


\begin{thebibliography}{99}

\bibitem{AGRS} {Aizenbud, Gourevitch, Rallis, Schiffmann, \textit{Multiplicity one theorems}, Ann. of Math. (2) 172 (2010), no. 2, 1407--1434. }


\bibitem{BP}{R. 
Beuzart-Plessis, \textit{
La conjecture locale de Gross-Prasad pour les reprŽsentations
tempŽrŽes des groupes unitaires}, arXiv:1205.2987}

\bibitem{CHH}{
Cowling, M.; Haagerup, U.; Howe, R \textit{
Almost L2 matrix coefficients.} 
J. Reine Angew. Math. 387 (1988), 97--110. }


\bibitem{CJ}{
Jacquet, Herv\'e; Chen, Nan.
\textit{Positivity of quadratic base change L-functions. }
Bull. Soc. Math. France 129 (2001), no. 1, 33--90. }


\bibitem{FLO}{
Feigon, Brooke; Lapid, Erez; Offen, Omer. \textit{On representations distinguished by unitary groups.} Publ. Math. IHES. 115 (2012), 185--323.}

\bibitem{F0}{
Flicker, Yuval Z.\textit{ Twisted tensors and Euler products. }
Bull. Soc. Math. France 116 (1988), no. 3, 295--313.}




\bibitem{GGP}{Wee Teck Gan, Benedict H. Gross and Dipendra
Prasad. \textit{Symplectic local root numbers, central critical
L-values, and restriction problems in the representation theory of
classical groups.} Asterisque 346, 1-110.}


\bibitem{GI}{Wee Teck Gan, Atsushi Ichino. \textit{On endoscopy and the refined Gross-Prasad conjecture for $(\SO_5,,\SO_4)$}. to appear in J. Math Institute Jussieu.}

\bibitem{Gar}{
P. Garrett,  {\em  Decomposition of Eisenstein series: Rankin
triple products. } Ann. of Math. (2)  125  (1987),  no. 2, 209--235.}


\bibitem{GJaR}{
Gelbart, Stephen; Jacquet, Herv¨¦; Rogawski, Jonathan.
\textit{Generic representations for the unitary group in three
variables.}  Israel J. Math.  126  (2001), 173--237.}

\bibitem{GJR1}{Ginzburg, David; Jiang, Dihua; Rallis, Stephen \textit{ On the nonvanishing
of the central value of the Rankin-Selberg L-functions.} J. Amer.
Math. Soc. 17 (2004), no. 3, 679--722}



\bibitem{GJR2}{Ginzburg, David; Jiang, Dihua; Rallis, Stephen. \textit{Models for certain residual representations of unitary groups.} Automorphic forms and L-functions I. Global aspects, 125--146, 
Contemp. Math., 488, Amer. Math. Soc., Providence, RI, 2009. }

\bibitem{Gra}{
Grafakos, Loukas.  \textit{Classical Fourier analysis.} Second edition. Graduate Texts in Mathematics, 249. Springer, New York, 2008. xvi+489 pp.  }


\bibitem{G87}{
Gross, Benedict H. \textit{Heights and the special values of L-series.}  Number theory (Montreal, Que., 1985), 115Ð187, CMS Conf. Proc., 7, Amer. Math. Soc., Providence, RI, 1987.}

\bibitem{G}{
Gross, Benedict H.\textit{ On the motive of a reductive group.}
Invent. Math.  130  (1997),  no. 2, 287--313.}


\bibitem{GZ}{B. Gross; D. Zagier: \textit{Heegner points
and derivatives of $L$-series.} Invent. Math. 84 (1986), no. 2,
225--320. }

\bibitem{HC1}{Harish-Chandra, \textit{Harmonic analysis on reductive $p$-adic
groups.} Notes by G. van Dijk. Lecture Notes in Mathematics, Vol.
162. Springer-Verlag, Berlin-New York, 1970 }


\bibitem{HC2}{
Harish-Chandra,\textit{ Admissible invariant distributions on
reductive $p$-adic groups.} Lie theories and their applications
(Proc. Ann. Sem. Canad. Math. Congr., Queen's Univ., Kingston,
Ont., 1977), pp. 281--347. Queen's Papers in Pure Appl. Math., No.
48, Queen's Univ., Kingston, Ont., 1978 }


\bibitem{HK}{
M. Harris and S.  Kudla, {\em On a conjecture of
Jacquet.} Contributions to automorphic forms, geometry, and number
theory, 355--371, Johns Hopkins Univ. Press, Baltimore, MD, 2004.}


\bibitem{HM}{
M. Harris. \textit{
L-functions and periods of adjoint motives}, to appear in Algebra and Number Theory.}


\bibitem{H}{
R. Neal Harris,  \textit{A refined Gross--Prasad conjecture for unitary groups}, arXiv:1201.0518.}


\bibitem{I}{
 Ichino Atsushi. \textit{ Trilinear forms and the central values of triple product
L-functions.}  Duke Math. J. Volume 145, Number 2 (2008), 281--307.}

\bibitem{II}{
A. Ichino and T. Ikeda. \textit{ On the periods of automorphic
forms on special orthogonal groups and the Gross-Prasad
conjecture,} Geom. Funct. Anal. 19 (2010), no. 5, 1378--1425.}


\bibitem{IZ}{
A. Ichino and W. Zhang. \textit{ Spherical characters for a strongly tempered pair}. Appendix to \cite{Zh2}.}

\bibitem{J}{ Jacquet, Herv\'e.\textit{ Archimedean Rankin-Selberg integrals.} Automorphic forms and L-functions II. Local aspects, 57Ð172, Contemp. Math., 489, Amer. Math. Soc., Providence, RI, 2009.}

 \bibitem{J2}{ Jacquet, Herv\'e. \textit{Distinction by the quasi-split unitary group. }Israel J. Math. 178 (2010), 269--324.}
 

\bibitem{JPSS}{ Jacquet, H.; Piatetski-Shapiro, I. I.; Shalika, J. A. \textit{Rankin-
Selberg convolutions. } Amer. J. Math.  105  (1983),  no. 2,
367--464.}


\bibitem{JR}{ Herv\'e Jacquet, S. Rallis.\textit{ On the Gross-Prasad conjecture for unitary
groups}. in {\em On certain L-functions}, 205Ð264, Clay Math. Proc., 13, Amer. Math. Soc., Providence, RI, 2011.}

\bibitem{JS}{Jacquet, H.; Shalika, J. A.  \textit{On Euler products and the
classification of automorphic representations. I.}  Amer. J. Math.
103  (1981), no. 3, 499--558.}


\bibitem{JS2}{Jacquet, H.; Shalika, J. A.   \textit{On Euler products and the classification of automorphic forms. II}. Amer. J. Math. 103 (1981), no. 4, 777--815.}  
 
\bibitem{L1}{E. Lapid, S. Rallis. \textit{On the non-negativity of $L(1/2,\pi)$ for $SO(2n+1)$,} Annals of Math. 157 (2003), 891-917.}



\bibitem{L2}{E. Lapid. \textit{
On the non-negativity of Rankin-Selberg $L$-functions at the center of symmetry}, International Mathematics Research Notices 2003 No.2 65-75.}


\bibitem{LM}{E. Lapid, Z. Mao. \textit{On a new functional equation for local integrals}. preprint.}


\bibitem{Mok}{
C-P. Mok, \textit{Endoscopic classification of representations of quasi-split unitary groups},	arXiv:1206.0882v1}




\bibitem{RR}{Rader, Cary; Rallis, Steve. \textit{Spherical characters on $p$-adic symmetric spaces}. Amer. J.
Math. 118 (1996), no. 1, 91--178. }


\bibitem{SV}{Y. Sakellaridis, A. Venkatesh. \textit{
Periods and harmonic analysis on spherical varieties}, arXiv:1203.0039.}


\bibitem{Sun}{
Sun, Binyong \textit{Bounding matrix coefficients for smooth vectors of tempered representations.} Proc. Amer. Math. Soc. 137 (2009), no. 1, 353--357.}


\bibitem{SZ}{Sun B., Zhu C. \textit{Multiplicity one theorems: the Archimedean case}, Annals Math., to appear.} 

\bibitem{T}{Ye Tian, \textit{
Congruent numbers and Heegner points}, preprint}


\bibitem{W1}{J. Waldspurger: \textit{Sur les valeurs de certaines fonctions
L automorphes en leur centre de sym\'etrie.}  Compositio Math. 54
(1985), no. 2, 173--242.}

\bibitem{W3}{J. Waldspurger: \textit{Une formule int¨¦grale reli¨¦e ¨¤ la conjecture locale de Gross-Prasad, 2¨¨me partie: extension aux repr¨¦sentations temp¨¦r¨¦es .} arXiv:0904.0314.}

\bibitem{Wh}{
P.-J., White. \textit{Tempered automorphic representations of the unitary group}, arXiv: 1106.1127 }


\bibitem{Wat}{T. Watson: \textit{Rankin triple products and quantum chaos}, to appear in Ann. of Math. (2), Ph.D. dissertation, Princeton University, Princeton, 2002.}

\bibitem{YZZ}{Xinyi Yuan, Shou-Wu Zhang, Wei Zhang. \textit{The Gross--Zagier formula on Shimura curves.} Ann. of Math. Studies \#184, Princeton University Press, 2012.}

\bibitem{Y}{Zhiwei Yun: \textit{The fundamental lemma of Jacquet and Rallis.} With an appendix by Julia Gordon. Duke
Math. J. 156 (2011), no. 2, 167--227. }

\bibitem{Z01}{Zhang, Shouwu. \textit{Heights of Heegner points on Shimura curves.} Ann. of Math. (2) 153 (2001), no. 1, 27--147}

\bibitem{Z01b}{
S. Zhang \textit{Gross-Zagier formula for GL(2).} Asian J. Math. 5 (2001), no. 2, 183--290. }

\bibitem{Z04}{S. Zhang.\textit{
Gross-Zagier formula for GL(2). II.} Heegner points and Rankin L-series, 191Ð214, Math. Sci. Res. Inst. Publ., 49, Cambridge Univ. Press, Cambridge, 2004.}

\bibitem{Z10}{S. Zhang. \textit{Linear forms,  algebraic cycles, and derivatives of L-series}, preprint.}


\bibitem{Zh1}{W. Zhang. \textit{On arithmetic fundamental lemmas}, Invent. Math., Volume 188, Number 1 (2012), 197--252. }


\bibitem{Zh2}{W. Zhang. \textit{Fourier transform and the global Gan--Gross--Prasad conjecture for unitary groups}, Preprint, 2011. 
}


\end{thebibliography}
\end{document}